\documentclass{daj}

\usepackage{amsmath, amsthm, amsopn, amssymb}
\usepackage{bbm}
\usepackage{enumerate}
\usepackage{enumitem}
\usepackage{stmaryrd} 

\dajAUTHORdetails{
  title = {An Efficient Container Lemma},
  author = {J\'ozsef Balogh and Wojciech Samotij},
  plaintextauthor = {Jozsef Balogh, Wojciech Samotij}
}

\dajEDITORdetails{%
   year={2020},
   number={17},
   received={23 March 2020},   
   published={10 December 2020},  
   doi={10.19086/da.17354},       
}   

\newcommand{\cA}{\mathcal{A}}
\newcommand{\gcA}{\hat{\mathcal{A}}}

\newcommand{\cB}{\mathcal{B}}
\newcommand{\cC}{\mathcal{C}}
\newcommand{\FF}{\mathcal{F}}
\newcommand{\GG}{\mathcal{G}}
\newcommand{\GGn}{\mathcal{G}_*}
\newcommand{\GGnv}{\mathcal{G}_*^v}
\newcommand{\HH}{\mathcal{H}}
\newcommand{\KK}{\mathcal{K}}
\newcommand{\cL}{\mathcal{L}}
\newcommand{\II}{\mathcal{I}}
\newcommand{\NN}{\mathbb{N}}
\newcommand{\RR}{\mathbb{R}}
\newcommand{\cP}{\mathcal{P}}
\newcommand{\cR}{\mathcal{R}}
\newcommand{\cS}{\mathcal{S}}
\newcommand{\TT}{\mathcal{T}}
\newcommand{\Ex}{\mathbb{E}}
\newcommand{\eps}{\varepsilon}
\newcommand{\ZZ}{\mathbb{Z}}

\newcommand{\Rind}{R_{\mathrm{ind}}}
\newcommand{\indarrow}{\rightarrow_{\mathrm{ind}}}

\newcommand{\Fclose}{\mathcal{F}_{\mathrm{close}}}
\newcommand{\Ffar}{\mathcal{F}_{\mathrm{far}}}

\newcommand{\hmax}{h_{\max}}
\newcommand{\height}{\mathrm{height}}

\newcommand{\Bin}{\mathrm{Bin}}
\newcommand{\Col}{\mathrm{Col}}

\newcommand{\Del}{\hat\Delta}
\newcommand{\one}{\mathbf{1}}

\newcommand{\ex}{\mathrm{ex}}

\newcommand{\alphan}{\alpha^*}

\newcommand{\Hyp}{\mathrm{Hyp}}

\newcommand{\Part}{\Pi}

\newtheorem{thm}{Theorem}
\newtheorem{lemma}[thm]{Lemma}
\newtheorem{prop}[thm]{Proposition}

\newtheorem{fact}[thm]{Fact}

\newtheorem{claim}[thm]{Claim}
\newtheorem{obs}[thm]{Observation}

\theoremstyle{definition}
\newtheorem*{remark}{Remark}

\numberwithin{thm}{section}

\newcommand{\br}[1]{\llbracket{#1}\rrbracket}

\newcommand{\lan}{\left\langle}
\newcommand{\ran}{\right\rangle}

\renewcommand{\le}{\leqslant}
\renewcommand{\ge}{\geqslant}

\begin{document}

\begin{frontmatter}[classification=text]
  
  \title{An Efficient Container Lemma}

  \author[jozsi]{J\'ozsef Balogh\thanks{Supported in part by NSF Grant DMS-1764123, Arnold O.\ Beckman Research Award (UIUC Campus Research Board RB 18132), and the Langan Scholar Fund from UIUC.}}
  \author[wojtek]{Wojciech Samotij\thanks{Supported by the Israel Science Foundation grants 1147/14 and 1145/18 (WS).}}

  \begin{abstract}
    We prove a new, efficient version of the hypergraph container theorems that is suited for hypergraphs with large uniformities. The main novelty is a refined approach to constructing containers that employs simple ideas from high-dimensional convex geometry. The existence of smaller families of containers for independent sets in such hypergraphs, which is guaranteed by the new theorem, allows us to improve upon the best currently known bounds for several problems in extremal graph theory, discrete geometry, and Ramsey theory. One of the applications of our efficient container lemma, a structural characterisation of $n$-vertex graphs with clique number $o(\log n / \log \log n)$, suggests that the new lemma is nearly best-possible for general hypergraphs of large uniformities.
  \end{abstract}
\end{frontmatter}

\section{Introduction}

The hypergraph container theorems, proved several years ago by Balogh, Morris, and Samotij~\cite{BalMorSam15} and, independently, by Saxton and Thomason~\cite{SaxTho15}, state that the family of independent sets of any uniform hypergraph whose edges are distributed somewhat evenly may be covered by (the families of subsets of) a small collection of sets, called containers, each of which is nearly independent. The original motivation for these results were several specific questions concerning enumeration of graphs avoiding a given subgraph and of sets of integers contining no arithmetic progressions of a given length. (The idea of considering these problems in the general context of independent sets in hypergraphs had been successfully pursued earlier in the breakthrough works of Conlon--Gowers~\cite{ConGow16} and Schacht~\cite{Sch16} on extremal properties of random graphs and random sets of integers.) However, over the years, the scope of applicability of the container theorems has grown quite substantially (see, for example, the survey~\cite{BalMorSam18} and references therein). The two major reasons for this are the general form of the theorems (many problems can be cast in the language of independent sets in auxiliary hypergraphs) and the explicit, optimal dependence between the various parameters disguised under the vague phrases `small collection', `evenly distributed', and `nearly independent' above.

The vast majority of applications of the container theorems concern sequences of hypergraphs of fixed uniformity and growing order and size. As a result, the explicit dependence of the various parameters involved in the statements of the container theorems on the uniformity of the hypergraph is merely a minor detail in all these works. Recently, however, the container theorems have been used to analyse sequences of hypergraphs whose uniformities grow with the numbers of vertices and edges. In these applications of the container method to questions in Ramsey theory~\cite{ConDelLaFRodSch17, RodRucSch17}, discrete geometry~\cite{BalSol18}, and extremal graph theory~\cite{BalBusColLiuMorSha17, MouNenSte14}, the explicit dependence between uniformity and other parameters turned out to lie at the heart of the matter, obstructing the way to obtaining optimal bounds for several well studied functions. It is only fair to note here that this dependence is more favourable in the version of the theorem proved by Saxton and Thomason~\cite{SaxTho15}. Having said that, the two constructions of containers presented in~\cite{BalMorSam15} and~\cite{SaxTho15} are essentially equivalent and the differences between the final results reflect merely the differences in their analyses. This analysis was performed more carefully, and with a wiser choice of parameters, by the authors of~\cite{SaxTho15}.

The basic container lemma, which is the building block of both proofs that really lies at the heart of the matter, is a statement that asserts the existence of a small family $\cC$ of containers for independent sets of an $s$-uniform hypergraph $\HH$ that satisfies $|C| \le (1-\delta)v(\HH)$ for every $C \in \cC$ and some positive constant $\delta$; see~\cite[Proposition~3.1]{BalMorSam15} and \cite[Theorem~3.4]{SaxTho15}. The stronger form of the theorem described in the first paragraph is then derived by recursively applying this basic lemma to the subhypergraphs of $\HH$ induced by the sets $C \in \cC$ as long as $C$ still contains many edges of $\HH$. The caveat here is that the proof methods used in both~\cite{BalMorSam15} and~\cite{SaxTho15} necessarily yield $\delta \le 1/s!$. (The short, non-algorithmic proof of the basic container lemma given recently by Bernshteyn, Delcourt, Towsner, and Tserunyan~\cite{BerDelTowTse19} seems to yield $\delta$ that is doubly-exponentially small in $s$.)   Since one typically requires the ratio $|C|/v(\HH)$ to be bounded away from one for each final container~$C$, at least $\exp\big(\Omega(s \log s)\big)$ iterations are required; this substantially blows up the final number of containers when $s$ is no longer a fixed constant. Finally, we remark that a different method of building containers for independent sets in hypergraphs was proposed and analysed by Saxton and Thomason~\cite{SaxTho16}. Even though the parameter $\delta$ in the basic container lemma proved in~\cite{SaxTho16} is only polynomially small in the uniformity, the upper bound on the number of containers is far from optimal. Moreover, the lemma applies only to simple hypergraphs (i.e., hypergraphs whose every pair of vertices is contained in at most one edge) whereas the hypergraphs considered in most applications of the container method are far from being simple.

The main result of this work is a new, more efficient version of the basic container lemma in which the parameter $\delta$ is only polynomially small in the uniformity. We postpone stating the strongest form of our new lemma until Section~\ref{sec:main-techn-result} and state here only its corollary that can be easily compared with~\cite[Proposition~3.1]{BalMorSam15}. Following the notational convention of~\cite{BalMorSam15}, given a nonempty $s$-uniform hypergraph $\HH$, we shall denote the numbers of its vertices and edges by $v(\HH)$ and $e(\HH)$, respectively. Moreover, for every $T \subseteq V(\HH)$, we define
\[
  \deg_{\HH} T = |\{A \in E(\HH) : T \subseteq A\}|
\]
and, for every $t \in \{1, \dotsc, s\}$, we let
\[
  \Delta_t(\HH) = \max\big\{\deg_\HH T : T \subseteq V(\HH) \text{ and } |T| = t\big\}.
\]
The following theorem, an efficient basic container lemma, is a simplified version of our main technical result, Theorem~\ref{thm:main} below.

\begin{thm}
  \label{thm:main-simple}
  Let $s$ be a positive integer and let $\HH$ be a nonempty~$s$-uniform hypergraph.
  Suppose that $q \in (0,1)$ and $K > 0$ are such that $q \cdot v(\HH) \ge 10^8 s^6 K$ and, for every $t \in \{1, \dotsc, s\}$,
  \begin{equation}
    \label{eq:Delta-main-simple}
    \Delta_t(\HH) \le K \cdot \left(\frac{q}{10^6 s^5}\right)^{t-1} \cdot \frac{e(\HH)}{v(\HH)}.
  \end{equation}
  Then, there exist a family $\cS \subseteq \binom{V(\HH)}{\le q \cdot v(\HH)}$ and functions $f \colon \cS \to \cP(V(\HH))$ and $g \colon \II(\HH) \to \cS$ such that, for every $I \in \II(\HH)$,
  \[
    g(I) \subseteq I \subseteq g(I) \cup f(g(I)) \qquad \text{and} \qquad |f(g(I))| \le (1-\delta) \cdot v(\HH),
  \]
  where $\delta = (10^3 s^4 K)^{-1}$. Moreover, if $g(I) \subseteq I'$ and $g(I') \subseteq I$ for some $I, I' \in \II(\HH)$, then $g(I) = g(I')$.
\end{thm}

Qualitatively, Theorem~\ref{thm:main-simple} is identical to the original basic container lemmas~\cite[Proposition~3.1]{BalMorSam15} and~\cite[Theorem~3.4]{SaxTho15}. Quantitively, however, it is a significant improvement of these results. In order to demonstrate this, we shall present four applications of our new theorem to problems in extremal graph theory, discrete geometry, and Ramsey theory that had previously been attacked using the original container theorems, obtaining an essential improvement of the state-of-the-art result in each case. We discuss these applications and the relevant background in detail in the next three sections. We just point out here that one of these applications, an efficient version of the classical theorem of Kolaitis, Pr\"omel, and Rothschild~\cite{KolProRot87}, Theorem~\ref{thm:clique-free} below, strongly suggests that Theorem~\ref{thm:main-simple} is, up to lower-order terms, optimal for general hypergraphs of large uniformities (see the discussion below the statement of Theorem~\ref{thm:clique-free}).

\subsection{The typical structure of graphs with no large cliques}
\label{sec:applications-cliques}

Given graphs $G$ and~$H$, we say that $G$ is \emph{$H$-free} if $G$ does not contain $H$ as a (not necessarily induced) subgraph. The study of typical properties of $H$-free graphs goes back to the seminal work of Erd\H{o}s, Kleitman, and Rothschild~\cite{ErdKleRot76}, who proved that almost all triangle-free graphs are bipartite.\footnote{That is, the probability that a uniformly random $K_3$-free subgraph of $K_n$ is bipartite tends to one as $n$ tends to infinity.} We shall consider here almost exclusively the case where $H$ is a clique; for a wealth of information regarding other graphs, we refer the reader to~\cite{BalBolSim09}. The result of~\cite{ErdKleRot76} was generalised by Kolaitis, Pr\"omel, and Rothschild~\cite{KolProRot87}, who showed that, for every fixed $r \ge 2$, almost all $K_{r+1}$-free graphs are $r$-partite.

Let us point out that this `structural' characterisation of a typical $K_{r+1}$-free graph is in fact an enumeration result in disguise. Since every $r$-partite graph is $K_{r+1}$-free, the main result of~\cite{KolProRot87} is that the number of $K_{r+1}$-free subgraphs of $K_n$ is asymptotically equal to the number of $r$-partite subgraphs of $K_n$. Taking this point of view, one can say that the result of~\cite{KolProRot87} was anticipated by the aforementioned work of Erd\H{o}s, Kleitman, and Rothschild~\cite{ErdKleRot76}, who also showed that, for every fixed $r$, there are $2^{\ex(n,K_{r+1})+o(n^2)}$ many $K_{r+1}$-free subgraphs of $K_n$.\footnote{We write $\ex(n,H)$ to denote the Tur\'an number of a graph $H$, that is, the largest number of edges in an $H$-free graph with $n$ vertices.} This estimate was generalised by Erd\H{o}s, Frankl, and R\"odl~\cite{ErdFraRod86}, who proved that, for every fixed graph $H$, there are $2^{\ex(n,H)+o(n^2)}$ many $H$-free subgraphs of~$K_n$.

In the above discussion, we have tried to stress that the forbidden graph $H$ is fixed whereas $n$, the number of vertices in the host graphs, tends to infinity. Much less is known if one allows the order/size of $H$ to grow with $n$. This more general question was considered only fairly recently by Bollob\'as and Nikiforov~\cite{BolNik10}; however, the results of~\cite{BolNik10} are only meaningful when the chromatic number of $H$ stays bounded. A few years later, Mousset, Nenadov, and Steger~\cite{MouNenSte14}, extending the Erd\H{o}s--Kleitman--Rothschild bound, proved that there are $2^{\ex(n, K_{r+1}) + o(n^2/r)}$ many $K_{r+1}$-free subgraphs of $K_n$ as long as $r \le (\log n)^{1/4}/2$.\footnote{The $o(n^2/r)$ error term is natural here as $\ex(n,K_{r+1}) = \binom{n}{2} - \Theta(n^2/r)$.} Somewhat later, Balogh, Bushaw, Collares, Liu, Morris, and Sharifzadeh~\cite{BalBusColLiuMorSha17} strengthened this result considerably by showing that, under the slightly weaker assumption $2 \le r \le (\log n)^{1/4}$, almost all $K_{r+1}$-free subgraphs of $K_n$ are $r$-partite. (Both~\cite{MouNenSte14} and~\cite{BalBusColLiuMorSha17} relied on the original hypergraph container theorems.) Our first application of the new, efficient container lemma is the following strengthening of this result.

\begin{thm}
  \label{thm:clique-free}
  If a function $r \colon \NN \to \NN$ satisfies $2 \le r(n) \le \log n / (121 \log \log n)$, then almost all $K_{r+1}$-free subgraphs of $K_n$ are $r$-partite.
\end{thm}

We point out that the assumption on the growth rate of $r$ in Theorem~\ref{thm:clique-free} is nearly optimal. Indeed, a standard first-moment calculation shows that, for every positive constant $\eps$, a uniformly random subgraph $G \subseteq K_n$ contains no clique with $\lfloor (2+\eps)\log_2n \rfloor$ vertices whereas $\chi(G) \ge \Omega(n/\log n)$. On the other hand, it may well be that the assertion of the theorem remains true as long as $r(n) \le (2-\eps)\log_2n$ for some positive constant $\eps$. However, even removing the doubly-logarithmic term from the denominator in the assumed upper bound on $r(n)$ will likely require significantly new ideas.

\subsection{Lower bounds on $\eps$-nets}
\label{sec:application-eps-nets}

Suppose that $X$ is a finite set and let $\cR$ be an arbitrary collection of subsets of $X$. For a positive number $\eps$, an \emph{$\eps$-net} in $\cR$ is any set $N \subseteq X$ that intersects every element of $\cR$ with cardinality at least $\eps |X|$. In other words, $N$ is an $\eps$-net if $N \cap A \neq \emptyset$ for every $A \in \cR$ with $|A| \ge \eps |X|$.

One is usually interested in finding a small $\eps$-net. However, this is not always possible. For example, if $\cR$ comprises all subsets of $X$, then every $\eps$-net in $\cR$ must have more than $(1-\eps)|X|$ elements. One can rule out such `pathological' examples by imposing a natural assumption on a measure of complexity of the family $\cR$ called the VC dimension. We say that a set $S$ is \emph{shattered} by a family $\cR$ if $\{A \cap S : A \in \cR\}$ contains all $2^{|S|}$ subsets of $S$. The \emph{VC dimension} (a shorthand for \emph{Vapnik--Chervonenkis dimension}) of $\cR$ is the largest cardinality of a set that $\cR$ shatters. A seminal result of Haussler and Welzl~\cite{HauWel87} states that every family of subsets whose VC dimension is at most $d$ admits an $\eps$-net with at most $\lceil (8d/\eps) \log (8d/\eps) \rceil$ elements, for every $\eps > 0$. Koml\'os, Pach, and Woeginger~\cite{KomPacWoe92} improved this upper bound on the smallest size of an $\eps$-net to $\big(d+o(1)\big) \cdot (1/\eps)\log(1/\eps)$, where $o(1)$ denotes some function tending to zero with $\eps$. Moreover, they constructed, for every $d \ge 2$, (random) families with VC dimension $d$ that have no $\eps$-net smaller than $\big(d-2+2/(d+1)-o(1)\big) \cdot (1/\eps) \log(1/\eps)$.

On the other hand, it was proved that various set families arising in geometry admit $\eps$-nets of cardinality merely $O(1/\eps)$, see~\cite{KomPacWoe92,MatSeiWel90}. In view of this, many researchers believed that in `geometric scenarios' (with bounded VC dimension), there always exists an $\eps$-net of size $O(1/\eps)$. This belief was shown to be wrong by Alon~\cite{Alo12}, who proved that, for arbitrary small $\eps$, there are finite sets $X$ of points in the plane such that every $\eps$-net for the family comprising the intersections of $X$ with straight lines (the \emph{range space} of lines on $X$) must have at least $(1/\eps) \cdot \omega(1/\eps)$ points, for some (very slowly growing) function~$\omega$ with $\lim_{x \to \infty} \omega(x) = \infty$. Alon speculated that there are planar sets of points~$X$ for which the factor $\omega(1/\eps)$ in the above statement could be replaced by $\Omega\big(\log(1/\eps)\big)$.

In a paper that served as the main motivation for this work, Balogh and Solymosi~\cite{BalSol18} showed that, for arbitrarily small $\eps > 0$, there are sets $X \subseteq \RR^2$ such that the range space of lines on $X$ does not have $\eps$-nets with fewer than $(1/\eps)\big(\log(1/\eps)\big)^{1/3-o(1)}$ points; their proof relied on the hypergraph container theorems. We review the construction of Balogh and Solymosi~\cite{BalSol18} and, using our new, efficient container lemma, we further improve their lower bound, replacing the constant $1/3$ in the exponent with~$1/2$.

\begin{thm}
  \label{thm:eps-nets}
  The following holds for every $\eps_0 > 0$. There exists an $\eps \in (0, \eps_0)$ and a~finite set $X \subseteq \RR^2$ such that the smallest size of an $\eps$-net for the family of intersections of straight lines with $X$ is at least
  \[
    \frac{1}{80\eps} \cdot \sqrt{\frac{\log(1/\eps)}{\log\log(1/\eps)}}.
  \]
\end{thm}

We should mention here that, several years prior to~\cite{BalSol18}, Pach and Tardos~\cite{PacTar13} showed that the families defined by intersections of finite point sets with axis parallel rectangles (in $\RR^2$) and axis-parallel boxes in $\RR^4$ may require $\eps$-nets of sizes $\Omega\big((1/\eps) \log\log(1/\eps)\big)$ and $\Omega\big((1/\eps)\log(1/\eps)\big)$, respectively; both these lower bounds are tight up to multiplicative constants, see~\cite{AroEzrSha10}.

\subsection{Upper bounds on Ramsey numbers}
\label{sec:application-Ramsey}

Given graphs $G$ and $H$ and a positive integer $k$, we write $G \rightarrow (H)_k$, and say that \emph{$G$ is Ramsey for $H$ in $k$ colours}, if every $k$-colouring of the edges of $G$ contains a monochromatic copy of $H$. In other words, $G \rightarrow (H)_k$ if, for every $c \colon E(G) \to \br{k}$, there is some $i \in \br{k}$ such that the graph $c^{-1}(i)$ contains $H$ as a subgraph.\footnote{Throughout the paper, we write $\br{k}$ as a shorthand for $\{1, \dotsc k\}$.} The famous theorem of Ramsey~\cite{Ram29} states that, for all  positive integers $n$ and $k$, there is an integer $N$ such that $K_N \rightarrow (K_n)_k$; we shall denote the smallest such $N$, the \emph{$k$-colour Ramsey number of $K_n$}, by $R(n;k)$. It is well-known that $R(n;k) \le (kn)!/(n!)^k \le k^{kn}$, see~\cite{ErdSze35,GreGle55}.

Fifty years ago, Folkman~\cite{Fol70} proved that, for every $n$, there exists a graph $G$ such that $G \nsupseteq K_{n+1}$ but, nevertheless, $G \rightarrow (K_n)_2$ and Ne\v{s}et\v{r}il and R\"odl~\cite{NesRod76} generalised this result to an arbitrary number of colours. Define the \emph{$k$-colour Folkman number for $K_n$} by
\[
  F(n;k) = \min\big\{N \in \NN : G \rightarrow (K_n)_k \text{ for some $K_{n+1}$-free $G \subseteq K_N$}\big\}.
\]
The constructions given in~\cite{Fol70,NesRod76} yielded upper bounds on $F(n;k)$ that are tower functions of height polynomial in $n$ and $k$. On the other hand, the strongest lower bound on $F(n;k)$, due to Lefmann~\cite{Lef87}, is only exponential in $kn$. In the recent years, the transference theorems of Conlon--Gowers~\cite{ConGow16} and Schacht~\cite{Sch16} (see also~\cite{FriRodSch10}) were used by R\"odl, Ruci\'nski, and Schacht~\cite{RodRucSch-weak} and by Conlon and Gowers (unpublished) to give improved upper bounds on $F(n;k)$ that were merely doubly-exponential in $n$ and $k$. Soon afterwards, the first of these two groups of authors~\cite{RodRucSch17} used the hypergraph container theorems to give the first exponential bound $F(n;k) \le \exp\big(O(n^4\log n + n^3k\log k)\big)$. Our next application of the efficient container lemma is the following improvement of this result.

\begin{thm}
  \label{thm:Folkman}
  There exists a constant $C$ such that, for all positive integers $n$ and $k$,
  \[
    F(n;k) \le \big(CknR(n;k)\big)^{21n^2} \le \exp\left(Ckn^3\log k\right).
  \]
\end{thm}

Another well studied variation of the classical Ramsey numbers are induced Ramsey numbers. Given graphs $G$ and $H$ and a positive integer $k$, we write $G \indarrow (H)_k$, and say that \emph{$G$ is induced-Ramsey for $H$ in $k$ colours}, if every $k$-colouring of the edges of $G$ contains a monochromatic induced copy of $H$. In other words, $G \indarrow (H)_k$ if, for every $c \colon E(G) \to \br{k}$, there are an $i \in \br{k}$ and an injection $\varphi \colon V(H) \to V(G)$ such that $\varphi\big(E(H)\big) \subseteq c^{-1}(i)$ and $\varphi\big(E(H)^c\big) \cap E(G) = \emptyset$. The existence of induced-Ramsey graphs for every $H$ and any number of colours $k$ was established, independently, by Deuber~\cite{Deu75}, by Erd\H{o}s, Hajnal, and P\'osa~\cite{ErdHajPos75}, and by R\"odl~\cite{Rod73}. We may thus define the \emph{$k$-colour Ramsey number of $H$} by
\[
  \Rind(H;k) = \min\big\{N \in \NN : G \indarrow (H)_k \text{ for some } G \subseteq K_N\big\}.
\]
The upper bounds on $\Rind(H;k)$ implied by the constructions of~\cite{Deu75,ErdHajPos75,Rod73} were enormous. In spite of that, Erd\H{o}s~\cite{Erd84} conjectured that, for every $n$-vertex graph $H$, the $2$-colour induced Ramsey number $\Rind(H;2)$ is only exponential in $n$. The best-known result to date was obtained by Conlon, Fox, and Sudakov~\cite{ConFoxSud12}, who proved that $\Rind(H;2) \le \exp\big(O(n\log n)\big)$ for every $n$-vertex graph $H$. However, the method of~\cite{ConFoxSud12} does not work when the number of colours is larger than two. The strongest general upper bound for $k$-colour induced Ramsey numbers in the case $k > 2$ that can be found in the literature is due to Fox and Sudakov~\cite{FoxSud09}, who showed that $\Rind(H;k) \le \exp(C_kn^3)$ for every $n$-vertex $H$, where $C_k$ depends only on $k$. However, Fox (private communication) informed us that the methods of~\cite{FoxSud09}, which were optimised for sparse graphs $H$, may be used to prove that $\Rind(H;k) \le \exp(Ckn^2\log k)$. Our final application of the efficient container lemma is a short derivation of this bound.

\begin{thm}
  \label{thm:induced-Ramsey}
  There exists a constant $C$ such that, for every positive integer $k$ and every $n$-vertex graph $H$,
  \[
    \Rind(H;k) \le \big(Cn^2kR(n;k)\big)^{7n} \le \exp\left(Ckn^2\log k\right).
  \]
\end{thm}

Finally, let us mention that Conlon, Dellamonica Jr., La Fleur, R\"odl, and Schacht~\cite{ConDelLaFRodSch17} used the original container theorems to prove strong bounds on the induced Ramsey numbers of uniform hypergraphs.

\subsection{Packaged statement}

Each of the four illustrations of Theorem~\ref{thm:main-simple} presented in this paper requires iterative/recursive applications of the theorem. In order to save ourselves (and the reader) from repeating similar, routine arguments and calculations several times, it will be convenient for us to work with the following `packaged' version of the theorem that is analogous to~\cite[Theorem~2.2]{BalMorSam15} and~\cite[Corollary~3.6]{SaxTho15}.

\begin{thm}
  \label{thm:main-packaged}
  Let $s$ be a positive integer and let $\HH$ be a nonempty~$s$-uniform hypergraph.
  Suppose that $\alpha, \beta, q \in (0,1)$ and $E \ge v(\HH)$ are such that $\alpha \beta q \cdot v(\HH) \ge 10^9 s^7$ and $10^4 s^5 q \le \beta$ and, for every $t \in \{2, \dotsc, s\}$,
  \[
    \Delta_t(\HH) \le \left(\frac{q}{10^6 s^5}\right)^{t-1} \cdot \frac{E}{v(\HH)}.
  \]
  Then there is a family $\cC \subseteq \cP(V(\HH))$ of at most $\exp\left(10^4s^5\beta^{-1} \log(e/\alpha) \cdot q \log(e/q) \cdot v(\HH) \right)$ sets such that:
  \begin{enumerate}[label=(\roman*)]
  \item
    For every $I \in \II(\HH)$, there is a $C \in \cC$ such that $I \subseteq C$.
  \item
    For every $C \in \cC$, either $|C| \le \alpha v(\HH)$ or there is a subset $W \subseteq C$ with $|W| \ge (1-\beta)|C|$ such that $e(\HH[W]) < E$.
  \end{enumerate}
\end{thm}

The derivation of Theorem~\ref{thm:main-packaged} from Theorem~\ref{thm:main-simple} is presented in Section~\ref{sec:simple-pack-vers}

\subsection{Organisation of the paper}
\label{sec:organisation-paper}

The remainder of this paper is organised as follows. In Section~\ref{sec:main-techn-result}, we introduce the crucial concept of degree measures, state our main technical result, Theorem~\ref{thm:main}, and derive from it Theorems~\ref{thm:main-simple} and~\ref{thm:main-packaged}. Section~\ref{sec:prop-degree-measure} is devoted to establishing key properties of degree measures; these properties are used in the subsequent Section~\ref{sec:proof}, which contains the proof of Theorem~\ref{thm:main}. Two probabilistic inequalities needed for the four applications of our new container lemma are stated in Section~\ref{sec:prob-ineq}. Finally, Section~\ref{sec:typical-structure-clique-free} is devoted to the proof of Theorem~\ref{thm:clique-free}, Section~\ref{sec:lower-bounds-eps-nets} gives the proof of Theorem~\ref{thm:eps-nets}, and Section~\ref{sec:upper-bounds-Ramsey} contains proofs of Theorems~\ref{thm:Folkman} and~\ref{thm:induced-Ramsey}.

\subsection{Acknowledgement}

First of all, we are indebted to Rob Morris, David Saxton, and Andrew Thomason for sharing their numerous insights about the container theorems that had a strong bearing on this work. The notion of degree measures, which is central to our approach here, as well as the important idea of allowing hypergraphs to have multiple edges were first introduced by Andrew Thomason and David Saxton~\cite{SaxTho15}. Additionally, we would like to thank Noga Alon for his comments and suggestions regarding lower bounds on $\eps$-nets. We are also indebted to the anonymous referee for their extremely careful reading of the earlier version of this paper and their helpful comments and suggestions, which saved us from having several embarrassing mistakes in the final version of this work.

The second named author thanks Jacob Fox, Frank Mousset, and Bhargav Narayanan for inspiring discussions about upper-bounding induced Ramsey numbers. Last but not least, the second named author owes his deepest gratitude to Lev Buhovski for an inspiring discussion about high-dimensional convex geometry that laid foundations for Lemma~\ref{lemma:geometry}, which lies at the very heart of the proof of Theorem~\ref{thm:main}.

\section{The main technical result}
\label{sec:main-techn-result}

\subsection{A word of motivation}

The key idea behind the proof of the container lemma due to Morris and the authors~\cite{BalMorSam15} is to, given an $(r+1)$-uniform hypergraph $\HH$ and an independent set $I \in \II(\HH)$, consider a sequence of vertices of $\HH$ for inclusion in a small `signature' set $S$ and construct an $r$-uniform hypergraph $\GG$ from the neighbourhoods (link hypergraphs) of those among the considered vertices that belong to $I$. Crucially, each element of this sequence is allowed to depend only on the intersection of $I$ with the set of its predecessors; this guarantees that $\GG$ depends solely on $S$. Since $\GG$ comprises only neighbourhoods of vertices in $I$, we have $I \in \II(\GG)$. This facilitates induction on the uniformity of the hypergraph.

Whereas there is essentially one way to define containers for independent sets in a $1$-uniform hypergraph, the general description of the inductive step given above leaves plenty of room for manoeuvre. The approach taken in~\cite{BalMorSam15} was, roughly speaking, to cap the degrees of all vertices of $\GG$ at some predefined value $\Delta$ and, at the same time, make sure that $e(\GG) \ge \beta \Delta v(\GG)$ for some constant $\beta$; this way, the ratio of the maximum and the average degrees of the constructed hypergraph $\GG$ remained bounded by a constant. The advantage of this approach was its relative simplicity. However, this simplicity came at a price; the gap between the maximum and the average degrees was forced to grow by a factor of at least $r+1$, at each step of the induction (reducing uniformity from $r+1$ to $r$).\footnote{For those readers who are somewhat familiar with the proof in~\cite{BalMorSam15}, the essence of the above shortcoming was the following. Only one vertex of degree $\Delta$ in $\GG$ forced us to remove an edge from $\HH$, but while counting the edges of $\GG$ that contain some vertex with degree $\Delta$, we accounted for the possibility that every edge of $\GG$ contains $r$ vertices of degree $\Delta$.} As a result, the crucial parameter $\delta$ in the basic container lemma could not exceed $1/s!$, where $s$ is the uniformity of the original hypergraph.

Here, we use a similar high-level inductive strategy. However, we take a refined approach to choosing a sequence of vertices of the $(r+1)$-uniform $\HH$ while constructing the $r$-uniform $\GG$; this yields a much more favourable dependence of the parameter $\delta$ on the uniformity $s$. The key new idea is to abandon the wish to control the maximum degree of $\GG$ and instead focus on the $\ell^2$-norm of its degree sequence. In other words, we measure hypergraphs with $\ell^2$-norms, rather than $\ell^\infty$-norms, of their degree sequences. Viewing hypergraphs as vectors in high-dimensional Euclidean spaces allows us to reduce the problem of constructing a sequence of vertices to be considered for inclusion in the `signature' to an elementary problem in convex geometry.

\subsection{Degree measures}

We begin by extending the notion of the \emph{degree measure} of a hypergraph, which was introduced by Saxton and Thomason~\cite{SaxTho15}. For a non-empty $r$-uniform hypergraph $\HH$ with vertex set $V$ and a $t \in \br{r}$, we define the \emph{$t$-degree measure} of $\HH$, denoted by $\sigma_\HH^{(t)}$, to be the probability distribution on~$\binom{V}{t}$, the family of all $t$-element subsets of $V$, given by
\[
  \sigma_\HH^{(t)}(T) = \deg_\HH T \cdot \left( \sum_{U \in \binom{V}{t}} \deg_\HH U \right)^{-1} = \frac{\deg_\HH T}{\binom{r}{t} \cdot e(\HH)}.
\]
In other words, $\sigma_\HH^{(t)}$ is the probability distribution induced by the following random experiment. Select an edge $A$ of $\HH$ uniformly at random and output a $t$-element subset $T \subseteq A$ chosen uniformly at random from $\binom{A}{t}$.

Throughout this paper, we shall identify (as we already did in the above definition) the measure $\sigma_\HH^{(t)}$ with its density (with respect to the counting measure), which we shall view as an element of the $\binom{|V|}{t}$-dimensional vector space of $\RR$-valued functions on $\binom{V}{t}$. Since the $1$-degree measure will be of particularly high importance, we shall refer to it simply as the degree measure and often suppress the superscript $^{(1)}$ from the notation, denoting it by $\sigma_\HH$. Given a positive integer $d$ and a vector $\xi =(\xi_1, \ldots, \xi_d) \in \RR^d$, we denote by $\|\xi\|$ its $\ell^2$-norm, so that
\[
  \|\xi\|^2 = \sum_{i=1}^d \xi_i^2.
\]

\subsection{The main technical result}
We are now ready to state the main technical result of this paper, Theorem~\ref{thm:main} below. We postpone the proof of the theorem to Section~\ref{sec:proof}; the proof will use several simple properties of degree measures that will be derived in Section~\ref{sec:prop-degree-measure}.

\begin{thm}
  \label{thm:main}
  Let $s \in \NN$ and suppose that a nonempty $s$-uniform hypergraph $\HH$ and reals $p, \delta \in (0,1)$ satisfy 
  \begin{equation}
    \label{eq:assumption-main}
    300s^4 \cdot \sum_{t=1}^s \binom{s-1}{t-1} \left(\frac{5000s^3}{p}\right)^{t-1} \| \sigma_\HH^{(t)} \|^2 \le \frac{1}{\delta \cdot v(\HH)} \le \frac{p}{500}.
  \end{equation}
  Then, there exist a family $\cS \subseteq \binom{V(\HH)}{\le 30s^2p \cdot v(\HH)}$ and functions $f \colon \cS \to \cP(V(\HH))$ and $g \colon \II(\HH) \to \cS$ such that, for every $I \in \II(\HH)$,
  \[
    g(I) \subseteq I \subseteq g(I) \cup f(g(I)) \qquad \text{and} \qquad |f(g(I))| \le (1-\delta) \cdot v(\HH).
  \]
  Moreover, if $g(I) \subseteq I'$ and $g(I') \subseteq I$ for some $I, I' \in \II(\HH)$, then $g(I) = g(I')$.
\end{thm}

\subsection{The simple and packaged versions}
\label{sec:simple-pack-vers}

In this section, we derive Theorems~\ref{thm:main-simple} and~\ref{thm:main-packaged} from our main technical result, Theorem~\ref{thm:main} above. We start with a short proof of Theorem~\ref{thm:main-simple}.

\begin{proof}[{Derivation of Theorem~\ref{thm:main-simple} from Theorem~\ref{thm:main}}]
  Let $p = q/(30s^2)$ and let $\delta = (10^3 s^4 K)^{-1}$. It suffices to verify that $\HH$, $p$, and $\delta$ satisfy the assumptions of Theorem~\ref{thm:main}, which will give us the claimed family $\cS$ and functions $g$ and $f$. To this end, note that
  \[
  \| \sigma_\HH^{(t)} \|^2 = \sum_{T \in \binom{V(\HH)}{t}} \left(\frac{\deg_\HH T}{\binom{s}{t} \cdot e(\HH)}\right)^2 \le \frac{\Delta_t(\HH)}{\binom{s}{t} \cdot e(\HH)} \cdot \sum_{T \in \binom{V(\HH)}{t}} \frac{\deg_\HH T}{\binom{s}{t} \cdot e(\HH)} = \frac{\Delta_t(\HH)}{\binom{s}{t} \cdot e(\HH)} 
  \]
  for every $t \in \br{s}$ and thus the assumptions of the theorem imply that
  \[
    \begin{split}
      \sum_{t=1}^s \binom{s-1}{t-1} \left(\frac{5000s^3}{p}\right)^{t-1} \| \sigma_\HH^{(t)} \|^2 & \le \sum_{t=1}^s \binom{s-1}{t-1} \left(\frac{150000s^5}{q}\right)^{t-1} \cdot \frac{\Delta_t(\HH)}{\binom{s}{t} \cdot e(\HH)} \\
      & \le \frac{K}{v(\HH)} \cdot \sum_{t=1}^s \frac{t}{s} \cdot \left(\frac{150000 s^5}{10^6 \cdot s^5}\right)^{t-1} \le \frac{2K}{v(\HH)}.
    \end{split}
  \]
  Moreover, $300s^4 \cdot 2K \le 1/\delta$ and $p \cdot \delta \cdot v(\HH) = q \cdot (30s^2 \cdot 10^3 s^4 K)^{-1} \cdot v(\HH) \ge 500$.
\end{proof}

We now turn to the proof of Theorem~\ref{thm:main-packaged}. The key ingredient here is the following lemma, which, roughly speaking, states that a hypergraph that is `robustly dense' contains a large subhypergraph whose maximum degree is not much larger than its average degree. The statement and the proof of the lemma are inspired by the work of Morris and Saxton~\cite{MorSax16}.

\begin{lemma}
  \label{lemma:Delta-1-supersaturation}
  Let $\HH$ be an $s$-uniform hypergraph and suppose that, for some positive $\beta$ and $M$, every set $W \subseteq V(\HH)$ with $|W| \ge (1-\beta)v(\HH)$ satisfies $e(\HH[W]) \ge M$. Then, there is a subhypergraph $\HH' \subseteq \HH$ with at least $M$ edges that satisfies
  \[
    \Delta_1(\HH') \le \left\lceil \frac{s}{\beta} \cdot \frac{e(\HH')}{v(\HH)} \right\rceil.
  \]
\end{lemma}
\begin{proof}
  Let $\HH'$ be a largest (in terms of the number of edges) subhypergraph of $\HH$ satisfying $\Delta_1(\HH') \le \lceil \frac{sM}{\beta v(\HH)} \rceil$ and let $X \subseteq V(\HH)$ be the set of vertices of $\HH'$ whose degree achieves the bound $\lceil \frac{sM}{\beta v(\HH)} \rceil$. Observe that every edge of $\HH$ that is disjoint from $X$ must belong to $\HH'$ and, consequently, $e(\HH') \ge e(\HH-X)$. If $|X| \le \beta v(\HH)$, then $e(\HH') \ge M$ by our assumption on $\HH$. Otherwise, if $|X| > \beta v(\HH)$,
  \[
    e(\HH') = \frac{1}{s} \sum_{v \in V(\HH)} \deg_{\HH'}v \ge \frac{|X|}{s} \cdot \left\lceil \frac{sM}{\beta v(\HH)} \right\rceil > M.
  \]
  This completes the proof of the lemma.
\end{proof}

\begin{proof}[{Proof of Theorem~\ref{thm:main-packaged}}]
  We shall say that a set $C \subseteq V(\HH)$ is a \emph{good container} if either $|C| \le \alpha v(\HH)$ or if there is a subset $W \subseteq C$ with $|W| \ge (1-\beta)|C|$ such that $e(\HH[W]) < E$. We will construct a rooted tree $\TT$ whose vertices are subsets of $V(\HH)$ that has the following properties:
  \begin{enumerate}[label=(\roman*)]
  \item
    \label{item:root}
    The root of $\TT$ is $V(\HH)$.
  \item
    If an independent set $I \in \II(\HH)$ is contained in a non-leaf vertex of $\TT$, then $I$ is contained in some child of this vertex in $\TT$.
  \item
    \label{item:good-container}
    Every leaf of $\TT$ is a good container.
  \item
    \label{item:max-deg-bound}
    Every non-leaf vertex of $\TT$ has at most $(e/q)^{q \cdot v(\HH)}$ children.
  \item
    \label{item:height-bound}
    The height of $\TT$ is at most $10^4s^5\beta^{-1} \cdot \log(e/\alpha)$.
  \end{enumerate}
  The set of leaves of $\TT$ will then form a collection of containers for the independent sets of $\HH$ that has the desired properties.

  We build such a tree starting from the root $V(\HH)$ by iteratively applying Theorem~\ref{thm:main-simple} to (a carefully chosen subhypergraph of) the subhypergraph of $\HH$ induced by a leaf $C$ of $\TT$ that is not yet a good container and attaching the resulting family of containers (for the independent sets of $\HH[C]$, that is, the independent sets of $\HH$ that are contained in $C$) as children of $C$ (which, as a result, ceases to be a leaf of $\TT$) until no such leaves are left. This way, properties~\ref{item:root}--\ref{item:good-container} are clearly satisfied. However, we still need to show that the final tree has properties~\ref{item:max-deg-bound} and~\ref{item:height-bound}.

  To this end, suppose that $C \subseteq V(\HH)$ is not a good container, that is, $|C| > \alpha v(\HH)$ and every $W \subseteq C$ with $|W| \ge (1-\beta)|C|$ satisfies $e(\HH[W]) \ge E$. Lemma~\ref{lemma:Delta-1-supersaturation} invoked with $\HH \leftarrow \HH[C]$ supplies a subhypergraph $\HH' \subseteq \HH[C]$ with at least $E$ edges that satisfies
  \[
    \Delta_1(\HH') \le \left\lceil \frac{s}{\beta} \cdot \frac{e(\HH')}{|C|} \right\rceil \le \frac{2s}{\beta} \cdot \frac{e(\HH')}{|C|} = \frac{2s}{\beta} \cdot \frac{e(\HH')}{v(\HH')},
  \]
  where the second inequality follows from our assumption that $e(\HH') \ge E \ge v(\HH) \ge |C|$. Since, for every $t \in \{2, \dotsc, s\}$,
  \[
    \Delta_t(\HH') \le \Delta_t(\HH) \le \left(\frac{q}{10^6 s^5}\right)^{t-1} \cdot \frac{E}{v(\HH)} \le \left(\frac{q}{10^6 s^5}\right)^{t-1} \cdot \frac{e(\HH')}{v(\HH')},
  \]
  Theorem~\ref{thm:main-simple} invoked with $\HH \leftarrow \HH'$ and $K \leftarrow 2s/\beta$ supplies sets $\cS \subseteq \binom{C}{\le q \cdot |C|}$ and functions $f \colon \cS \to \cP(C)$ and $g \colon \II(\HH') \to \cS$ such that, for every $I \in \II(\HH')$,
  \[
    I \subseteq g(I) \cup f(g(I)) \qquad \text{and} \qquad \frac{|f(g(I))|}{|C|} \le 1 - \frac{\beta}{2 \cdot 10^3 s^5}.
  \]
  Since $\HH' \subseteq \HH[C]$, we have $\II(\HH[C]) \subseteq \II(\HH')$ and we may define
  \[
    \cC_C = \big\{g(I) \cup f(g(I)) : I \in \II(\HH[C])\big\}.
  \]
  By construction, the family $\cC_C$ is a family of containers for the independent sets of $\HH[C]$ and
  \[
    |\cC_C| \le |\cS| \le \sum_{i = 0}^{q \cdot |C|} \binom{|C|}{i} \le \left(\frac{e}{q}\right)^{q \cdot |C|} \le \left(\frac{e}{q}\right)^{q \cdot v(\HH)},
  \]
  establishing~\ref{item:max-deg-bound}. Finally, for every $D \in \cC_C$,
  \[
    \frac{|D|}{|C|} \le q + 1 - \frac{\beta}{2s} \cdot \frac{1}{10^3s^4} \le 1 - \frac{\beta}{10^4s^5},
  \]
  since we assumed that $10^4 s^5 q \le \beta$. In particular, if a set $C$ is a non-leaf vertex of the final tree $\TT$ that lies at distance $d$ from the root, then
  \[
    \alpha \le \frac{|C|}{v(\HH)} \le \left(1 - \frac{\beta}{10^4s^5}\right)^d \le \exp\left(-\frac{\beta d}{10^4s^5}\right).
  \]
  This implies that
  \[
    \height(\TT) \le \frac{10^4s^5}{\beta} \cdot \log\left(\frac{1}{\alpha}\right) + 1 \le \frac{10^4s^5}{\beta} \cdot \log\left(\frac{e}{\alpha}\right),
  \]
  establishing~\ref{item:height-bound}.
\end{proof}

\section{Properties of degree measures}
\label{sec:prop-degree-measure}

\subsection{Norms of degree measures}

As we shall be estimating the $\ell^2$-norms of $t$-degree measures of various uniform hypergraphs, we collect here several useful properties of this quantity. We first give general lower and upper bounds on the $\ell^2$-norm of the $t$-degree measure of a hypergraph in terms of the numbers of its vertices and edges and its maximum $t$-degree. Throughout this section, $r$ is a~positive integer. We stress here that all of our hypergraphs are allowed to have multiple edges, that is, every edge can have an arbitrary positive multiplicity. (This idea was first introduced by Saxton and Thomason~\cite{SaxTho15}.) Moreover, when computing $\deg_\HH$ and $e(\HH)$, we always count edges with multiplicities.

\begin{fact}
  \label{fact:codegree-measure-norm-maximum-degree}
  Suppose that $\HH$ is a nonempty $r$-uniform hypergraph. For every $t \in \br{r}$,
  \[
  \max\left\{ \frac{1}{\binom{v(\HH)}{t}} , \frac{1}{\binom{r}{t} \cdot e(\HH)} \right\} \le \| \sigma_\HH^{(t)} \|^2 \le \frac{\Delta_t(\HH)}{\binom{r}{t} \cdot e(\HH)}.
  \]
\end{fact}
\begin{proof}
  The upper bound is straightforward:
  \[
  \| \sigma_\HH^{(t)} \|^2 = \sum_{T \in \binom{V(\HH)}{t}} \left(\frac{\deg_\HH T}{\binom{r}{t} \cdot e(\HH)}\right)^2 \le \frac{\Delta_t(\HH)}{\binom{r}{t} \cdot e(\HH)} \cdot \sum_{T \in \binom{V(\HH)}{t}} \frac{\deg_\HH T}{\binom{r}{t} \cdot e(\HH)} = \frac{\Delta_t(\HH)}{\binom{r}{t} \cdot e(\HH)}.
  \]
  For the lower bound, let
  \[
    \TT = \left\{ T \in \binom{V(\HH)}{t} : \deg_\HH T > 0\right\}
  \]
  and observe that
  \[
    |\TT| \le \min\left\{ \binom{v(\HH)}{t}, \binom{r}{t} \cdot e(\HH) \right\}.
  \]
  It follows form the Cauchy--Schwarz inequality that
  \[
  \| \sigma_\HH^{(t)} \|^2 = \sum_{T \in \TT} \left(\frac{\deg_\HH T}{\binom{r}{t} \cdot e(\HH)}\right)^2 \ge \frac{1}{|\TT|} \cdot \left(\sum_{T \in \TT} \frac{\deg_\HH T}{\binom{r}{t} \cdot e(\HH)}\right)^2 = \frac{1}{|\TT|},
  \]
  implying the lower bound. 
\end{proof}

Our second observation states that the $\ell^2$-norm of a $t$-degree measure of a hypergraph cannot increase much when one deletes from it a small proportion of its edges.

\begin{fact}
  \label{fact:codegree-measure-subhypergraph}
  If $\HH'$ is a nonempty subhypergraph of an $r$-uniform hypergraph $\HH$, then, for every $t \in \br{r}$,
  \[
  \|\sigma_{\HH'}^{(t)}\| \le \frac{e(\HH)}{e(\HH')} \cdot \|\sigma_{\HH}^{(t)}\|.
  \]
\end{fact}
\begin{proof}
  The assertion follows simply because $\deg_{\HH'} T \le \deg_\HH T$ for every $T \subseteq V(\HH')$ and hence
  \[
  \sigma_{\HH'}^{(t)}(T) \le \frac{e(\HH)}{e(\HH')} \cdot \sigma_\HH^{(t)}(T).\qedhere
  \]
\end{proof}

Our final lemma relates the $\ell^2$-norm of the degree measure of a uniform hypergraph to a simple property of its edge distribution.

\begin{lemma}
  \label{lemma:degree-measure}
  Suppose that $\HH$ is a nonempty $r$-uniform hypergraph. If a~set $D \subseteq V(\HH)$ satisfies $e(\HH - D) \le (1-\eps) \cdot e(\HH)$ for some $\eps > 0$, then
  \[
    |D| \ge \left(\frac{\eps}{r}\right)^2 \cdot \| \sigma_\HH \|^{-2}.
  \]
\end{lemma}
\begin{proof}
  Let $\one_D \in \RR^{V(\HH)}$ be the characteristic vector of $D$. It follows from the Cauchy--Schwarz inequality that
  \[
    \lan \one_D, \sigma_\HH\ran^2 \le \| \one_D \|^2 \cdot \| \sigma_\HH \|^2 = |D| \cdot \| \sigma_\HH \|^2.
  \]
  Since at least an $\eps$-proportion of edges of $\HH$ contain at least one vertex of $D$,
  \[
    \lan \one_D, \sigma_\HH\ran = \frac{1}{r \cdot e(\HH)} \cdot \sum_{v \in D} \deg_\HH v \ge \frac{\eps}{r},
  \]
  giving the desired lower bound on $|D|$.
\end{proof}

\subsection{Degree measures and link hypergraphs}
 
Suppose that $\HH$ is an $(r+1)$-uniform hypergraph with vertex set $V$. Given a $v \in V$, we shall denote by $\HH_v$ the \emph{link hypergraph} of $v$ (the \emph{neighbourhood} of $v$ in $\HH$), that is, the $r$-uniform hypergraph with vertex set $V$ whose edges are all the $r$-element sets $A$ such that $\{v\} \cup A$ is an edge of $\HH$. A property of crucial importance for us is that, for each $t \in \br{r}$, the $t$-degree measure of $\HH$ is a convex combination of the $t$-degree measures of the link hypergraphs of its vertices. Moreover, each of these convex combinations has the same coefficients -- the coordinates of the $1$-degree measure vector $\sigma_\HH$.

\begin{remark}
  Even though $\sigma_\HH^{(t)}$ was defined only for nonempty hypergraphs $\HH$, for the sake of brevity, we shall often write $0 \cdot \sigma_\HH^{(t)}$ even if $\HH$ has no edges. In this case, $0 \cdot \sigma_\HH^{(t)}$ should be interpreted as the zero vector of appropriate dimension.
\end{remark}

\begin{fact}
  \label{fact:codegree-measure-neighbourhood-measure}
  Suppose that $\HH$ is a nonempty $(r+1)$-uniform hypergraph with vertex set $V$. For every $t \in \br{r}$,
  \[
  \sum_{v \in V} \sigma_\HH(v) \cdot \sigma_{\HH_v}^{(t)} = \sigma_\HH^{(t)}.
  \]
\end{fact}
\begin{proof}
  It follows from our definition of a link hypergraph that $e(\HH_v) = \deg_\HH v$ for each $v \in V$ and, more generally, for every $T \in \binom{V}{t}$,
  \[
    \deg_{\HH_v} T =
    \begin{cases}
      \deg_\HH (T \cup \{v\}) & \text{if $v \notin T$},\\
      0 & \text{if $v \in T$}.
    \end{cases}
  \]
  Consequently,
  \[
  \begin{split}
    \sum_{v \in V} \sigma_\HH(v) \cdot \sigma_{\HH_v}^{(t)}(T) & = \sum_{v \in V} \frac{\deg_\HH v}{(r+1) \cdot e(\HH)} \cdot \frac{\deg_{\HH_v} T}{\binom{r}{t} \cdot e(\HH_v)} \\
    & = \frac{1}{(r+1)\binom{r}{t} \cdot e(\HH)} \cdot \sum_{v \in V \setminus T} \deg_\HH (T \cup \{v\}) \\
    & = \frac{(r+1-t) \cdot \deg_\HH T}{(r+1) \binom{r}{t} \cdot e(\HH)} = \frac{\deg_\HH T}{\binom{r+1}{t} \cdot e(\HH)} = \sigma_\HH^{(t)}(T),
  \end{split}
  \]
  where we used the identity $(r+1)\binom{r}{t} = \binom{r+1}{t}(r+1-t)$.
\end{proof}

In our arguments, we shall employ the following relation between the $\ell^2$-norm of the $(t+1)$-degree measure of a hypergraph and the $\ell^2$-norms of the $t$-degree measures of the link hypergraphs of its vertices.

\begin{fact}
  \label{fact:codegree-measure-norm-square}
  Suppose that $\HH$ is a nonempty $(r+1)$-uniform hypergraph with vertex set $V$. For every $t \in \br{r}$,
  \[
  \sum_{v \in V} \sigma_\HH(v)^2 \cdot \| \sigma_{\HH_v}^{(t)}\|^2 = \frac{\| \sigma_\HH^{(t+1)} \|^2}{t+1}.
  \]
\end{fact}
\begin{proof}
  A quick way to verify the claimed identity is to observe that both the left- and the right-hand sides of the claimed equality express the probability of obtaining the same outcome in two independent executions of the following random process: Pick an edge $A$ of $\HH$ uniformly at random, choose a $(t+1)$-element subset $S$ of $A$ uniformly at random, mark a vertex $v \in S$ chosen uniformly at random, and return the pair $(v,S)$.

  More explicitly, using the identities $\deg_\HH v = e(\HH_v)$, valid for every $v \in V$, and $\deg_{\HH_v} T = \deg_{\HH} (T \cup \{v\})$, valid for each $T \in \binom{V}{t}$ and $v \in V \setminus T$, and $(r+1)\binom{r}{t} = (t+1)\binom{r+1}{t+1}$, we get
  \begin{align*}
    \sum_{v \in V} \sigma_\HH(v)^2 \cdot \|\sigma_{\HH_v}^{(t)}\|^2
    & = \sum_{v \in V} \left(\frac{\deg_\HH v}{(r+1) \cdot e(\HH)}\right)^2 \cdot \sum_{T \in \binom{V}{t}} \left(\frac{\deg_{\HH_v} T}{\binom{r}{t} \cdot e(\HH_v)}\right)^2 \\
    & = \left(\frac{1}{(r+1)\binom{r}{t} \cdot e(\HH)}\right)^2 \cdot \sum_{T \in \binom{V}{t}} \sum_{v \in V \setminus T} \big(\deg_\HH (T \cup \{v\})\big)^2 \\
    & = \left(\frac{1}{(t+1)\binom{r+1}{t+1} \cdot e(\HH)}\right)^2 \cdot \sum_{S \in \binom{V}{t+1}} \sum_{v \in S} \left(\deg_\HH S\right)^2 \\
    & = \frac{1}{t+1} \cdot \sum_{S \in \binom{V}{t+1}} \left(\frac{\deg_\HH S}{\binom{r+1}{t+1} \cdot e(\HH)} \right)^2 = \frac{\| \sigma_\HH^{(t+1)}\|^2}{t+1}.\qedhere
  \end{align*}
\end{proof}

\subsection{Linear combinations of degree measures}
\label{sec:lin-comb-deg-meas}

It will be convenient to introduce another piece of notation. Given a vector $\alpha \in \RR^r$ with nonnegative coordinates and a nonempty hypergraph $\KK$ with uniformity at least $r$, we define
\[
  \sigma_\alpha(\KK) = \left( \alpha_1^{1/2} \cdot \sigma_{\KK}^{(1)}, \dotsc, \alpha_r^{1/2} \cdot \sigma_{\KK}^{(r)} \right) \in \RR^{\sum_{t=1}^r \binom{v(\KK)}{t}},
\]
so that
\[
  \| \sigma_\alpha(\KK) \|^2 = \sum_{t=1}^r \alpha_t \cdot \| \sigma_{\KK}^{(t)} \|^2.
\]
The following generalisation of Fact~\ref{fact:codegree-measure-neighbourhood-measure} holds.

\begin{fact}
  \label{fact:codegree-measure-neighbourhood-measure-alpha}
  Suppose that $\KK$ is a nonempty hypergraph with uniformity at least $r+1$. For every $\alpha \in \RR^r$ with nonnegative coordinates,
  \[
    \sigma_\alpha(\KK) = \sum_{v \in V(\KK)} \sigma_\KK(v) \cdot \sigma_\alpha(\KK_v).
  \]
\end{fact}

\section{Proof}
\label{sec:proof}

\subsection{Outline}
\label{sec:outline}

Our proof of Theorem~\ref{thm:main} follows the general strategy of~\cite{BalMorSam15}. We construct functions $g, f^* \colon \II(\HH) \to \cP(V(\HH))$ that satisfy the following three conditions for every $I \in \II(\HH)$:
\begin{enumerate}[label={(\textit{\alph*})}]
\item 
  $g(I) \subseteq I \subseteq g(I) \cup f^*(I)$;
\item
  $|g(I)| \le 30s^2 \cdot p \cdot v(\HH)$ and $|f^*(I)| \le (1-\delta) \cdot v(\HH)$;
\item
  \label{item:outline-consistency}
  if $g(I) \subseteq I'$ and $g(I') \subseteq I$ for some $I' \in \II(\HH)$, then $g(I) = g(I')$ and $f^*(I) = f^*(I')$.
\end{enumerate}
The existence of such functions easily implies the assertion of the theorem. Indeed, condition~\ref{item:outline-consistency} guarantees that there is an implicit decomposition $f^* = f \circ g$.

Given an independent set $I$, the sets $g(I)$ and $f^*(I)$ are constructed by an algorithm that operates in a sequence of at most $s-1$ rounds, which are indexed by $r = s-1, \dotsc, 1$. At the start of round~$r$, the algorithm receives as input an $(r+1)$-uniform hypergraph $\HH^{(r+1)}$ satisfying $I \in \II(\HH^{(r+1)})$; we initialise the first round, indexed by $r = s-1$, with $\HH^{(s)} = \HH$. During the round, the algorithm tries to construct an $r$-uniform hypergraph $\HH^{(r)}$ with $I \in \II(\HH^{(r)})$ and some additional desirable properties. The definition of `desirable' is where we significantly depart from the previous approaches. In~\cite{BalMorSam15}, as well as in~\cite{SaxTho15}, this desirability property was defined in terms of a lower bound on the number of edges of $\HH^{(r)}$ and upper bounds on the maximum degrees $\Delta_t(\HH^{(r)})$, for all $t \in \br{r}$. Here, we aim to control the $\ell^2$-norms of the $t$-degree measures of $\HH^{(r)}$. More precisely, `desirable' means that a carefully chosen linear combination of $\| \sigma_{\HH^{(r)}}^{(t)} \|^2$, where $t$ ranges over $\br{r}$, is small. As in~\cite{BalMorSam15}, in the event that such a hypergraph $\HH^{(r)}$ cannot be constructed for our input set $I$, the algorithm is able to define the required set $f^*(I)$ already at the end of round $r$. Crucially, the amount of information about the set $I$ that is needed to describe $\HH^{(r)}$, or the set $f^*(I)$, is rather small. More precisely, it naturally corresponds to a set of $O(p \cdot s \cdot v(\HH))$ elements of $I$, which we shall denote here by $S^{(r)}$. In particular, if we let $g(I)$ be the union of the sets $S^{(r)}$ from all of the (at most $s-1$) rounds of the algorithm, the knowledge of $g(I)$ alone (without any additional knowledge of the set $I$ other than the fact that it is independent) is sufficient to recreate the entire execution of the algorithm, and thus also the final set $f^*(I)$.

In order to construct the $r$-uniform $\HH^{(r)}$ given the $(r+1)$-uniform $\HH^{(r+1)}$ and the set $I$, our algorithm considers a sequence of queries `Does $v$ belong to $I$?' for some carefully chosen sequence of vertices $v \in V(\HH)$. We record all the positive answers by placing the respective vertices $v$ in the (initially empty) set $S^{(r)}$. Each queried vertex $v$ is clearly not in the set $I \setminus S^{(r)}$ and hence it may be omitted from the set $f^*(I)$. In particular, the algorithm will produce the desired set $f^*(I)$ if at least $\delta \cdot v(\HH)$ queries are made. In case a queried vertex $v$ does belong to $I$, we add its $r$-uniform link hypergraph $\HH_v^{(r+1)}$ to the (initially empty) hypergraph $\HH^{(r)}$. Note that this guarantees that $I \in \II(\HH^{(r)})$.

Recall that our aim is to produce a hypergraph $\HH^{(r)}$ whose $t$-degree measures have small $\ell^2$-norms. The crux of the matter is the choice of the next vertex $v$ to be queried for membership in $I$. Indeed, if $v$ happens to belong to $I$, then $\HH_v^{(r+1)}$ will be added to $\HH^{(r)}$ and, as a result of this operation, the $t$-degree measures of $\HH^{(r)}$ will change -- the new $\sigma_{\HH^{(r)}}^{(t)}$ will be a convex combination of the old $\sigma_{\HH^{(r)}}^{(t)}$ and of $\sigma_{\HH_v^{(r+1)}}^{(t)}$, with appropriate coefficients (our hypergraphs are allowed to have multiple edges). It turns out that choosing the `right' candidate vertex $v$ is an optimisation problem that admits a rather simple geometric description. The solution to this geometric problem, presented as Lemma~\ref{lemma:geometry} and expressed in the language of degree measures of hypergraphs in Proposition~\ref{prop:geometry-hypergraph}, lies at the heart of our argument.

The bottom line is that there is a way to choose a sequence of vertices to be queried for membership in $I$ such that, if at least $\Omega(p \cdot s \cdot v(\HH))$ out of the first $\delta \cdot v(\HH)$ queried vertices belong to the set $I$, some linear combination of $\| \sigma_{\HH^{(r)}}^{(t)} \|^2$, where $t$ ranges over $\br{r}$, will be at most $1+O(1/s)$ times larger than a respective linear combination of $\| \sigma_{\HH^{(r+1)}}^{(t)} \|^2$, where $t$ ranges over $\br{r+1}$. Consequently, either one of the $s-1$ rounds of the algorithm will output a desired set $f^*(I)$ of size $\delta \cdot c(\HH)$ or the algorithm will eventually produce a $1$-uniform hypergraph $\HH^{(1)}$ such that $I \in \II(\HH^{(1)})$ and
\[
  \| \sigma_{\HH^{(1)}} \| \le \big(1+O(1/s)\big)^{s-1} \cdot \| \sigma_{\HH} \| \le O(1) \cdot \|\sigma_\HH\|.
\]
In case the latter happens, we may simply let $f^*(I)$ comprise all vertices $v$ such that $\{v\} \not\in \HH^{(1)}$. The upper bound on $\| \sigma_{\HH^{(1)}} \|$ implies that there are at most $(1-\delta) \cdot v(\HH)$ such vertices, as shown in Lemma~\ref{lemma:degree-measure}.

\subsection{The key lemma}
\label{sec:key-lemma}

The following lemma summarises a single round of our new, refined algorithm for constructing containers. We denote by $\Hyp_r(V)$ the family of $r$-uniform hypergraphs with vertex set $V$; recall again that we allow our hypergraphs to have multiple edges.

\begin{lemma}
  \label{lemma:key}
  Let $\GG$ be an $(r+1)$-uniform hypergraph with vertex set $V$. Suppose that $\eps \in \big(0, (9(r+1))^{-1}\big)$ and $p \in (0,1)$ satisfy
  \[
    \| \sigma_\GG \|^2 \le \frac{\eps^3 p}{ 50 (r+1)}.
  \]
  Let $\alpha \in \RR^r$ be a vector with nonnegative coordinates and define
  \begin{equation}
    \label{eq:def-alphan-b}
    \alphan = (1+\eps)^{10} \cdot (\alpha, 0) + \frac{50 (r+1)}{\eps^2 p} \cdot (0, \alpha) \in \RR^{r+1}
    \qquad \text{and} \qquad
    b = \left\lceil \frac{2p}{\eps} \cdot |V|\right\rceil.
  \end{equation}
  Then, there exist (disjoint) families $\cS', \cS'' \subseteq \binom{V}{\le b}$ and functions $S \colon \II(\GG) \to \cS' \cup \cS''$,  $C \colon \cS' \to \cP(V)$, and $\FF \colon \cS'' \to \Hyp_r(V)$ such that, for every $I \in \II(\GG)$, we have $S_I \subseteq I$. Moreover:
  \begin{enumerate}[label={(\arabic*)}]
  \item
    If $S_I \in \cS'$, then $I \setminus S_I \subseteq C(S_I)$ and $|V| - |C(S_I)| \ge \frac{\eps^2}{(r+1)^2} \cdot \| \sigma_\GG \|^{-2}$.
  \item
    If $S_I \in \cS''$, then $I \in \II(\FF(S_I))$ and $\| \sigma_\alpha(\FF(S_I)) \| \le \| \sigma_{\alphan}(\GG) \|$.
  \end{enumerate}
  Finally, if $S_I \subseteq I'$ and $S_{I'} \subseteq I$ for some $I, I' \in \II(\GG)$, then $S_I = S_{I'}$.
\end{lemma}

Before we embark on the proof of Lemma~\ref{lemma:key}, we shall first show, in the next subsection, how it implies Theorem~\ref{thm:main}. The remainder of this section, Subsections~\ref{sec:pruning-hypergraphs}--\ref{sec:proof-key-lemma} will be devoted to the proof of the lemma.

\subsection{Derivation of Theorem~\ref{thm:main}}
\label{sec:derivation-thm-main}

Let $\HH$ be a nonempty $s$-uniform hypergraph with vertex set $V$ and suppose that $\delta, p \in (0,1)$ satisfy~\eqref{eq:assumption-main}, that is,
\[
  \tag{\ref{eq:assumption-main}}
  300s^4 \cdot \sum_{t=1}^s \binom{s-1}{t-1} \left(\frac{5000s^3}{p}\right)^{t-1} \| \sigma_\HH^{(t)} \|^2 \le \frac{1}{\delta \cdot v(\HH)} \le \frac{p}{500}.
\]
Define
\[
  \eps = \frac{1}{10s} \qquad \text{and} \qquad \Gamma = \frac{50s}{\eps^2}
\]
and let $\alpha^{(1)} \in \RR^1$, \ldots, $\alpha^{(s)} \in \RR^s$ be vectors defined by
\[
  \alpha^{(r)}_t = \binom{r-1}{t-1} \cdot (1+\eps)^{10(r-t)} \cdot \left(\frac{\Gamma}{p}\right)^{t-1},
\]
for every $r \in \br{s}$ and each $t \in \br{r}$. Given an independent set $I$ of $\HH$, we construct the sets $g(I)$ and $f^*(I)$ using the following procedure.

\medskip
\noindent
\textbf{Construction of the container.}
Let $\HH^{(s)} = \HH$. Do the following for $r = s-1, \dotsc, 1$:
\begin{enumerate}[label=(C\arabic*)]
\item
  \label{item:constr-run-algorithm}
  Invoke Lemma~\ref{lemma:key} with $\GG \leftarrow \HH^{(r+1)}$ and $\alpha \leftarrow \alpha^{(r)}$ to obtain families $\cS'$ and $\cS''$ and functions $S$, $C$, and $\FF$, as in the statement of the lemma.\footnote{In order to do so, we have to make sure that $\|\sigma_{\HH^{(r+1)}}\| \le \frac{\eps^3p}{50(r+1)}$. In the analysis of the procedure, below, we will verify that this is always the case.}
\item
  \label{item:constr-update-signature}
  Let $S^{(r)} \leftarrow S_I$.
\item
  \label{item:constr-determine-many}
  If $S_I \in \cS'$, then let $g(I) = S^{(s-1)} \cup \dotsb \cup S^{(r)}$ and $f^*(I) = C(S_I)$ and \texttt{STOP}.
\item
  \label{item:constr-reduce-uniformity}
  Otherwise, if $S_I \in \cS''$, we let $\HH^{(r)} \leftarrow \FF(S_I)$ and \texttt{CONTINUE}.
\end{enumerate}
If \texttt{STOP} has not been called, then $r=1$ and $\HH^{(1)}$ has been defined. Let $g(I) = S^{(s-1)} \cup \dotsb \cup S^{(1)}$ and $f^*(I) = \big\{ v \in V : \{v\} \not\in \HH^{(1)} \big\}$.

\medskip

In the remainder of this section, we shall show that, for every $I \in \II(\HH)$, the above procedure indeed constructs sets $g(I)$ and $f^*(I)$ that have the desired properties; in particular, we shall show that $f^*$ decomposes as $f^* = f \circ g$.

\begin{claim}
  \label{claim:container-propties}
  For each $r \in \br{s}$, the hypergraph $\HH^{(r)}$, if it was defined, satisfies:
  \begin{enumerate}[label=(\textit{\roman*})]
  \item
    \label{item:container-propty-1}
    $I \in \II(\HH^{(r)})$ and
  \item
    \label{item:container-propty-2}
    $\| \sigma_{\alpha^{(r)}}(\HH^{(r)}) \|^2 \le \frac{\eps^2}{s^2 \delta |V|}$.
  \end{enumerate}
\end{claim}

The proof of Claim~\ref{claim:container-propties} requires the following simple fact, which justifies our definition of the vectors $\alpha^{(1)}, \dotsc, \alpha^{(s)}$.

\begin{fact}
  \label{fact:alpha-induction}
  Suppose that $r \in \br{s-1}$. Let $\alpha \in \RR^r$ and let $\alphan \in \RR^{r+1}$ be defined as in~\eqref{eq:def-alphan-b}. If $\alpha_t \le \alpha_t^{(r)}$ for all $t \in \br{r}$, then $\alphan_t \le \alpha_t^{(r+1)}$ for all $t \in \br{r+1}$.
\end{fact}
\begin{proof}
  Note first that
  \[
    \alphan_1 = (1+\eps)^{10} \cdot \alpha_1 \le (1+\eps)^{10} \cdot \alpha_1^{(r)} = (1+\eps)^{10r} = \alpha_1^{(r+1)}
  \]
  and, since $50(r+1) \le 50s = \Gamma\eps^2$,
  \[
    \alphan_{r+1} = \frac{50(r+1)}{\eps^2p} \cdot \alpha_r \le \frac{\Gamma}{p} \cdot \alpha_r^{(r)} = \left(\frac{\Gamma}{p}\right)^{r} = \alpha_{r+1}^{(r+1)}.
  \]
  Finally, if $1 < t \le r$, then
  \[
    \begin{split}
      \alphan_t & = (1+\eps)^{10} \cdot \alpha_t + \frac{50(r+1)}{\eps^2p} \cdot \alpha_{t-1} \le (1+\eps)^{10} \cdot \alpha_t^{(r)} + \frac{\Gamma}{p} \cdot \alpha_{t-1}^{(r)} \\
      & = \binom{r-1}{t-1} (1+\eps)^{10(r-t+1)} \left(\frac{\Gamma}{p}\right)^{t-1} + \binom{r-1}{t-2} (1+\eps)^{10(r-(t-1))} \left(\frac{\Gamma}{p}\right)^{t-2+1} \\
      & = \binom{r}{t-1} (1+\eps)^{10(r+1-t)} \left(\frac{\Gamma}{p}\right)^{t-1} = \alpha_t^{(r+1)},
    \end{split}
  \]
  where the second to last equality is Pascal's formula.
\end{proof}

\begin{proof}[{Proof of Claim~\ref{claim:container-propties}}]
  We prove the claim by induction on $s-r$. The basis of the induction is the case $r = s$. Property~\ref{item:container-propty-1} is satisfied, as $\HH^{(s)} = \HH$ and $I \in \II(\HH)$. In order to see that property~\ref{item:container-propty-2} holds as well, note first that the first inequality in the main assumption~\eqref{eq:assumption-main} of Theorem~\ref{thm:main} gives
  \[
    \begin{split}
      \|\sigma_{\alpha^{(s)}}(\HH^{(s)})\|^2 = \| \sigma_{\alpha^{(s)}}(\HH) \|^2 & = \sum_{t=1}^s \binom{s-1}{t-1} (1+\eps)^{10(s-t)} \left(\frac{\Gamma}{p}\right)^{t-1} \| \sigma_\HH^{(t)} \|^2 \\
      & \le (1+\eps)^{10s} \cdot  \sum_{t=1}^s \binom{s-1}{t-1} \left(\frac{\Gamma}{p}\right)^{t-1} \| \sigma_\HH^{(t)} \|^2  \\
      & \le (1+\eps)^{10s} \cdot \frac{1}{300s^4 \cdot \delta \cdot |V|} \le \frac{\eps^2}{s^2 \cdot \delta \cdot |V|},
    \end{split}
  \]
  where the second inequality holds as $\Gamma = 5000s^3$ and the last inequality holds since $\eps = 1/(10s)$ and, consequently, $(1+\eps)^{10s} \le e^{10\eps s} = e \le 3$.
    
  Suppose now that $r \in \br{s-1}$ and that the hypergraph $\HH^{(r+1)}$ was defined. We first argue that we are allowed to invoke Lemma~\ref{lemma:key} in step~\ref{item:constr-run-algorithm} above. Indeed, since the second inequality in \eqref{eq:assumption-main} implies that
  \[
    \frac{1}{s \cdot \delta \cdot |V|} = \frac{10\eps}{\delta \cdot |V|} \le \frac{\eps p}{50}
  \]
  and $\alpha_1^{(r+1)} = (1+\eps)^{10r} \ge 1$, our inductive assumption implies that
  \[
    \|\sigma_{\HH^{(r+1)}}\|^2 \le \alpha_1^{(r+1)} \cdot \|\sigma_{\HH^{(r+1)}}\|^2 \le \|\sigma_{\alpha^{(r+1)}} (\HH^{(r+1)})\|^2 \le \frac{\eps^2}{s^2\delta |V|} \le \frac{\eps^3p}{50(r+1)},
  \]
  as needed (in order to apply Lemma~\ref{lemma:key} with $\GG \leftarrow \HH^{(r+1)}$). If the hypergraph $\HH^{(r)}$ was defined, in step~\ref{item:constr-reduce-uniformity}, then Lemma~\ref{lemma:key} guarantees that
  \[
    \|\sigma_{\alpha^{(r)}}(\HH^{(r)})\|^2 \le \| \sigma_{\alpha^*} (\HH^{(r+1)}) \|^2,
  \]
  where $\alphan$ is defined as in \eqref{eq:def-alphan-b}, with $\alpha \leftarrow \alpha^{(r)}$. However, Fact~\ref{fact:alpha-induction} states that $\alphan \le \alpha^{(r+1)}$ coordinate-wise and, therefore, we may conclude that
  \[
    \|\sigma_{\alpha^{(r)}}(\HH^{(r)})\|^2 \le \| \sigma_{\alpha^*} (\HH^{(r+1)}) \|^2 \le \| \sigma_{\alpha^{(r+1)}} (\HH^{(r+1)}) \|^2,
  \]
  as needed.
\end{proof}

We now verify that $g(I)$ and $f^*(I)$ have the desired properties and that $f^*$ decomposes as $f^* = f \circ g$. Since $g(I) = S^{(s-1)} \cup \dotsb \cup S^{(r)}$, for some $r \in \br{s-1}$, the fact that $g(I) \subseteq I$ is an immediate consequence of the definitions of $S^{(s-1)}, \dotsc, S^{(r)}$, made in step~\ref{item:constr-update-signature}, and the fact that the respective sets $S_I$ are all contained in $I$, as guaranteed by Lemma~\ref{lemma:key}. Moreover, since each of these sets $S_I$ has at most $\lceil 2p|V|/\eps \rceil$ elements, see Lemma~\ref{lemma:key}, we have
\[
  |g(I)| \le |S^{(s-1)}| + \dotsb + |S^{(r)}| \le (s-1) \cdot \left\lceil \frac{2p \cdot |V|}{\eps} \right\rceil \le \frac{2sp}{\eps} \cdot |V| = 20s^2p \cdot |V|,
\]
where the last inequality holds because the main assumption~\eqref{eq:assumption-main} and Fact~\ref{fact:codegree-measure-norm-maximum-degree} imply that $p \ge s^4 \cdot \| \sigma_\HH^{(1)}\|^2 \ge s^4 / |V|.$

The set $f^*(I)$ is defined either in step~\ref{item:constr-determine-many}, for some $r \in \br{s-1}$, or at the end of the procedure, if the $1$-uniform hypergraph $\HH^{(1)}$ is constructed. In the former case, $f^*(I) = C(S_I)$ for functions $S$ and $C$ obtained from Lemma~\ref{lemma:key}. Note that $I \setminus g(I) \subseteq I \setminus S_I \subseteq C(S_i) = f^*(I)$, as $g(I) \supseteq S^{(r)} = S_I$. Moreover, $|f^*(I)| \le (1-\delta)|V|$, since, on the one hand, Lemma~\ref{lemma:key} guarantees that
\[
  |V| - |f^*(I)| \ge \frac{\eps^2}{(r+1)^2} \cdot \| \sigma_{\HH^{(r+1)}} \|^{-2}
\]
and, on the other hand, by Claim~\ref{claim:container-propties}, as $\alpha_1^{(r+1)} = (1+\eps)^{10r} \ge 1$,
\[
  \| \sigma_{\HH^{(r+1)}} \|^2 \le \alpha_1^{(r+1)} \cdot \|\sigma_{\HH^{(r+1)}}\|^2 \le \|\sigma_{\alpha^{(r+1)}} (\HH^{(r+1)})\|^2 \le \frac{\eps^2}{s^2 \delta|V|} \le \frac{\eps^2}{(r+1)^2 \delta |V|}.
\]
In the latter case, $f^*(I) = \big\{v \in V : \{v\} \not\in \HH^{(1)} \big\}$. In particular, we must have $I \subseteq f^*(I)$, since otherwise $I$ would not be an independent set in $\HH^{(1)}$, which would contradict property~\ref{item:container-propty-1} in Claim~\ref{claim:container-propties}. Moreover, Lemma~\ref{lemma:degree-measure}, invoked with $\HH \leftarrow \HH^{(1)}$, $D \leftarrow V \setminus f^*(I)$, and $\eps \leftarrow 1$, gives $|V \setminus f^*(I)| \ge \| \sigma_{\HH^{(1)}} \|^{-2}$. Since $\alpha_1^{(1)} = 1$, we have $\sigma_{\HH^{(1)}} = \sigma_{\alpha^{(1)}} (\HH^{(1)})$ and property~\ref{item:container-propty-2} in Claim~\ref{claim:container-propties} allows us to conclude that
\[
  |V| - |f^*(I)| \ge \| \sigma_{\alpha^{(1)}}(\HH^{(1)}) \|^{-2} \ge \frac{s^2 \delta |V|}{\eps^2} \ge \delta |V|.
\]

Finally, we show that $f^*$ decomposes as $f^* = f \circ g$. To this end, it suffices to show that if $g(I) = g(I')$, for some $I, I' \in \II(\HH)$, then $f^*(I) = f^*(I')$.
In fact, we shall prove the following stronger statement.

\begin{claim}
  \label{claim:f-star-decomposes}
  If $g(I) \subseteq I'$ and $g(I') \subseteq I$ for some $I, I' \in \II(\HH)$, then $g(I) = g(I')$ and $f^*(I) = f^*(I')$.
\end{claim}

Note that Claim~\ref{claim:f-star-decomposes} implies the desired property of $f^*$. Indeed, assume that $g(I) = g(I')$.
Since $g(I) \subseteq I$ and $g(I') \subseteq I'$, as shown above, we have $g(I) = g(I') \subseteq I \cap I'$ and the claim yields $f^*(I) = f^*(I')$.

\begin{proof}[{Proof of Claim~\ref{claim:f-star-decomposes}}]
  The claim is an easy consequence of the respective property of the function $S$ from the statement of Lemma~\ref{lemma:key}. Indeed, it suffices to show that the container-constructing procedure described above defines the same sets $S^{(r)}$ and the same hypergraphs $\HH^{(r)}$ when applied to both $I$ and $I'$; this is because $g(I)$ and $g(I')$ are unions of the respective sets $S^{(r)}$ and the sets $f^*(I)$ and $f^*(I')$ depend only on the sets $S^{(r)}$ and the hypergraphs $\HH^{(r)}$. One may prove this assertion by induction on $s-r$. For the induction step, note that while the procedure performs step~\ref{item:constr-run-algorithm}, the respective hypergraphs $\HH^{(r+1)}$ are identical (by the inductive assumption) and therefore so are the functions $S$, $C$, and $\FF$. Moreover, since $S_I \subseteq g(I)$ and $S_{I'} \subseteq g(I')$, by our definition of $g(I)$ and $g(I')$, we may conclude that $S_I \subseteq I'$ and $S_{I'} \subseteq I$ and thus, the final assertion of Lemma~\ref{lemma:key} gives us the equality $S_I = S_{I'}$.
\end{proof}

\subsection{Pruning hypergraphs}
\label{sec:pruning-hypergraphs}

In order to streamline the analysis of our algorithm that constructs the $r$-uniform hypergraph $\HH^{(r)}$ from the $(r+1)$-uniform $\HH^{(r+1)}$, we will first prune the latter hypergraph by removing from it vertices with unusually high degree. More precisely, define, for a nonempty $r$-uniform hypergraph $\HH$ and $t \in \br{r}$,
\begin{equation}
  \label{eq:Del}
  \Del_t(\HH) = \frac{r}{t} \cdot e(\HH) \cdot \| \sigma_\HH^{(t)} \|^2;
\end{equation}
one should think of $\Del_t(\HH)$ as a robust analogue of the maximum degree $\Delta_t(\HH)$. In particular, Fact~\ref{fact:codegree-measure-norm-maximum-degree} implies that $\Del_1(\HH) \le \Delta_1(\HH)$; even though equality sometimes holds (when $\HH$ is regular), in general the ratio $\Delta_1(\HH) / \Del_1(\HH)$ can be arbitrarily large.  Our next lemma shows that this inequality becomes nearly tight, up to a multiplicative factor of $O(r)$, after we delete a small proportion of the edges of $\HH$.

\begin{lemma}
  \label{lemma:Delta-Del}
  Suppose that $\HH$ is a nonempty $r$-uniform hypergraph. Then, for every $R \ge r$, there is an $\HH' \subseteq \HH$ with $e(\HH') > (1 - \frac{r}{R}) \cdot e(\HH)$ such that $\Delta_1(\HH') \le R \cdot \Del_1(\HH)$.
\end{lemma}
\begin{proof}
  Given a nonempty $r$-uniform hypergraph $\HH$ and $R \ge r$, define
  \[
    X = \left\{v \in V(\HH) : \deg_\HH v > R \cdot \Del_1(\HH) \right\}.
  \]
  By the definition of $X$ and $\Del_1(\HH)$, we have
  \[
  \| \sigma_\HH \|^2 \ge \sum_{v \in X} \left(\frac{\deg_\HH v}{r \cdot e(\HH)}\right)^2 > \frac{R \cdot \Del_1(\HH)}{r \cdot e(\HH)} \cdot \sum_{v \in X} \frac{\deg_\HH v}{r \cdot e(\HH)} = R \cdot \| \sigma_\HH \|^2 \cdot \sum_{v \in X} \frac{\deg_\HH v}{r \cdot e(\HH)},
  \]
  which implies that
  \[
  \sum_{v \in X} \deg_\HH v < \frac{r}{R} \cdot e(\HH).
  \]
  In particular, deleting from $\HH$ all edges containing at least one vertex of $X$ yields a hypergraph $\HH'$ satisfying the assertion of this lemma.
\end{proof}

Since the degree measures $\sigma^{(t)}(\HH)$ are not defined when $\HH$ is an empty hypergraph, in order to streamline our analysis, we will start building the hypergraph $\HH^{(r)}$ by seeding it with a fixed well-behaved $r$-uniform hypergraph. In order to guarantee that, at the end of the algorithm, this initial seed constitutes only a negligible proportion of $\HH^{(r)}$, we need to make sure that the link hypergraphs $\HH_v^{(r+1)}$ that the algorithm adds to $\HH^{(r)}$ are somewhat large. We will achieve this by (temporarily) removing vertices of very small degree from various subhypergraphs of $\HH^{(r+1)}$.

\begin{fact}
  \label{fact:min-avg-degree}
  Suppose that $\HH$ is a nonempty hypergraph with vertex set $V$. Then, for every $\beta > 0$, there is a spanning $\HH' \subseteq \HH$ such that $e(\HH') \ge (1-\beta) \cdot e(\HH)$ and, for each $v \in V$, $\deg_{\HH'} v$ is either zero or at least $\beta \cdot e(\HH) / |V|$.
\end{fact}
\begin{proof}
  Form a spanning subgraph $\HH'$ of $\HH$ by iteratively deleting all edges containing some vertex with degree smaller than $\beta \cdot e(\HH) / |V|$. Clearly, each edge of $\HH'$ contains only vertices with degrees at least $\beta \cdot e(\HH) / |V|$. Moreover, the number of edges deleted from $\HH$ while forming $\HH'$ cannot exceed $\beta \cdot e(\HH)$.
\end{proof}

\subsection{The algorithm}
\label{sec:algorithm}

We are ready to present the algorithm that underlies the proof of Lemma~\ref{lemma:key}.

\subsubsection*{Input}

Let $\GG$ be an $(r+1)$-uniform hypergraph with vertex set $V$. Let $\eps \in \big(0, (9(r+1))^{-1}\big)$ and $p \in (0,1)$, define
\[
  a = \frac{25}{\eps^2},
\]
and suppose that
\begin{equation}
  \label{eq:sigma-GG-assumption}
  \| \sigma_\GG \|^2 \le \frac{\eps^3 p}{ 50 (r+1)} = \frac{\eps p}{2 a (r+1)}.
\end{equation}
Observe that uniformly scaling the multiplicities of all edges of $\GG$ by a positive integer factor $k$ does not affect the $t$-degree measure $\sigma_{\GG}^{(t)}$, for any $t \in \br{r}$, nor does it change the family $\II(\GG)$ of independent sets of $\GG$. It does, however, increase the value of $e(\GG)$, and thus also the value of $\Del_t(\GG)$, for each $t \in \br{r}$, by the same multiplicative factor $k$. Consequently, we may assume, without loss of generality, that there is a (large) positive integer $m$ such that
\begin{equation}
  \label{eq:m-choice}
  \frac{m}{2a} \cdot \binom{|V|}{r} \le \Del_1(\GG) \le \frac{m}{a} \cdot \binom{|V|}{r}.
\end{equation}
Finally, let $\alpha \in \RR^r$ be a vector with nonnegative coordinates and let $I$ be an independent set of $\GG$.

\subsubsection*{Setup}

Let $L$ be the empty set and let $\GGn^{(0)}$ be the hypergraph obtained from the complete $r$-uniform hypergraph with vertex set $V$ by changing the multiplicities of all of its edges to $m$, so that $e(\GGn^{(0)}) = m \binom{|V|}{r}$.
Further, apply Lemma~\ref{lemma:Delta-Del}, with $R \leftarrow \frac{r+1}{\eps}$, to find an $\cA^{(0)} \subseteq \GG$ satisfying
\begin{equation}
  \label{eq:A0-properties}
  e(\cA^{(0)}) \ge (1-\eps) \cdot e(\GG) \qquad \text{and} \qquad \Delta_1(\cA^{(0)}) \le \frac{r+1}{\eps} \cdot \Del_1(\GG).
\end{equation}
Finally, let
\[
  b = \left\lceil \frac{2p}{\eps} \cdot |V|\right\rceil
  \qquad \text{and} \qquad
  \sigma = (1-3\eps)^{-1} \cdot \| \sigma_\alpha(\GG) \|.
\]

\subsubsection*{Main loop}

Do the following for $j = 0, 1, \ldots$:
\begin{enumerate}[label={(S\arabic*)}]
\item
  \label{item:alg-stop}
  If $|L| = b$ or $e(\cA^{(j)}) < (1-2\eps) \cdot e(\GG)$, then let $J = j$ and \texttt{STOP}.
\item
  \label{item:alg-best-offer}
  Let $\gcA^{(j)}$ be a canonically chosen spanning subgraph of $\cA^{(j)}$ satisfying the assertion of Fact~\ref{fact:min-avg-degree} with $\beta \leftarrow \eps$. For each $v \in V$, let $\GGn^{(j,v)} = \GGn^{(j)} \cup \gcA_v^{(j)}$ and let $v_j$ be a canonically chosen vertex that minimises the quantity
  \[
  e(\GGn^{(j,v)}) \cdot \left( \| \sigma_\alpha(\GGn^{(j,v)}) \|^2 - (1+\eps) \cdot \sigma^2 \right)
  \]
  over all $v \in V$ whose degree in $\gcA^{(j)}$ is nonzero.
  
\item
  \label{item:alg-main-step}
  If $v_j \in I$, then add $j$ to the set $L$ and let
  \[
  \GGn^{(j+1)} = \GGn^{(j,v_j)} = \GGn^{(j)} \cup \gcA_{v_j}^{(j)}.
  \]
  Otherwise, let $\GGn^{(j+1)} = \GGn^{(j)}$.
\item
  \label{item:alg-cleanup}
  Let $\cA^{(j+1)}$ be the hypergraph obtained from $\cA^{(j)}$ by removing all edges containing $v_j$.
\end{enumerate}

\subsubsection*{Output}

After \texttt{STOP} is called in the main loop, let $\cA = \cA^{(J)}$ and $\GGn = \GGn^{(J)} \setminus \GGn^{(0)}$, that is, $\GGn$ is the hypergraph satisfying $\GGn^{(J)} = \GGn \cup \GGn^{(0)}$. (Recall once more that all our hypergraphs contain edges with multiplicities.)

\subsection{Basic properties of the algorithm and the key dichotomy}
\label{sec:basic-prop-alg}

In this section, we establish several basic properties of the algorithm and state its key `dichotomy' property, which we shall derive in later sections. Moreover, we explain how to use the algorithm to prove Lemma~\ref{lemma:key}. We start by showing that the algorithm terminates on every input and that the output hypergraph $\GGn$ and the final set $L$ retain important information about the input set $I$.

\begin{obs}
  \label{obs:algorithm-basic-properties}
  For every $I \in \II(\GG)$, the map $\{0, \dotsc, J-1\} \ni j \mapsto v_j \in V$ is injective and hence the algorithm terminates. Moreover, $I \in \II(\GGn)$ and $L = \big\{j \in \{0, \dotsc, J-1\} : v_j \in I\big\}$.
\end{obs}
\begin{proof}
  The first assertion holds because in step~\ref{item:alg-cleanup}, all edges containing $v_j$ are removed, and hence its degree remains zero in each $\cA^{(j')}$ with $j' > j$. Therefore, the algorithm stops after at most $|V|$ iterations of the main loop. The second assertion holds because $\GGn$ comprises only edges of the link hypergraphs $\GG_v$ for which $v \in I$ and because $j \in L$ if and only if $v_j \in I$.
\end{proof}

We next observe that the set $L$ contains all the information about the input set $I$ that is needed to reconstruct the execution of the algorithm.

\begin{obs}
  \label{obs:algorithm-consistency}
  If the algorithm produces the same set $L$ for two inputs $I, I' \in \II(\GG)$, then also the hypergraphs $\GGn$, the numbers $J$, and the sequences $v_0, \dotsc, v_{J-1}$ are the same for both inputs.
\end{obs}
\begin{proof}
  The only decisions that depend on the input set $I$ are taken in step~\ref{item:alg-main-step} of the algorithm. Each time this step is executed, the decision taken is encoded in the set $L$ by placing, or not placing, the index $j$ in $L$.
\end{proof}

The function $S \colon \II(\GG) \to \cP(V)$ whose existence is asserted by Lemma~\ref{lemma:key} will be defined as follows:
\begin{equation}
  \label{eq:SI-def}
  S_I = \{v_j \colon j \in L\} = \{v_j \colon 0 \le j < J \text{ and } v_j \in I\}.
\end{equation}
In other words, $S_I$ comprises precisely those among the queried vertices $v_0, \dotsc, v_{J-1}$ that belong to the input set $I$. Note that $|S_I| = |L| \le b$, since the algorithm terminates in step~\ref{item:alg-stop} as soon as $L$ has $b$ elements. We shall now show that the knowledge of the set $S_I$ is enough to reconstruct the final set $L$ and hence, as stated in Observation~\ref{obs:algorithm-consistency}, the entire execution of the algorithm. In fact, the following stronger statement is true.

\begin{lemma}
  \label{lemma:S-greedy-property}
  Suppose that, for two inputs $I, I' \in \II(\GG)$, we have $S_I \subseteq I'$ and $S_{I'} \subseteq I$. Then, for both these inputs, the algorithm outputs the same set $L$.
\end{lemma}

Suppose that $S_I = S_{I'}$ for some $I, I' \in \II(\GG)$. As $S_I \subseteq I$ and $S_{I'} \subseteq I'$, by construction, Lemma~\ref{lemma:S-greedy-property} implies that the output set $L$ must be the same for both $I$ and $I'$.

\begin{proof}[Proof of Lemma~\ref{lemma:S-greedy-property}]
  Suppose that two inputs $I$ and $I'$ yield sets $L$ and $L'$, respectively, with $L \neq L'$. Let $j$ be the smallest index such that $j \in (L \setminus L') \cup (L' \setminus L)$; without loss of generality, we may assume that $j \in L \setminus L'$. Since $L \cap \{0, \dotsc, j-1\} = L' \cap \{0, \dotsc, j-1\}$, the algorithm produces the same sequences $v_0, \dotsc, v_j$ while working with inputs $I$ and $I'$. Since $j \in L \setminus L'$, we must have $v_j \in S_I$ and $v_j \notin I'$. In particular, $S_I \nsubseteq I'$.
\end{proof}

Finally, define the vector $\alphan \in \RR^{r+1}$ as in~\eqref{eq:def-alphan-b}:
\[
  \alphan = (1+\eps)^{10} \cdot (\alpha, 0) + \frac{50(r+1)}{\eps^2 p} \cdot (0, \alpha).
\]
The key dichotomy property, stated in our next lemma, is that either the algorithm inspects many vertices of the hypergraph (before encountering the $b$th vertex of $I$) or the final hypergraph $\GGn$ is a good `model' of $\GG$, in the sense that the $\ell^2$-norm of $\sigma_\alpha(\GGn)$ does not exceed the $\ell^2$-norm of $\sigma_{\alphan}(\GG)$.

\begin{lemma}
  \label{lemma:analysis}
  At least one of the following holds:
  \begin{enumerate}[label={(\arabic*)}]
  \item
    \label{item:many-skipped}
    $J \ge \frac{\eps^2}{(r+1)^2} \cdot \| \sigma_\GG \|^{-2}$,
  \item
    \label{item:good-GGn}
    $\| \sigma_\alpha(\GGn) \| \le \| \sigma_{\alphan}(\GG) \|$.
  \end{enumerate}
\end{lemma}

We shall prove Lemma~\ref{lemma:analysis}, which lies at the heart of the matter, in the next two sections. We finish the current section with a short derivation of Lemma~\ref{lemma:key}, which is now straightforward. Given an $(r+1)$-uniform hypergraph $\GG$ and numbers $\eps$ and $p$ as in the statement of the lemma, we may define the function $S \colon I(\GG) \to \cP(V)$ as in~\eqref{eq:SI-def}, by running the algorithm on each input $I \in \II(\GG)$. If $J \ge \frac{\eps^2}{(r+1)^2} \cdot \| \sigma_\GG \|^{-2}$, we place $S_I$ in the family $\cS'$ and let $C(S_I) = V \setminus \{v_j : 0 \le j < J\}$; note that $I \setminus S_I \subseteq C(S_I)$ as $S_I = I \cap (V \setminus C(S_I))$ by~\eqref{eq:SI-def}. Otherwise, we place $S_I$ in the family $\cS''$ and let $\FF(S_I) = \GGn$; Lemma~\ref{lemma:analysis} implies the desired property of each such hypergraph $\FF(S_I)$. Lemma~\ref{lemma:S-greedy-property} and Observation~\ref{obs:algorithm-consistency} guarantee that the set $C(S_I)$ or the hypergraph $\FF(S_I)$ depend only on the set $S_I$, and not on $I$ itself, and that the function $S$ has the claimed consistency property.

\subsection{The geometric lemma}

The most important elementary operation performed by the algorithm described in Section~\ref{sec:algorithm} is to choose some $v \in V$ and add the $r$-uniform link hypergraph $\gcA_v^{(j)}$ to the hypergraph $\GGn^{(j)}$, obtaining a new hypergraph $\GGn^{(j,v)} = \GGn^{(j)} \cup \gcA_v^{(j)}$. Since we want the final hypergraph $\GGn$ to have small $\ell^2$-norm of $\sigma_{\alpha}(\GGn)$, in step~\ref{item:alg-best-offer} of the algorithm, we consider a vertex $v$ that, essentially, minimises the $\ell^2$-norm of $\sigma_{\alpha}(\GGn^{(j,v)})$ over all eligible $v \in V$. The following proposition, which is the core of the proof of Theorem~\ref{thm:main}, bounds the minimum of $\| \sigma_\alpha(\GGn^{(j,v)}) \|$ from above. Given an $(r+1)$-uniform hypergraph $\cA$ and a vector $\alpha \in \RR^r$, we define
\[
  \Del_\alpha(\cA) = \sum_{t=1}^r \alpha_t \cdot \Del_{t+1}(\cA).
\]

\begin{prop}
  \label{prop:geometry-hypergraph}
  Suppose that $\cA$ is an $(r+1)$-uniform hypergraph with vertex set $V$ and that $\GGn$ is an $r$-uniform hypergraph with the same vertex set. Suppose that $\alpha \in \RR^r$ has nonnegative coordinates. Then, there exists a vertex $v \in V$ with nonzero degree in $\cA$ such that the hypergraph $\GGnv = \GGn \cup \cA_v$ satisfies
  \begin{equation}
    \label{eq:sigma-GG-norm-change}
    \| \sigma_\alpha(\GGn^v) \|^2 \le \| \sigma_\alpha(\GGn) \|^2 + \frac{\deg_\cA v}{e(\GGnv)} \cdot \left( \left(2 \cdot \frac{\|\sigma_\alpha(\cA)\| }{\|\sigma_\alpha(\GGn)\|}  - 2 + \frac{\Del_1(\cA)}{e(\GGn)}\right) \cdot \| \sigma_\alpha(\GGn) \|^2 + \frac{\Del_{\alpha}(\cA)}{e(\GGn)} \right).
  \end{equation}
\end{prop}

Since the right-hand side of~\eqref{eq:sigma-GG-norm-change} is rather complicated, let us explain the underlying intuition. The two terms $\Del_1(\cA) / e(\GGn)$ and $\Del_\alpha(\cA) / e(\GGn)$ should be viewed as `error terms'. If we assumed that they are both zero, inequality~\eqref{eq:sigma-GG-norm-change} would simplify to
\begin{equation}
  \label{eq:sigma-GG-norm-change-intuition}
  \| \sigma_\alpha(\GGn^v) \|^2 \le \| \sigma_\alpha(\GGn) \|^2 + \frac{2\deg_\cA v}{e(\GGnv)} \cdot \big( \|\sigma_\alpha(\cA)\| - \|\sigma_\alpha(\GGn)\| \big) \cdot \| \sigma_\alpha(\GGn) \|.
\end{equation}
This simplified inequality~\eqref{eq:sigma-GG-norm-change-intuition} states that, as long as the $\ell^2$-norm of $\sigma_\alpha(\GGn)$ exceeds that of $\sigma_\alpha(\cA)$, there is a vertex $v \in V$ such that the $\ell^2$-norm of $\sigma_\alpha(\GGn^v)$ is strictly smaller than that of $\sigma_\alpha(\GGn)$. Moreover, the difference $\| \sigma_\alpha(\GGn) \| - \| \sigma_\alpha(\GGn^v) \|$ is proportional to the difference $\| \sigma_\alpha(\cA) \| - \| \sigma_\alpha(\GGn) \|$. Proposition~\ref{prop:geometry-hypergraph} will allow us to show that, as we repeatedly update $\GGn \leftarrow \GGn^v$ in step~\ref{item:alg-main-step} of the algorithm, the value $\| \sigma_\alpha(\GGn) \|$ drifts, rather quickly, towards $\| \sigma_\alpha(\cA) \|$.

The reason why Proposition~\ref{prop:geometry-hypergraph} is true stems from Fact~\ref{fact:codegree-measure-neighbourhood-measure-alpha}, which states that the vector $\sigma_\alpha(\cA)$ is a convex combination of the vectors $\sigma_\alpha(\cA_v)$, where $v$ ranges over $V$, and the coefficient of each $\sigma_\alpha(\cA_v)$ in this combination is proportional to $\deg_\cA v$. This basic property of the degree measures enables us to express the problem of minimising $\|\sigma_\alpha(\GGn^v)\|$, solved by the proposition, in a simple, abstract way, as we now do in the next lemma.

\begin{lemma}
  \label{lemma:geometry}
  Suppose that $\nu_1, \dotsc, \nu_k \in \RR^d$ and $\lambda \in \RR^k$ all have nonnegative coordinates and $\|\lambda\|_1 =  \lambda_1 + \dotsb + \lambda_k = 1$. Define
  \[
  \nu = \sum_{i=1}^k \lambda_i \cdot \nu_i.
  \]
  For every positive $x$, every $\mu \in \RR^d$ with nonnegative coordinates, and all $x_1, \dotsc, x_k \in (0, x]$, there exists an $i \in \br{k}$ such that $\lambda_i > 0$ and the vector $\mu_i$ defined by
  \[
  \mu_i = (1 - x_i\lambda_i) \cdot \mu + x_i \lambda_i \cdot \nu_i
  \]
  satisfies
  \begin{equation}
    \label{eq:mui-mu-norm-difference}
    \| \mu_i \|^2 \le \| \mu \|^2 + \lambda_i x_i \cdot \left(\left( 2 \cdot \frac{\|\nu\|}{\|\mu\|} - 2 + x \cdot \|\lambda\|^2 \right) \cdot \|\mu\|^2 + x \cdot \sum_{j=1}^k \lambda_j^2 \|\nu_j\|^2 \right).
  \end{equation}
\end{lemma}
\begin{proof}
  Note first that, for every $i \in \br{k}$,
  \begin{equation}
    \label{eq:norm-mu-i}
    \|\mu_i\|^2 = \| \mu + x_i\lambda_i \cdot (\nu_i - \mu) \|^2 = \| \mu \|^2 + 2 x_i\lambda_i \cdot \lan \nu_i - \mu, \mu \ran + x_i^2\lambda_i^2  \cdot \|\nu_i - \mu\|^2.
  \end{equation}
  Since $\lan \nu_i, \mu \ran \ge 0$, by our assumption on non-negativity of the coordinates, we have
  \[
    \|\nu_i - \mu\|^2 = \|\nu_i\|^2 - 2 \lan \nu_i, \mu \ran + \|\mu\|^2 \le \|\mu\|^2 + \|\nu_i\|^2.
  \]
  Substituting this inequality into~\eqref{eq:norm-mu-i}, dividing both sides by $x_i$, and summing over $i \in \br{k}$ yields
  \begin{equation}
    \label{eq:norm-mu-i-sum}
    \sum_{i=1}^k \frac{\|\mu_i\|^2 - \|\mu\|^2}{x_i} \le \sum_{i=1}^k 2\lambda_i \cdot \lan \nu_i - \mu, \mu \ran+ \max_ix_i \cdot \left( \|\lambda\|^2 \|\mu\|^2 + \sum_{i=1}^k \lambda_i^2 \|\nu_i\|^2 \right).    
  \end{equation}
  The definition of $\nu$, the assumption $\| \lambda \|_1 = 1$, and the Cauchy--Schwarz inequality give
  \[
  \sum_{i=1}^k \lambda_i \cdot \lan \nu_i - \mu, \mu \ran = \lan \nu - \mu, \mu \ran = \lan \nu, \mu \ran - \|\mu\|^2  \le  \|\nu\| \cdot \|\mu\| -  \|\mu\|^2.
  \]
  Substituting this inequality into~\eqref{eq:norm-mu-i-sum}, recalling the assumption that $\max_i x_i \le x$, yields
  \begin{equation}
    \label{eq:sum-norm-mu-i-mu}
    \sum_{i=1}^k \frac{\|\mu_i\|^2 - \|\mu\|^2}{x_i} \le 2 \left( \|\nu\| \cdot \| \mu \| - \|\mu\|^2 \right) + x \cdot \left( \|\lambda\|^2 \|\mu\|^2 + \sum_{i=1}^k \lambda_i^2 \|\nu_i\|^2 \right).
  \end{equation}
  Finally, as $\|\lambda\|_1 = 1$, there must exist an $i \in \br{k}$ such that $\lambda_i > 0$ and the $i$th summand in the left-hand side of~\eqref{eq:sum-norm-mu-i-mu} is at most $\lambda_i$ times the right-hand side of~\eqref{eq:sum-norm-mu-i-mu}. This gives
  \[
    \| \mu_i \|^2 - \| \mu \|^2 \le \lambda_i x_i \cdot \left(2 \left( \|\nu\| \cdot \|\mu\| - \|\mu\|^2 \right) + x \left(\|\lambda\|^2 \|\mu\|^2 + \sum_{j=1}^k \lambda_j^2 \|\nu_j\|^2 \right)\right),
  \]
  which is easily seen to be equivalent to the desired inequality~\eqref{eq:mui-mu-norm-difference}.
\end{proof}

\begin{proof}[{Proof of Proposition~\ref{prop:geometry-hypergraph}}]
  For every $v \in V$, let $\GGn^v = \GGn \cup \cA_v$. We claim that, for each $t \in \br{r}$, the $t$-degree measure of $\GGn^v$ is a convex combination of the $t$-degree measures of $\GGn$ and $\cA_v$ and the coefficients in this convex combination are proportional to $e(\GGn)$ and $\deg_\cA v$, respectively. Indeed, since for every $T \subseteq V$, the degree of $T$ in $\GGn^{v}$ is simply the sum of the degrees of $T$ in $\GGn$ and $\cA_v$, we have
  \[
    e(\GGn^v) \cdot \sigma_{\GGn^v}^{(t)} = e(\GGn) \cdot \sigma_{\GGn}^{(t)} + e(\cA_v) \cdot \sigma_{\cA_v}^{(t)} = e(\GGn) \cdot \sigma_{\GGn}^{(t)} + \deg_\cA v \cdot \sigma_{\cA_v}^{(t)}.
  \]
  Dividing the above equality through by $e(\GGn^v) = e(\GGn) + \deg_\cA v$, we obtain
  \[
    \sigma_{\GGn^v}^{(t)} = \left(1 - \frac{\deg_\cA v}{e(\GGn^v)}\right) \cdot \sigma_{\GGn}^{(t)} + \frac{\deg_{\cA}v}{e(\GGn^v)} \cdot \sigma_{\cA_v}^{(t)}.
  \]
  Define, for each $v \in V$,
  \begin{equation}
    \label{eq:xv-def}
    x_v = \frac{(r+1)\cdot e(\cA)}{e(\GGn^v)} \le \frac{(r+1) \cdot e(\cA)}{e(\GGn)}.
  \end{equation}
  Since $\cA$ is $(r+1)$-uniform, we have, for each $v \in V$,
  \[
    x_v \sigma_\cA(v) =  \frac{\deg_{\cA}v}{e(\GGn^v)}
  \]
  and, consequently, for each $t \in \br{r}$,
  \begin{equation}
    \label{eq:GGnv-convex-combination}
    \sigma_{\GGn^v}^{(t)} = \big( 1 - x_v \sigma_\cA(v) \big) \cdot \sigma_{\GGn}^{(t)} + x_v \sigma_\cA(v) \cdot \sigma_{\cA_v}^{(t)}.
  \end{equation}
  
  We now invoke Lemma~\ref{lemma:geometry} with $k = |V|$, the vectors $\nu_1, \dotsc, \nu_k$ replaced by $\{\sigma_\alpha(\cA_v) : v \in V\}$, the vector $\lambda$ replaced by $\sigma_{\cA}$, the vector $\mu$ replaced by $\sigma_\alpha(\GGn)$, the numbers $x_1, \ldots, x_k$ replaced by $\{x_v : v \in V\}$, and $x$ replaced by $(r+1) \cdot e(\cA) / e(\GGn)$; note that $x_1, \dotsc, x_k \le x$, see~\eqref{eq:xv-def}. Fact~\ref{fact:codegree-measure-neighbourhood-measure-alpha} implies that
  \[
    \nu = \sum_{i=1}^k \lambda_i \cdot \nu_i = \sum_{v \in V} \sigma_\cA(v) \cdot \sigma_\alpha(\cA_v) = \sigma_\alpha(\cA),
  \]
  and, if $i \in \br{k}$ corresponds to $v \in V$, then $\lambda_i x_i = x_v \sigma_\cA (v) = \deg_{\cA}v / e(\GGn^v)$ and thus $\mu_i = \sigma_\alpha(\GGn^v)$. Recalling the definition of $\Del_1(\cdot)$ from~\eqref{eq:Del}, we further have
  \[
    x \cdot \| \lambda \|^2 = \frac{(r+1) \cdot e(\cA)}{e(\GGn)} \cdot  \| \sigma_\cA \|^2 = \frac{\Del_1(\cA)}{e(\GGn)}.
  \]
  Finally, Fact~\ref{fact:codegree-measure-norm-square} implies that
  \[
  \begin{split}
    x \cdot \sum_{i=1}^k \lambda_i^2 \| \nu_i \|^2 & = \frac{(r+1) \cdot e(\cA)}{e(\GGn)} \cdot \sum_{v \in V} \sigma_\cA(v)^2 \cdot \|\sigma_\alpha(\cA_v)\|^2 \\
    & = \frac{(r+1) \cdot e(\cA)}{e(\GGn)} \cdot \sum_{t=1}^r \alpha_t \cdot \sum_{v \in V} \sigma_\cA(v)^2 \cdot \| \sigma_{\cA_v}^{(t)} \|^2 \\
    & = \frac{(r+1) \cdot e(\cA)}{e(\GGn)} \cdot \sum_{t=1}^r \alpha_t \cdot \frac{\|\sigma_\cA^{(t+1)}\|^2}{t+1} \\
    & = \frac{1}{e(\GGn)} \cdot \sum_{t=1}^r \alpha_t \cdot \Del_{t+1}(\cA) = \frac{\Del_\alpha(\cA)}{e(\GGn)}.
  \end{split}
  \]
  It is now straightforward to verify that Lemma~\ref{lemma:geometry} implies the existence of a vertex $v \in V$ satisfying the assertion of the proposition.
\end{proof}

\subsection{Proof of the key dichotomy property}
\label{sec:dichotomy}

In this section, we use Proposition~\ref{prop:geometry-hypergraph} to bound the expression from step~\ref{item:alg-best-offer} in the description of our algorithm. This is the most technically demanding part of the proof. Throughout this section, we use the notation introduced in Section~\ref{sec:algorithm}. We start with an easy dichotomy.

\begin{lemma}
  \label{lemma:few-edges-left}
  If $e(\cA) < (1-2\eps) \cdot e(\GG)$, then
  \[
    J \ge \frac{\eps^2}{(r+1)^2} \cdot \| \sigma_\GG \|^{-2}.
  \]
\end{lemma}
\begin{proof}
  Since $\cA$ is obtained from $\cA^{(0)}$ by removing all edges that contain at least one of the vertices $v_0, \dotsc, v_{J-1}$, we have
  \[
    e(\cA^{(0)}) - e(\cA) \le \sum_{j=0}^{J-1} \deg_{\cA^{(0)}} v_J \le J \cdot \Delta_1(\cA^{(0)}).
  \]
  Consequently, it follows from~\eqref{eq:A0-properties} and our upper bound on $e(\cA)$ that
  \[
    J \ge \frac{e(\cA^{(0)}) - e(\cA)}{\Delta_1(\cA^{(0)})} \ge \frac{\eps^2 \cdot e(\GG)}{(r+1) \cdot \Del_1(\GG)} = \frac{\eps^2}{(r+1)^2} \cdot \| \sigma_\GG \|^{-2}.\qedhere
  \]
\end{proof}

\begin{lemma}
  \label{lemma:many-edges-left}
  If $e(\cA) \ge (1-2\eps) \cdot e(\GG)$, then $e(\GGn) \ge p \cdot e(\GG)$.
\end{lemma}
\begin{proof}
  By construction,
  \[
    e(\GGn) = e(\GGn^{(J)}) - e(\GGn^{(0)}) = \sum_{j=0}^{J-1} \left( e(\GGn^{(j+1)}) - e(\GGn^{(j)})\right) = \sum_{j \in L} \deg_{\gcA^{(j)}} v_j \ge \sum_{j \in L} \frac{\eps \cdot e(\cA^{(j)})}{|V|}.
  \]
  Our assumption $e(\cA) \ge (1-2\eps) \cdot e(\GG)$ implies that the algorithm terminated with $|L| = b \ge 2p/\eps \cdot |V|$, see step~\ref{item:alg-stop}. Consequently, recalling that $\eps < 1/4$,
  \[
    e(\GGn) \ge |L| \cdot \frac{\eps \cdot e(\cA)}{|V|} \ge \frac{2p \cdot |V|}{\eps} \cdot \frac{\eps \cdot (1-2\eps) \cdot e(\GG)}{|V|} \ge p \cdot e(\GG).\qedhere
  \]
\end{proof}

Our next lemma lies at the heart of the matter. For brevity, define
\[
  \Del := 3 \cdot \Del_\alpha(\GG).
\]

\begin{lemma}
  \label{lemma:main-estimate}
  For every $j \in \{0, \dotsc, J\}$,
  \begin{equation}
    \label{eq:main-estimate}
    \| \sigma_\alpha(\GGn^{(j)}) \|^2 \le (1+\eps) \cdot \sigma^2 + \frac{a \cdot \Del}{e(\GGn^{(j)})}.
  \end{equation}
\end{lemma}
\begin{proof}
  We prove~\eqref{eq:main-estimate} by induction on $j$. Since $\GGn^{(0)}$ is (an integer multiple of) the complete $r$-uniform hypergraph, $\sigma_{\GGn^{(0)}}^{(t)}$ is the uniform probability measure on $\binom{V}{t}$ and, consequently,
  \[
    \| \sigma_{\GGn^{(0)}}^{(t)} \|^2 = \binom{|V|}{t}^{-1}
  \]
  for every $t \in \br{r}$. On the other hand, Fact~\ref{fact:codegree-measure-norm-square} implies that, for every $t \in \br{r}$,
  \[
    \Del_{t+1}(\GG) = \frac{r+1}{t+1} \cdot e(\GG) \cdot \| \sigma_{\GG}^{(t+1)} \|^2 = (r+1) \cdot e(\GG) \cdot \sum_{v \in V} \sigma_\GG(v)^2 \cdot \| \sigma_{\GG_v}^{(t)} \|^2.
  \]
  Fact~\ref{fact:codegree-measure-norm-maximum-degree} implies that, for every $v \in V$ such that $\GG_v$ is nonempty,
  \[
    \| \sigma_{\GG_v}^{(t)} \|^2 \ge \binom{|V|}{t}^{-1} = \| \sigma_{\GGn^{(0)}}^{(t)} \|^2
  \]
  and therefore
  \[
    \Del_{t+1}(\GG) \ge (r+1) \cdot e(\GG) \cdot \sum_{v \in V} \sigma_\GG(v)^2 \cdot \| \sigma_{\GGn^{(0)}}^{(t)} \|^2 = \Del_1(\GG) \cdot \| \sigma_{\GGn^{(0)}}^{(t)} \|^2.
  \]
  Recall from~\eqref{eq:m-choice} that we have chosen $m$ so that
  \[
    \Del_1(\GG) \ge \frac{m}{2a} \cdot \binom{|V|}{r} = \frac{e(\GGn^{(0)})}{2a},
  \]
  which, substituted into the previous inequality, implies that
  \begin{equation}
    \label{eq:sigma-GGn-t-basis}
    \| \sigma_{\GGn^{(0)}}^{(t)} \|^2 \le \frac{2a}{e(\GGn^{(0)})} \cdot \Del_{t+1}(\GG).
  \end{equation}
  After we multiply both sides of~\eqref{eq:sigma-GGn-t-basis} by $\alpha_t$ and sum the resulting inequalities over all $t \in \br{r}$, we obtain
  \[
    \| \sigma_{\alpha}(\GGn^{(0)}) \|^2 \le 2a \cdot \frac{\Del_\alpha(\GG)}{e(\GGn^{(0)})},
  \]
  which implies~\eqref{eq:main-estimate} when $j=0$.

  Suppose now that $j \ge 0$ and assume that~\eqref{eq:main-estimate} holds; we shall show that this inequality remains true after we replace $j$ with $j+1$. We may assume that $\GGn^{(j+1)} = \GGn^{(j)} \cup \gcA_{v_j}^{(j)}$, as otherwise $\GGn^{(j+1)} = \GGn^{(j)}$ and there is nothing to prove. Let $v$ be a vertex satisfying the assertion of Proposition~\ref{prop:geometry-hypergraph} with $\cA \leftarrow \gcA^{(j)}$ and $\GGn \leftarrow \GGn^{(j)}$. The vertex $v_j$ was chosen in step~\ref{item:alg-best-offer} so that
  \begin{equation}
    \label{eq:main-estimate-reduced}
    e(\GGn^{(j+1)}) \cdot \left( \| \sigma_\alpha(\GGn^{(j+1)}) \|^2 - (1+\eps) \cdot \sigma^2 \right) \le e(\GGn^{(j,v)}) \cdot \left( \| \sigma_\alpha(\GGn^{(j,v)}) \|^2 - (1+\eps) \cdot \sigma^2 \right),
  \end{equation}
  so it suffices to bound the right-hand side of~\eqref{eq:main-estimate-reduced} from above by $a \cdot \Del$.

  The assertion of Proposition~\ref{prop:geometry-hypergraph}, inequality~\eqref{eq:sigma-GG-norm-change}, is equivalent to the inequality
  \begin{multline}
    \label{eq:sigma-GG-error-change}
    e(\GGn^{(j,v)}) \cdot \left( \| \sigma_\alpha(\GGn^{(j,v)}) \|^2 - \sigma^2 \right) \le e(\GGn^{(j,v)}) \cdot \left( \| \sigma_\alpha(\GGn^{(j)}) \|^2 - \sigma^2 \right) \\
    + \deg_{\gcA^{(j)}} v \cdot \left( \left(\frac{\Del_1(\gcA^{(j)})}{e(\GGn^{(j)})} + 2 \cdot \frac{\|\sigma_\alpha(\gcA^{(j)})\| }{\|\sigma_\alpha(\GGn^{(j)})\|}  - 2\right) \cdot \| \sigma_\alpha(\GGn^{(j)}) \|^2 + \frac{\Del_{\alpha}(\gcA^{(j)})}{e(\GGn^{(j)})} \right).
  \end{multline}
  Since $e(\GGn^{(j,v)}) = e(\GGn^{(j)}) + \deg_{\gcA^{(j)}} v$, inequality~\eqref{eq:sigma-GG-error-change} may be rewritten as
  \begin{multline}
    \label{eq:sigma-GG-error-change-rewritten}
    e(\GGn^{(j,v)}) \cdot \left( \| \sigma_\alpha(\GGn^{(j,v)}) \|^2 - \sigma^2 \right) \le e(\GGn^{(j)}) \cdot \left( \| \sigma_\alpha(\GGn^{(j)}) \|^2 - \sigma^2 \right) \\
    + \deg_{\gcA^{(j)}} v \cdot \left( \left(\frac{\Del_1(\gcA^{(j)})}{e(\GGn^{(j)})} + 2 \cdot \frac{\|\sigma_\alpha(\gcA^{(j)})\| }{\|\sigma_\alpha(\GGn^{(j)})\|}  - 1\right) \cdot \| \sigma_\alpha(\GGn^{(j)}) \|^2 - \sigma^2 + \frac{\Del_{\alpha}(\gcA^{(j)})}{e(\GGn^{(j)})} \right).
  \end{multline}

  We shall now simplify the right-hand side of~\eqref{eq:sigma-GG-error-change-rewritten} somewhat. To this end, observe first that, as the algorithm did not terminate in step~\ref{item:alg-stop}, we must have
  \[
    e(\gcA^{(j)}) \ge (1-\eps) \cdot e(\cA^{(j)}) \ge (1-\eps) \cdot (1-2\eps) \cdot e(\GG) \ge (1-3\eps) \cdot e(\GG).
  \]
  Consequently, Fact~\ref{fact:codegree-measure-subhypergraph} implies that, for every $t \in \br{r+1}$,
  \begin{equation}
    \label{eq:sigma-cAj-upper}
    \| \sigma_{\gcA^{(j)}}^{(t)} \| \le (1-3\eps)^{-1} \cdot \| \sigma_{\GG}^{(t)}\|
  \end{equation}
  and, since clearly $e(\gcA^{(j)}) \le e(\GG)$,
  \begin{equation}
    \label{eq:Delta-cAj-upper}
    \Del_t(\gcA^{(j)}) \le (1-3\eps)^{-2} \cdot \Del_t(\GG) \le 3 \cdot \Del_t(\GG).
  \end{equation}
  Summing~\eqref{eq:sigma-cAj-upper}, with both sides squared, and~\eqref{eq:Delta-cAj-upper} over all $t$, with appropriate weights, yields
  \begin{equation}
    \label{eq:sigma-GG-error-change-error-1}
    \| \sigma_\alpha(\gcA^{(j)}) \| \le (1-3\eps)^{-1} \cdot \| \sigma_\alpha(\GG) \| = \sigma \qquad \text{and} \qquad \Del_\alpha(\gcA^{(j)}) \le \Del.
  \end{equation}
  Furthermore, recall from~\eqref{eq:m-choice} that we have chosen $m$ so that
  \[
    \Del_1(\GG) \le \frac{m}{a} \cdot \binom{|V|}{r} = \frac{e(\GGn^{(0)})}{a} \le \frac{e(\GGn^{(j)})}{a}.
  \]
  Consequently
  \begin{equation}
    \label{eq:sigma-GG-error-change-error-2}
    \frac{\Del_1(\gcA^{(j)})}{e(\GGn^{(j)})} \le \frac{(1-3\eps)^{-2}}{a} \le \frac{3}{a}
  \end{equation}
  and, by~\eqref{eq:A0-properties},
  \begin{equation}
    \label{eq:Delta1-cAj-upper}
    \begin{split}
      \Delta_1(\gcA^{(j)}) & \le \Delta_1(\cA^{(j)}) \le \Delta_1(\cA^{(0)}) \le \frac{r+1}{\eps} \cdot \Del_1(\GG) \\
      & \le \frac{r+1}{\eps} \cdot \frac{e(\GGn^{(j)})}{a} \le \frac{a}{2} \cdot e(\GGn^{(j)}),
    \end{split}
  \end{equation}
  where the last inequality holds as $a^2 = 625/\eps^4 \ge 2(r+1)/\eps$. We may now substitute~\eqref{eq:sigma-GG-error-change-error-1} and~\eqref{eq:sigma-GG-error-change-error-2} into~\eqref{eq:sigma-GG-error-change-rewritten} and rearrange the terms to obtain the following inequality:
  \begin{multline}
    \label{eq:sigma-GG-error-change-final}
    e(\GGn^{(j,v)}) \cdot \left( \| \sigma_\alpha(\GGn^{(j,v)}) \|^2 - \sigma^2 \right) \le e(\GGn^{(j)}) \cdot \left( \| \sigma_\alpha(\GGn^{(j)}) \|^2 - \sigma^2 \right) \\
    + \deg_{\gcA^{(j)}} v \cdot \left( \left(\frac{3}{a} - \left(\frac{ \sigma  }{\|\sigma_\alpha(\GGn^{(j)})\|} - 1\right)^2\right) \cdot \| \sigma_\alpha(\GGn^{(j)}) \|^2 + \frac{\Del}{e(\GGn^{(j)})} \right).
  \end{multline}

  We now consider two cases, depending on how large $\| \sigma_\alpha(\GGn^{(j)}) \|^2$ is.

  \medskip
  \noindent
  \textbf{Case 1. $\| \sigma_\alpha(\GGn^{(j)}) \|^2 \le (1+\eps) \cdot \sigma^2 + a/2 \cdot \Del / e(\GGn^{(j)})$.}
  \medskip

  We first claim that
  \begin{equation}
    \label{eq:main-term-case-1}
    \left(\frac{3}{a} - \left(\frac{ \sigma  }{\|\sigma_\alpha(\GGn^{(j)})\|} - 1\right)^2\right) \cdot \| \sigma_\alpha(\GGn^{(j)}) \|^2 \le \frac{12\sigma^2}{a} \le \eps \cdot \sigma^2.
  \end{equation}
  To see this, note that the left-hand side of~\eqref{eq:main-term-case-1} is negative when $\| \sigma_\alpha(\GGn^{(j)}) \| > 2\sigma$, as $a \ge 12$. Otherwise, the first factor in the left-hand side is at most $3/a$ and the second factor is at most $4\sigma^2$. Substituting~\eqref{eq:main-term-case-1} into~\eqref{eq:sigma-GG-error-change-final}, using the assumed upper bound on $\| \sigma_\alpha(\GGn^{(j)}) \|^2$, yields
  \[
    \begin{split}
      e(\GGn^{(j,v)}) \cdot \left( \| \sigma_\alpha(\GGn^{(j,v)}) \|^2 - \sigma^2 \right) & \le e(\GGn^{(j)}) \cdot \eps \cdot \sigma^2 + \frac{a}{2} \cdot \Del + \deg_{\gcA^{(j)}} v \cdot \left(\eps \cdot \sigma^2 + \frac{\Del}{e(\GGn^{(j)})} \right) \\
      & \le \left( e(\GGn^{(j)}) + \deg_{\gcA^{(j)}} v\right) \cdot \eps \cdot \sigma^2 + \left(\frac{a}{2}+\frac{a}{2}\right) \cdot \Del,
    \end{split}
  \]
  where the second inequality holds because $\deg_{\gcA^{(j)}} v \le \Delta_1(\gcA^{(j)}) \le (a/2) \cdot e(\GGn^{(j)})$, see~\eqref{eq:Delta1-cAj-upper}. Finally, since $e(\GGn^{(j)}) + \deg_{\gcA^{(j)}} v =  e(\GGn^{(j,v)})$, we may conclude that
  \begin{equation}
    \label{eq:case-1-conclusion}
    e(\GGn^{(j,v)}) \cdot \left(\| \sigma_\alpha(\GGn^{(j,v)}) \|^2 -  (1+\eps) \cdot \sigma^2 \right)  \le a \cdot \Del.
  \end{equation}
  By~\eqref{eq:main-estimate-reduced}, this proves the desired estimate (inequality~\eqref{eq:main-estimate} with $j$ replaced by $j+1$).
  
  \medskip
  \noindent
  \textbf{Case 2. $(1+\eps) \cdot \sigma^2 + a/2 \cdot \Del / e(\GGn^{(j)}) < \| \sigma_\alpha(\GGn^{(j)}) \|^2 \le (1+\eps) \cdot \sigma^2 + a \cdot \Del / e(\GGn^{(j)})$.}
  \medskip

  We will show that the second term in the right-hand side of~\eqref{eq:sigma-GG-error-change-final} is nonpositive, which will give
  \[
    \begin{split}
      e(\GGn^{(j,v)}) \cdot \| \sigma_\alpha(\GGn^{(j,v)}) \|^2 & \le e(\GGn^{(j)}) \cdot \left( \| \sigma_\alpha(\GGn^{(j)}) \|^2 - \sigma^2 \right) + e(\GGn^{(j,v)}) \cdot \sigma^2 \\
      & \le e(\GGn^{(j)}) \cdot \eps \cdot \sigma^2 + a \cdot \Del + e(\GGn^{(j,v)}) \cdot \sigma^2,
    \end{split}
  \]
  which in turn implies the desired estimate (inequality~\eqref{eq:case-1-conclusion}), as $e(\GGn^{(j)}) \le e(\GGn^{(j,v)})$. The lower bound on $\| \sigma_{\alpha} (\GGn^{(j)})\|^2$ assumed in Case~2 implies that
  \[
    \frac{\Del}{e(\GGn^{(j)})} < \frac{2}{a} \cdot \| \sigma_\alpha(\GGn^{(j)}) \|^2
  \]
  and hence
  \[
    \left(\frac{3}{a} - \left(\frac{ \sigma  }{\|\sigma_\alpha(\GGn^{(j)})\|} - 1\right)^2\right) \cdot \| \sigma_\alpha(\GGn^{(j)}) \|^2 + \frac{\Del}{e(\GGn^{(j)})} \le \left(\frac{5}{a}  - \left(\frac{ \sigma  }{\|\sigma_\alpha(\GGn^{(j)})\|} - 1\right)^2\right) \cdot \| \sigma_\alpha(\GGn^{(j)}) \|^2.
  \]
  Moreover, since $\| \sigma_\alpha(\GGn^{(j)}) \|^2 \ge (1+\eps) \cdot \sigma^2$, we have
  \[
  \begin{split}
    \frac{5}{a}  - \left(\frac{\sigma}{\|\sigma_\alpha(\GGn^{(j)})\|}-1\right)^2 &
    \le \frac{5}{a} - \left(\frac{1}{\sqrt{1+\eps}}-1\right)^2
    = \frac{5}{a} - \frac{\left(\sqrt{1+\eps} - 1\right)^2}{1+\eps} \\
    & = \frac{\eps^2}{5} - \frac{\eps^2}{\left(1 + \sqrt{1+\eps}\right)^2 \cdot (1+\eps)} \le 0,
  \end{split}
  \]
  where the last inequality holds due to our assumption that $\eps < (9(r+1))^{-1} \le 1/18$.
\end{proof}

\subsection{Proof of the key lemma}
\label{sec:proof-key-lemma}

After having made all the preparations, we are finally ready to prove Lemma~\ref{lemma:analysis}.

\begin{proof}[{Proof of Lemma~\ref{lemma:analysis}}]
  If $e(\cA) < (1-2\eps) \cdot e(\GG)$, then Lemma~\ref{lemma:few-edges-left} immediately gives~\ref{item:many-skipped}. We may therefore assume that $e(\cA) \ge (1-2\eps) \cdot e(\GG)$. Recall from~\eqref{eq:m-choice} that we have chosen $m$ so that
  \[
    e(\GGn^{(0)}) = m \cdot \binom{|V|}{r} \le 2a \cdot \Del_1(\GG) = 2a \cdot (r+1) \cdot e(\GG) \cdot \| \sigma_\GG \|^2.
  \]
  It thus follows from~\eqref{eq:sigma-GG-assumption} and Lemma~\ref{lemma:many-edges-left} that
  \[
    e(\GGn^{(0)}) \le 2a \cdot (r+1) \cdot \frac{e(\GGn)}{p} \cdot \frac{\eps p}{2a(r+1)} \le \eps \cdot e(\GGn).
  \]
  Consequently, Fact~\ref{fact:codegree-measure-subhypergraph} implies that
  \[
    \| \sigma_\alpha(\GGn) \| \le \frac{e(\GGn^{(J)})}{e(\GGn)} \cdot \| \sigma_\alpha(\GGn^{(J)}) \| = \frac{e(\GGn) + e(\GGn^{(0)})}{e(\GGn)} \cdot \| \sigma_\alpha (\GGn^{(J)})\| \le (1+\eps) \cdot \| \sigma_\alpha(\GGn^{(J)}) \|.
  \]
  It now follows from Lemma~\ref{lemma:main-estimate} that
  \[
    \| \sigma_\alpha(\GGn) \|^2 \le (1+\eps)^3 \cdot \sigma^2 + (1+\eps)^2 \cdot \frac{a \cdot \Del}{e(\GGn^{(J)})}.
  \]
  Since $e(\GGn^{(J)}) \ge e(\GGn) \ge p \cdot e(\GG)$, by Lemma~\ref{lemma:many-edges-left} and $(1+\eps)^7(1-3\eps)^2 \ge 1$, as $\eps < (9(r+1))^{-1} \le 1/18$, we further have
  \[
    \begin{split}
      \| \sigma_\alpha(\GGn) \|^2 & \le (1+\eps)^3 \cdot \sigma^2 + (1+\eps)^2 \cdot \frac{a \cdot \Del}{p \cdot e(\GG)} \\
      & \le \frac{(1+\eps)^3}{(1-3\eps)^2} \cdot \|\sigma_\alpha(\GG)\|^2 + (1+\eps)^2 \cdot \frac{3a \cdot \Del_\alpha(\GG)}{p \cdot e(\GG)} \\
      & \le (1+\eps)^{10} \cdot \| \sigma_\alpha(\GG) \|^2 + \frac{4a \cdot \Del_\alpha(\GG)}{p \cdot e(\GG)} \\
      & = \sum_{t=1}^{r} \alpha_t \cdot \left( (1+\eps)^{10} \cdot \| \sigma_\GG^{(t)} \|^2 + \frac{4a}{p} \cdot \frac{r+1}{t+1} \cdot \|\sigma_\GG^{(t+1)} \|^2  \right).
    \end{split}
  \]
  Recall the definition of $\alphan$ given in~\eqref{eq:def-alphan-b}. Since $4a/(t+1) \le 50/\eps^2$ for every $t \in \br{r}$, we may conclude that
  \[
    \| \sigma_\alpha(\GGn) \|^2 \le \| \sigma_{\alphan}(\GG) \|^2,
  \]
  as claimed.
\end{proof}

\section{Probabilistic inequalities}
\label{sec:prob-ineq}

The proofs of Theorems~\ref{thm:clique-free}, \ref{thm:eps-nets}, and~\ref{thm:Folkman} make use of well-known probabilistic inequalities. The first of them are standard tail bounds for binomial distributions.

\begin{lemma}
  \label{lemma:binomial-tails}
  Let $n$ be a positive integer, let $p \in [0,1]$, and suppose that $X \sim \Bin(n,p)$. Then, for every $\delta \in [0,1]$,
  \[
    \Pr\big(X \le (1-\delta)np\big) \le \exp(-\delta^2np/2)
    \quad
    \text{and}
    \quad
    \Pr\big(X \ge (1+\delta)np\big) \le \exp(-\delta^2np/3).
  \]
\end{lemma}

We shall also need the following version of Janson's inequality~\cite{Jan90}, which can be easily deduced from the statements found in~\cite[Chapter~8]{AloSpe16}.

\begin{thm}[Janson's inequality]
  \label{thm:Janson}
  Suppose that $\Omega$ is a finite set and let $B_1, \dotsc, B_k$ be arbitrary subsets of $\Omega$. Form a random subset $R \subseteq \Omega$ by independently keeping each $\omega \in \Omega$ with probability $p_\omega \in [0,1]$. For each $i \in \br{k}$, let $X_i$ be the indicator of the event that $B_i \subseteq R$ and define
  \[
    \mu = \sum_{i=1}^k \Ex[X_i] \qquad \text{and} \qquad \Delta = \sum_{\substack{1 \le i < j \le k \\ B_i \cap B_j \neq \emptyset}} \Ex[X_i X_j].
  \]
  Then,
  \[
    \Pr\big( \text{$B_i \nsubseteq R$ for all $i \in \br{k}$} \big) \le \exp \left(-\min\left\{\frac{\mu}{2}, \frac{\mu^2}{4\Delta}\right\}\right).
  \]
\end{thm}

\section{The typical structure of $K_{r+1}$-free graphs}
\label{sec:typical-structure-clique-free}

\subsection{Outline}
\label{sec:outline-clique-free}

The first, key part of the proof of Theorem~\ref{thm:clique-free} is showing that, for sufficiently small $\delta$, the number of $K_{r+1}$-free subgraphs of $K_n$ that are not $\delta n^2$-close to being $r$-partite is much smaller than $2^{\ex(n,K_{r+1})}$, which is a trivial lower bound on the number of $K_{r+1}$-free graphs. This statement is derived from a container lemma for $K_{r+1}$-free graphs (Proposition~\ref{prop:containers-clique-free} below), which is obtained by applying Theorem~\ref{thm:main-packaged} to the $\binom{r+1}{2}$-uniform hypergraph that encodes copies of $K_{r+1}$ in $K_n$, and the `supersaturated' version of the stability theorem of Erd\H{o}s and Simonovits proved in~\cite{BalBusColLiuMorSha17} and stated as Lemma~\ref{lemma:supersaturated-stability-cliques} below. Proposition~\ref{prop:containers-clique-free}, which is the main result of this section, supplies a covering of all $K_{r+1}$-free subgraphs of $K_n$ with few containers, each of which is a subgraph of $K_n$ with either fewer than $n^2/8$ edges or fewer than $n^{r+1/2}$ copies of $K_{r+1}$ (after we delete from it some $n^{2-1/(8r)}$ edges), whereas Lemma~\ref{lemma:supersaturated-stability-cliques} is used to show that all containers with nearly $\ex(n,K_{r+1})$ edges must be close to being $r$-partite.

The remainder of the proof is showing that all but an $2^{-n/(10r)^4}$-proportion of $K_{r+1}$-free subgraphs of $K_n$ that are $\delta n^2$-close to being $r$-partite are in fact $r$-partite. Our three-step argument is loosely based on the methods of~\cite{BalMorSamWar16}. First, we show that all but a tiny fraction of graphs in our collection admit an optimal, balanced $r$-partition with at most $\delta n^2$ monochromatic edges (i.e., edges whose both endpoints belong to the same part of the partition); an $r$-partition is optimal if it minimises the number of monochromatic edges and it is balanced if each partite set comprises at least $n/(2r)$ vertices. Second, we bound from above the number of remaining graphs whose associated $r$-partition induces a monochromatic copy of $K_{1,D}$ in one of the parts, where $D =  \lfloor n/(2^{14} r^5\log n) \rfloor$. Third, we bound from above the number of remaining graphs whose associated $r$-partition induces a monochromatic matching with a given number of edges in one of the parts. The second and third steps complement one another as every graph with $t$ edges contains either a copy of $K_{1,D}$ or a matching with at least $t/D$ edges.

\subsection{An efficient container lemma for $K_{r+1}$-free graphs}
\label{sec:container-lemma-clique-free}

The following statement, which is the main technical result of this section, is an efficient container lemma for $K_{r+1}$-free subgraphs of~$K_n$. It is obtained by applying Theorem~\ref{thm:main-packaged} to the $\binom{r+1}{2}$-uniform hypergraph that encodes copies of $K_{r+1}$ in $K_n$.

\begin{prop}
  \label{prop:containers-clique-free}
  For almost all $n$ and every $r$ satisfying $2 \le r \le \log n / (121 \log \log n)$, there exists a collection $\GG$ of at most $\exp\left(n^{2-1/(8r)}\right)$ subgraphs of $K_n$ such that:
  \begin{enumerate}[label=(\roman*)]
  \item
    Each $K_{r+1}$-free subgraph of $K_n$ is contained in some member of $\GG$.
  \item
    Each $G \in \GG$ either has fewer than $n^2/8$ edges or it contains a subgraph $G'$ with $e(G') \ge e(G) - n^{2-1/(8r)}$ that has fewer than $n^{r+1/2}$ copies of $K_{r+1}$.
  \end{enumerate}
\end{prop}

\begin{proof}
  Let $n$ be a large integer and suppose that $r$ satisfies $2 \le r \le \log n / (121 \log \log n)$. Let $\gamma = 1/(8r)$ and observe that
  \begin{equation}
    \label{eq:n-gamma-bounds}
    n^{\gamma} = \exp\left(\frac{\log n}{8r}\right) \ge \exp(12 \log\log n) = (\log n)^{12}.
  \end{equation}
  Let $\HH$ be the $\binom{r+1}{2}$-uniform hypergraph with vertex set $E(K_n)$ whose edges are the edge sets of all copies of $K_{r+1}$ in $K_n$ and set
  \[
    s = \binom{r+1}{2}, \qquad q = n^{-3\gamma}, \qquad \text{and} \qquad E = n^{r+1/2}.
  \]
  We now verify that we may apply Theorem~\ref{thm:main-packaged}, with $\alpha \leftarrow 1/4$ and $\beta \leftarrow n^{-\gamma}$, to the hypergraph $\HH$. First, as $s \le r^2 \le (\log n)^2$, we have
  \[
    \alpha \beta q \cdot v(\HH) = \frac{n^{-4\gamma}}{4} \cdot \binom{n}{2} \ge n \ge 10^9 s^7
  \]
  and, using~\eqref{eq:n-gamma-bounds},
  \[
    10^4s^5q = 10^4s^5n^{-3\gamma} \le 10^4 (\log n)^{10} n^{-3\gamma} \le n^{-\gamma} = \beta,
  \]
  provided that $n$ is sufficiently large. Second, suppose that $t \in \{2, \dotsc, s\}$ and let $\ell \in \{3, \dotsc, r+1\}$ be the unique integer satisfying $\binom{\ell-1}{2} < t \le \binom{\ell}{2}$, so that
  \[
    t-1 \le \binom{\ell}{2}-1 = \frac{(\ell+1)(\ell-2)}{2} \le r(\ell-2).
  \]
  Since a graph with $t$ edges must have at least $\ell$ vertices, we have
  \[
    \Delta_t(\HH) = \binom{n - \ell}{r+1-\ell} \le n^{r+1-\ell}
  \]
  and, consequently,
  \begin{align*}
    \left(\frac{q}{10^6s^5}\right)^{t-1} \cdot \frac{E}{v(\HH)} \cdot \frac{1}{\Delta_t(\HH)}
    & \ge \left(\frac{q}{10^6s^5}\right)^{r(\ell-2)} \cdot \frac{n^{r+1/2}}{n^2} \cdot \frac{1}{n^{r+1-\ell}} \\
    & = \left(\frac{n^{-3\gamma}}{10^6 s^5}\right)^{r(\ell-2)} \cdot n^{\ell-5/2} \ge \left(\frac{n^{1/(2r)-3\gamma}}{10^6s^5}\right)^{r(\ell-2)} \\
    & \ge \left(\frac{n^{\gamma}}{10^6(\log n)^{10}}\right)^{r(\ell-2)} \ge 1,
  \end{align*}
  where the last inequality follows from~\eqref{eq:n-gamma-bounds}. Theorem~\ref{thm:main-packaged} supplies a collection $\GG$ of containers for the independent sets of $\HH$ (that is, $K_{r+1}$-free subgraphs of $K_n$) satisfying
  \[
    \begin{split}
      |\GG| & \le \exp\left(10^4s^5\beta^{-1}\log(e/\alpha) \cdot q\log(e/q) \cdot v(\HH)\right) \\
      & \le \exp\left(10^5(\log n)^{10}n^\gamma \cdot n^{-3\gamma} \log n \cdot n^2\right) \le \exp\left(n^{2-\gamma}\right),
    \end{split}
  \]
  where the last inequality follows from~\eqref{eq:n-gamma-bounds}, such that, for every $G \in \GG$, either $e(G) \le \alpha \cdot v(\HH) \le n^2/8$ or there is a subgraph $G' \subseteq G$ with $e(G') \ge (1-\beta)e(G) \ge e(G) - n^{2-\gamma}$ and $e(\HH[G']) < E$ (that is, $G'$ contains fewer than $E$ copies of $K_{r+1}$).
\end{proof}

\subsection{Almost all $K_{r+1}$-free graphs are almost $r$-partite}
\label{sec:almost-all-almost-partite}

The following theorem, which is a rather straightforward consequence of our container lemma for $K_{r+1}$-free graphs (Proposition~\ref{prop:containers-clique-free}) and the `supersaturated' version of the stability theorem of Erd\H{o}s and Simonovits (Lemma~\ref{lemma:supersaturated-stability-cliques} below) proved by Balogh, Bushaw, Collares, Liu, Morris, and Sharifzadeh~\cite{BalBusColLiuMorSha17}, may be viewed as an approximate version of Theorem~\ref{thm:clique-free}. It states that, under the assumptions of Theorem~\ref{thm:clique-free}, almost all $K_{r+1}$-free subgraphs of $K_n$ are almost $r$-partite. To make this notion precise, given nonnegative integers $r$ and $t$ with $r \ge 2$, we shall say that a graph $G$ is \emph{$t$-close to being $r$-partite} if $G$ can be made $r$-partite by removing from it at most $t$ edges. In other words, $G$ is $t$-close to being $r$-partite if $G$ contains an $r$-partite subgraph $G'$ with $e(G') \ge e(G) - t$. Conversely, we shall say that a graph $G$ is \emph{$t$-far from being $r$-partite} if $G$ is not $t$-close to being $r$-partite or, in other words, if $\chi(G') > r$ for every $G' \subseteq G$ with $e(G') \ge e(G) - t$.

\begin{thm}
  \label{thm:almost-all-almost-partite}
  The following holds for sufficiently large $n$ and all $r$ satisfying $2 \le r \le \log n / (121 \log \log n)$. Let $\FF$ denote the family of $K_{r+1}$-free subgraphs of $K_n$ that are $(8\log n)^{-15}n^2$-far from being $r$-partite. Then
  \[
    |\FF| \le 2^{\ex(n, K_{r+1}) - n}.
  \]
\end{thm}

\begin{lemma}[\cite{BalBusColLiuMorSha17}]
  \label{lemma:supersaturated-stability-cliques}
  Suppose that $n$, $r$, and $t$ are positive integers. Every $n$-vertex graph $G$ that is $t$-far from being $r$-partite contains at least 
  \[
    \frac{n^{r-1}}{e^{2r}\cdot r!} \left( e(G)+t-\left(1-\frac{1}{r}\right)\frac{n^2}{2}\right)
  \]
  copies of $K_{r+1}$.
\end{lemma}

\begin{proof}[Proof of Theorem~\ref{thm:almost-all-almost-partite}]
  Let $\delta = (8\log n)^{-15}$ so that $\FF$ is the family of $K_{r+1}$-free subgraphs of $K_n$ that are $\delta n^2$-far from being $r$-partite. Let $\GG$ be the family of containers for $K_{r+1}$-free graphs supplied by Proposition~\ref{prop:containers-clique-free}. We partition $\GG$ into two parts as follows:
  \[
    \GG_1 = \left\{G \in \GG : e(G) \ge \left(1-\frac{1}{r}\right)\frac{n^2}{2} - 2n^{2-1/(8r)}\right\} \qquad \text{and} \qquad \GG_2 = \GG \setminus \GG_1.
  \]
  Fix an arbitrary $G \in \GG_1$. Since $e(G) > n^2/8$, it must be the case that $G$ contains a subgraph $G'$ with
  \begin{equation}
    \label{eq:eG'-lower}
    e(G') \ge e(G) - n^{2-1/(8r)} \ge \left(1-\frac{1}{r}\right) \frac{n^2}{2} - 3n^{2-1/(8r)}
  \end{equation}
  that contains fewer than $n^{r+1/2}$ copies of $K_{r+1}$. Let $G'$ be any such subgraph of $G$ and let $t'$ be the smallest number of edges one can delete from $G'$ to make it $r$-partite. Lemma~\ref{lemma:supersaturated-stability-cliques} implies that
  \[
    \frac{n^{r-1}}{e^{2r}\cdot r!} \left( e(G')+t'-1-\left(1-\frac{1}{r}\right)\frac{n^2}{2}\right) < n^{r+1/2}
  \]
  and hence, by~\eqref{eq:eG'-lower},
  \[
    t' - 3n^{2-1/(8r)} \le e^{2r} \cdot r! \cdot n^{3/2} +1 \le r^{4r} \cdot n^{3/2} \le n^{7/4}.
  \]
  Since $G'$ is $t'$-close to being $r$-partite, $G$ is $t$-close from being $r$-partite for all $t \ge 5n^{2-1/(8r)}$, since
  \[
    5n^{2-1/(8r)} > n^{7/4} + 4n^{2-1/(8r)} \ge t' + e(G) - e(G').
  \]
  In particular, since
  \[
    n^{-1/(8r)} = \exp\left(-\frac{\log n}{8r}\right) = \exp\left(-\frac{121}{8} \cdot \log \log n\right) \ll \frac{1}{(\log n)^{15}},
  \]
  neither $G$ nor any of its subgraphs can be $\delta n^2$-far from being $r$-partite.

  Thus, every graph in $\FF$ must be contained in some element of $\GG_2$. Since
  \[
    \ex(n, K_{r+1}) \ge \left(1 - \frac{1}{r}\right) \binom{n}{2} \ge \left(1-\frac{1}{r}\right)\frac{n^2}{2} - n,
  \]
  we have
  \begin{align*}
    \frac{|\FF|}{2^{\ex(n, K_{r+1})}}
    & \le \sum_{G \in \GG_2} 2^{e(G)-\left(1-\frac{1}{r}\right)\frac{n^2}{2}+n} \le |\GG| \cdot 2^{n-2n^{2-1/(8r)}} \\
    & \le \exp\left(n^{2-1/(8r)} + 2\log 2 \cdot \left(n- n^{2-1/(8r)}\right)\right) \le 2^{-n},
  \end{align*}
  as claimed.
\end{proof}

\subsection{Balanced and unbalanced $r$-partitions}
\label{sec:balanced-partitions}

Let $\Part$ be an arbitrary $r$-partition of~$\br{n}$. We shall say that $\Part$ is \emph{balanced} if $\min_{P \in \Part} |P| \ge \frac{n}{2r}$ and that it is \emph{unbalanced} otherwise. In the sequel, we denote by $K_\Part$ the complete $r$-partite graph whose colour classes are the parts of $\Part$.

\begin{fact}
  \label{fact:unbalanced-partition}
  Suppose that $r \ge 2$ and let $\Part$ be an unbalanced $r$-partition of $\br{n}$. Then
  \[
    e(K_\Part) \le \ex(n, K_{r+1}) - \frac{n^2}{16r^2}+ n.
  \]
\end{fact}
\begin{proof}
  Let $P \in \Part$ be an arbitrary part satisfying $|P| < \frac{n}{2r}$ and let $Q \in \Part$ be an arbitrary part satisfying $|Q| \ge \frac{n}{r}$. Set
  \[
    m = \left\lfloor \frac{|Q| - |P|}{2} \right\rfloor,
  \]
  let $\Part'$ be the partition obtained from $\Part$ by moving some $m$ vertices from $Q$ to $P$, and observe that
  \[
    e(K_{\Part'}) - e(K_\Part) = (|P| + m)(|Q| - m) - |P||Q| = (|Q| - |P|) \cdot m - m^2 \ge m^2.
  \]
  Since $|Q| - |P| > \frac{n}{2r}$, then $m \ge \left\lfloor\frac{n}{4r}\right\rfloor$ and, consequently, $m^2 \ge \frac{n^2}{16r^2} - n$. Finally, since $K_{\Part'}$ is an $r$-partite graph, we have $e(K_{\Part'}) \le \ex(n, K_{r+1})$ and the claimed upper bound on $e(K_\Part)$ follows.
\end{proof}

Our next lemma bounds from above the number of subgraphs of $K_n$ that admit an unbalanced $r$-partition with few monochromatic edges.

\begin{lemma}
  \label{lemma:number-graphs-close-to-unbalanced}
  The following holds for all sufficiently large $n$ and all $r$ satisfying $2 \le r \le \log n$. Let $\FF$ denote the family of all $G \subseteq K_n$ that satisfy $e(G \setminus K_\Part) \le \frac{n^2}{(r\log n)^2}$ for some unbalanced $r$-partition $\Part$. Then
  \[
    |\FF| \le 2^{\ex(n, K_{r+1}) - n}.
  \]
\end{lemma}
\begin{proof}
  Denote by $\cP_u$ the family of all unbalanced $r$-partitions of $\br{n}$. For every $\Part \in \cP_u$, let $\FF_\Part$ denote the family of all graphs $G \subseteq K_n$ that satisfy $e(G \setminus K_\Part) \le \frac{n^2}{(r \log n)^2}$. We have
  \[
    |\FF_\Part| \le \sum_{t=0}^{n^2/(r \log n)^2}\binom{\binom{n}{2}}{t} \cdot 2^{e(K_\Part)} \le \left(n^2\right)^{\frac{n^2}{(r \log n)^2}} \cdot 2^{e(K_\Part)} \le 2^{e(K_\Part) + \frac{4n^2}{r^2\log n}}.
  \]
  Since $|\cP_u| \le r^n \le 2^{2n\log n}$, Fact~\ref{fact:unbalanced-partition} gives
  \[
    |\FF| \le \sum_{\Pi \in \cP_u} |\FF_\Part| \le 2^{\ex(n,K_{r+1}) - \frac{n^2}{16r^2} + \frac{4n^2}{r^2\log n} + n + 2n\log n} \le 2^{\ex(n, K_{r+1}) - n},
  \]
  provided that $n$ is sufficiently large.
\end{proof}

Let $\Col_r(n)$ denote the family of all $r$-partite subgraphs of $K_n$. Even though some graphs in $\Col_r(n)$ admit many different proper $r$-colourings, our next lemma, which is implicit in the work of Pr\"omel and Steger~\cite{ProSte95}, shows that the average number of proper $r$-colourings of a graph in $\Col_r(n)$ is only slightly larger than one. Our proof of the lemma is an adaptation of the argument underlying the proof of \cite[Proposition~5.5]{BalMorSamWar16}.

\begin{lemma}
  \label{lemma:number-r-col}
  The following holds for all sufficiently large $n$ and all $r$ satisfying $2 \le r \le \log n$. Denoting by $\cP$ the family of all $r$-partitions of $\br{n}$, we have
  \[
    \sum_{\Part \in \cP} 2^{e(K_\Part)} \le \left(1+2^{-\frac{n}{5r^4}}\right) \cdot |\Col_r(n)|.
  \]
\end{lemma}
\begin{proof}
  Denote by $\cP_b$ the family of all balanced $r$-partitions of $\br{n}$ and let $\cP_u = \cP \setminus \cP_b$. Since $|\cP_u| \le r^n \le 2^{2n\log n}$, Fact~\ref{fact:unbalanced-partition} gives
  \[
    \sum_{\Pi \in \cP_u} 2^{e(K_\Part)} \le 2^{\ex(n,K_{r+1}) - \frac{n^2}{16r^2} + n + 2n\log n} \le 2^{\ex(n, K_{r+1}) - n},
  \]
  provided that $n$ is sufficiently large. (This bound is also a consequence of Lemma~\ref{lemma:number-graphs-close-to-unbalanced}.) As $\ex(n,K_{r+1})$ is the (largest) number of edges of an $n$-vertex $r$-colourable graph, we have
  \[
    \sum_{\Pi \in \cP_u} 2^{e(K_\Part)} \le 2^{-n} \cdot |\Col_r(n)|.
  \]

  Since, for every pair $\Part, \Part' \in \cP$, there are exactly $2^{e(K_\Part \cap K_{\Part'})}$ subgraphs of $K_n$ that are properly $r$-coloured by both $\Part$ and $\Part'$, Bonferroni's inequality (the inclusion-exclusion principle) gives
  \[
    |\Col_r(n)| \ge \sum_{\Part \in \cP_b} 2^{e(K_\Part)} - \sum_{\{\Part, \Part'\} \in \binom{\cP_b}{2}} 2^{e(K_\Part \cap K_{\Part'})}.
  \]
  The claimed inequality will thus follow once we establish the following claim.

  \begin{claim}
    For every $\Part \in \cP_b$,
    \[
      \sum_{\Part' \in \cP_b \setminus \{\Part\}} 2^{e(K_\Part \cap K_{\Part'})} \le 2^{e(K_\Part) - \frac{n}{4r^4}}.
    \]
  \end{claim}

  Fix distinct $\Part, \Part' \in \cP_b$. Suppose that $\Part = \{P_1, \dotsc, P_r\}$ and $\Part' = \{P_1', \dotsc, P_r'\}$ and, for all $i, j \in \br{r}$, let $P_{i,j} = P_i \cap P_j'$. We will say that the vertices in $P_{i,j}$ are \emph{moved} from $P_i$ to $P_j'$. For every $i \in \br{r}$, define $L_i$ and $S_i$ as the largest and the second largest subclasses of $P_i$, respectively (with ties broken arbitrarily). Note that $|P_i| \ge \frac{n}{2r}$ implies that $|L_i| \ge \frac{n}{2r^2}$. Set $s = \max_{j \in \br{r}} |S_j|$ and let $S = S_j$ for the smallest $j$ for which the maximum in the definition of $s$ is achieved. Note that $1 \le s \le n/2$, as $s = 0$ would imply that $(P_1', \dotsc, P_r')$ is a permutation of $(P_1, \dotsc, P_r)$, and therefore $\Part = \Part'$.

  By the pigeonhole principle, either some pair $\{L_i, L_j\}$ of largest subclasses or some largest subclass $L_i$ and $S$, where $S \nsubseteq P_i$, are moved to the same vertex class $P_k'$. Since $P_k'$ is an independent set in $K_{\Part'}$, it follows that $K_{\Part} \cap K_{\Part'}$ has no edges between the sets $L_i$ and $L_j$ or $L_i$ and $S$. Since,
  \[
    \min\{ |L_i| \cdot |L_j|, |L_i| \cdot |S| \} \ge \min\left\{ \left(\frac{n}{2r^2}\right)^2, \frac{n}{2r^2} \cdot s \right\} \ge \frac{sn}{2r^4},
  \]
  we have $e(K_\Part \cap K_{\Part'}) \le e(K_\Part) - \frac{sn}{2r^4}$.
  
  Observe that, given a $\Part \in \cP$, we can describe any $\Part' \in \cP \setminus \{\Part\} $ by first picking the (ordered) partitions $(P_{i,j})_{j \in \br{r}}$ for every $i$ and then setting $P_j' = \bigcup_{i \in \br{r}} P_{i,j}$. We claim that, for every $s$, the number of ways to choose all $P_{i,j}$ in such a way that $\max_{i \in \br{r}} |S_i| = s$ is at most $n^{r^2} \cdot n^{sr^2}$. Indeed, one may first specify the sequence $\big(|P_{i,j}|\big)_{i,j \in \br{r}}$ and then specify, for each $i \in \br{r}$, the elements of each $P_{i,j}$ with $j \in \br{r}$, apart from $L_i$ (which will comprise all the remaining, unspecified elements of $P_i$).

  We may thus conclude that
  \[
    \begin{split}
      \sum_{\Part' \in \cP_b \setminus \{\Part\}} 2^{e(K_\Part \cap K_{\Part'})}
      & \le \sum_{s \ge 1} n^{(s+1)r^2} \cdot 2^{e(K_\Part)-\frac{sn}{2r^4}} \\
      & \le 2^{e(K_\Part)} \cdot \sum_{s \ge 1} 2^{4sr^2\log n - \frac{sn}{2r^4}} \le 2^{e(K_\Part) - \frac{n}{4r^4}},
    \end{split}
  \]
  as claimed.
\end{proof}

\subsection{The number of $K_{r+1}$-free graphs with a monochromatic star}
\label{sec:large-star}

 The following lemma will be used to bound from above the number of $K_{r+1}$-free graphs whose optimal $r$-partition induces a monochromatic copy of $K_{1,D}$ in one of the parts.

\begin{lemma}
  \label{lemma:Kr-free-large-degree}
  Let $D$ be an integer satisfying $D \ge 2^rr$. Suppose that $\Part$ is an $r$-partition of $\br{n}$ and that $S$ is a copy of $K_{1, D}$ with $V(K_{1,D}) \subseteq P$ for some $P \in \Part$. If $v \in P$ is the centre vertex of $S$, then
  \[
    \left|\big\{G \subseteq K_\Part : \text{$G \cup S \nsupseteq K_{r+1}$ and $\deg_G(v,Q) \ge D$ for all $Q \in \Part \setminus \{P\}$}\big\}\right| \le 2^{e(K_\Part) - \frac{D^2}{8r^2}}.
  \]
\end{lemma}
\begin{proof}
  Let $G$ be a uniformly chosen random subgraph of $K_\Part$. Expose $G$ on all the edges of $K_\Part$ that have an endpoint in $P$ and condition on $\deg_G(v,Q) \ge D$ for all $Q \in \Part \setminus \{P\}$. It suffices to show that, for every such conditioning, $\Pr(G \cup S \nsupseteq K_{r+1}) \le 2^{-\frac{D^2}{8r^2}}$. For each $Q \in \Part \setminus \{P\}$, choose an arbitrary set of $D$ neighbours of $v$ in $Q$ and let $\KK$ be the family of all $D^r$ copies of $K_r$ in $K_\Part$ whose vertices belong to the chosen $D$-element sets or to $V(K_{1,D}) \setminus \{v\} \subseteq P$. Since, under our conditioning, $v$ is adjacent (in $G \cup S$) to all the vertices of each $K \in \KK$, we have $\Pr(G \cup S \nsupseteq K_{r+1}) \le \Pr(K \nsubseteq G \text{ for all } K \in \KK)$. We may bound the latter probability from above using Janson's inequality. Define, as in the statement of Theorem~\ref{thm:Janson},
  \[
    \mu = \sum_{K \in \KK} \left(\frac{1}{2}\right)^{e(K)}
    \qquad \text{and} \qquad
    \Delta = \sum_{\substack{\{K, K'\} \in \binom{\KK}{2} \\ E(K) \cap E(K') \neq \emptyset}} \left(\frac{1}{2}\right)^{e(K \cup K')}.
  \]
  Observe that the function $\{2, \dotsc, r\} \ni \ell \mapsto D^\ell \cdot 2^{-\binom{\ell}{2}}$ is increasing. Indeed, our assumption implies that, for each $\ell \le r-1$, we have $D \cdot 2^{\binom{\ell}{2} - \binom{\ell+1}{2}} \ge D \cdot 2^{-r} \ge 1$. It follows that
  \[
    \frac{\mu}{2} = \frac{D^r}{2} \cdot \left(\frac{1}{2}\right)^{\binom{r}{2}} \ge \frac{D^2}{4}.
  \]
  We now estimate $\Delta$. To this end, fix some $K \in \KK$ and observe that, for each $k \in \{2, \dotsc, r-1\}$, there are at most $\binom{r}{k} \cdot D^{r-k}$ many $K' \in \KK$ that share exactly $k$ vertices with $K$ and that $e(K \cup K') = 2\binom{r}{2} - \binom{k}{2}$ for each such $K'$. We conclude that
  \[
    2\Delta \le \mu \cdot \sum_{k=2}^{r-1} \binom{r}{k} \cdot D^{r-k} \cdot \left(\frac{1}{2}\right)^{\binom{r}{2}-\binom{k}{2}} \le \mu^2 \cdot \sum_{k=2}^{r-1} \frac{2^{\binom{k}{2}}r^k}{D^k}.
  \]
  Our assumption implies that, for every $k \in \{2, \dotsc, r-2\}$,
  \[
    \frac{2^{\binom{k+1}{2}}r^{k+1}}{D^{k+1}} = \frac{2^kr}{D} \cdot \frac{2^{\binom{k}{2}}r^k}{D^k} \le \frac{1}{2} \cdot \frac{2^{\binom{k}{2}}r^k}{D^k}
  \]
  and hence,
  \[
    \frac{2\Delta}{\mu^2} \le 2 \cdot \frac{2r^2}{D^2}.
  \]
  We conclude that
  \[
    \Pr(G \cup S \nsubseteq K_{r+1}) \le \exp\left(- \min\left\{\frac{\mu}{2}, \frac{\mu^2}{4\Delta}\right\}\right) \le \exp\left(-\frac{D^2}{8r^2}\right),
  \]
  as claimed.
\end{proof}

\subsection{The number of $K_{r+1}$-free graphs with a monochromatic matching}
\label{sec:large-matching}

The following lemma will be used to bound from above the number of $K_{r+1}$-free graphs whose optimal $r$-partition induces a monochromatic matching with a given number of edges in one of the parts.

\begin{lemma}
  \label{lemma:Kr-free-matching}
  Suppose that $\Part$ is a balanced $r$-partition of $\br{n}$ and that $M$ is a matching with $m$ edges such that $V(M) \subseteq P$ for some $P \in \Part$. If $r^2 \cdot 2^{r+3} \le n$, then
  \[
    \left|\big\{G \subseteq K_\Part : G \cup M \nsupseteq K_{r+1} \big\}\right| \le 2^{e(K_\Part) - \frac{mn}{2^{10}r^4}}.
  \]
\end{lemma}
\begin{proof}
  Let $G$ be a uniformly chosen random subgraph of $K_\Part$, so that the assertion of the lemma becomes equivalent to the inequality $\Pr(G \cup M \nsupseteq K_{r+1}) \le 2^{-\frac{mn}{2^{10}r^4}}$. Let $N = \prod_{Q \in \Part \setminus \{P\}} |Q|$ and note that the assumption that $\Part$ is balanced implies that $N \ge \left(\frac{n}{2r}\right)^{r-1}$. Denote by $K_{r+1}^-$ the graph obtained from $K_{r+1}$ by removing from it a single edge and let $\KK$ be the collection of all copies of $K_{r+1}^-$ in $K_\Part$ that form a $K_{r+1}$ with an edge of $M$. Note that $|\KK| = mN$ and that $\Pr(G \cup M \nsupseteq K_{r+1}) = \Pr(K \nsubseteq G \text{ for all } K \in \KK)$. We may thus bound this probability from above using Janson's inequality. Define, as in the statement of Theorem~\ref{thm:Janson},
  \[
    \mu = \sum_{K \in \KK} \left(\frac{1}{2}\right)^{e(K)}
    \qquad \text{and} \qquad
    \Delta = \sum_{\substack{\{K, K'\} \in \binom{\KK}{2} \\ E(K) \cap E(K') \neq \emptyset}} \left(\frac{1}{2}\right)^{e(K \cup K')}.
  \]
  Observe that the function $\{2, \dotsc, r\} \ni \ell \mapsto \left(\frac{n}{2r}\right)^{\ell-1} \cdot 2^{-\binom{\ell+1}{2}}$ is increasing. Indeed, our assumption on $r$ implies that, for each $\ell \le r-1$, we have $\frac{n}{2r} \cdot 2^{\binom{\ell}{2} - \binom{\ell+1}{2}} \ge \frac{n}{2r} \cdot 2^{-r} \ge 1$. It follows that
  \[
    \frac{\mu}{2} = \frac{mN}{2} \cdot \left(\frac{1}{2}\right)^{\binom{r+1}{2}-1} \ge m \cdot \left(\frac{n}{2r}\right)^{r-1} \cdot \left(\frac{1}{2}\right)^{\binom{r+1}{2}} \ge \frac{mn}{16r}.
  \]
  We now estimate $\Delta$. To this end, fix some $K \in \KK$ and observe that:
  \begin{enumerate}[label=(\alph*)]
  \item
    For each $k \in \{1, \dotsc, r-2\}$, there are at most $\binom{r-1}{k} \cdot N \cdot \left(\frac{n}{2r}\right)^{-k}$ many $K' \in \KK$ that share with $K$ the two vertices in $P$ and exactly $k$ other vertices; we have $e(K \cup K') = 2\binom{r+1}{2} - \binom{k+2}{2} -1$ for each such $K'$.
  \item
    For each $k \in \{2, \dotsc, r-1\}$, there are at most $m \cdot \binom{r-1}{k} \cdot N \cdot \left(\frac{n}{2r}\right)^{-k}$ many $K' \in \KK$ that share with $K$ only some $k$ vertices outside of $P$; we have $e(K \cup K') = 2\binom{r+1}{2} - \binom{k}{2} -2$ for each such $K'$.
  \end{enumerate}
  We conclude that
  \[
    \begin{split}
      \frac{2\Delta}{\mu^2}
      & \le \frac{1}{m} \cdot \sum_{k=1}^{r-2} \binom{r-1}{k} \cdot \left(\frac{n}{2r}\right)^{-k} \cdot 2^{\binom{k+2}{2}-1} + \sum_{k=2}^{r-1} \binom{r-1}{k} \cdot \left(\frac{n}{2r}\right)^{-k} \cdot 2^{\binom{k}{2}} \\
      & \le \frac{1}{m} \cdot \sum_{k=1}^{r-2} \frac{(2r^2)^k \cdot 2^{\binom{k+2}{2}-1}}{n^k} + \sum_{k=2}^{r-1} \frac{(2r^2)^k \cdot 2^{\binom{k}{2}}}{n^k} \\
      & \le \frac{8r^2}{mn} + \sum_{k=2}^r \frac{(2r^2)^k \cdot 2^{\binom{k+2}{2}}}{n^k}.
    \end{split}
  \]
  Our assumption on $r$ implies that, for every $k \in \{2, \dotsc, r-1\}$,
  \[
    \frac{(2r^2)^{k+1} 2^{\binom{k+3}{2}}}{n^{k+1}} = \frac{r^2 2^{k+3}}{n} \cdot \frac{(2r^2)^k 2^{\binom{k+2}{2}}}{n^k} \le \frac{1}{2} \cdot \frac{(2r^2)^k 2^{\binom{k+2}{2}}}{n^k}
  \]
  and hence, as $m \le n/2$ and $r \ge 2$,
  \[
    \frac{2\Delta}{\mu^2} \le\frac{8r^2}{mn} + 2 \cdot \frac{4r^4 \cdot 2^6}{n^2} \le \frac{258r^4}{mn}.
  \]
  We conclude that
  \[
    \Pr(G \cup M \nsubseteq K_{r+1}) \le \exp\left(- \min\left\{\frac{\mu}{2}, \frac{\mu^2}{4\Delta}\right\}\right) \le \exp\left(-\frac{mn}{2^{10}r^4}\right),
  \]
  as claimed.
\end{proof}

\subsection{Proof of Theorem~\ref{thm:clique-free}}
\label{sec:proof-theorem-clique}

Suppose that positive integers $n$ and $r$ satisfy $2 \le r \le \log n / (121 \log \log n)$ and let $\Col_r(n)$ and $\FF$ denote the families of all $r$-partite and all $K_{r+1}$-free subgraphs of $K_n$, respectively. Since $\Col_r(n) \subseteq \FF$, it suffices to show that
\begin{equation}
  \label{eq:FF-minus-Col}
  |\FF \setminus \Col_r(n)| \le 2^{-\frac{n}{(10r)^4}} \cdot |\Col_r(n)|.
\end{equation}

Let $\delta = (8 \log n)^{-15}$ and let $\cP$ be the family of all $r$-partitions of $\br{n}$. Define, for every graph $G \in \FF$,
\[
  t(G) = \min\{e(G \setminus K_\Part) : \Part \in \cP\}
\]
and let
\[
  \Fclose = \{G \in \FF : 1 \le t(G) \le \delta n^2\}
  \qquad \text{and} \qquad
  \Ffar = \{G \in \FF : t(G) > \delta n^2\},
\]
so that $\Fclose \cup \Ffar = \FF \setminus \Col_r(n)$. Furthermore, for every $G \in \Fclose$, let $\Part(G)$ be an arbitrary $r$-partition that achieves the minimum in the definition of $t(G)$. Let $\Fclose^b$ comprise these $G$ in $\Fclose$ for which $\Part(G)$ is a balanced partition and let $\Fclose^u = \Fclose \setminus \Fclose^b$. Finally, for every balanced partition $\Part \in \cP$ and every integer $t$ satisfying $1 \le t \le \delta n^2$, define
\[
  \FF_{t,\Part} = \{G \in \Fclose^b : t(G) = t \text{ and } \Part(G) = \Part\}.
\]
Letting $\cP_b$ denote the set of balanced $r$-partitions of $\br{n}$, we thus have
\begin{equation}
  \label{eq:FF-minus-Col-upper}
  |\FF \setminus \Col_r(n)| \le |\Ffar| + |\Fclose^u| + \sum_{\Part \in \cP_b} \sum_{t=1}^{\delta n^2} |\FF_{t,\Part}|.
\end{equation}
It follows from Theorem~\ref{thm:almost-all-almost-partite} and Lemma~\ref{lemma:number-graphs-close-to-unbalanced} that the first and the second terms in the right-hand side of~\eqref{eq:FF-minus-Col-upper} are at most $2^{\ex(n,K_{r+1})-n}$ each. To bound the final term, we shall derive the following estimate.

\begin{claim}
  \label{claim:FtPart-upper}
  For every integer $t$ satisfying $1 \le t \le \delta n^2$ and each $\Part \in \cP_b$,
  \[
    |\FF_{t,\Part}| \le 2^{e(K_\Part)-\frac{n}{(8r)^4}-t}.
  \]
\end{claim}

Let us first argue that inequality~\eqref{eq:FF-minus-Col-upper} and Claim~\ref{claim:FtPart-upper} imply~\eqref{eq:FF-minus-Col}. Indeed, assuming Claim~\ref{claim:FtPart-upper}, we have
\[
  \begin{split}
    |\FF \setminus \Col_r(n)|
    & \le 2 \cdot 2^{\ex(n,K_{r+1})-n} + \sum_{\Part \in \cP_b} \sum_{t=1}^{\delta n^2} 2^{e(K_\Part)-\frac{n}{(8r)^4}-t} \\
    & \le 2^{\ex(n,K_{r+1})-n+1} + 2^{-\frac{n}{(8r)^r}} \cdot \sum_{\Part \in \cP} 2^{e(K_\Part)}.
  \end{split}
\]
Finally, since $|\Col_r(n)| \ge 2^{\ex(n,K_{r+1})}$ and $\sum_{\Part \in \cP} 2^{e(K_\Part)} \le 2|\Col_r(n)|$, by Lemma~\ref{lemma:number-r-col}, we conclude that
\[
  |\FF \setminus \Col_r(n)| \le \left(2^{-n+1} + 2^{-\frac{n}{(8r)^4} + 1} \right) \cdot |\Col_r(n)|,
\]
which yields~\eqref{eq:FF-minus-Col}. It thus suffices to prove Claim~\ref{claim:FtPart-upper}.

\begin{proof}[{Proof of Claim~\ref{claim:FtPart-upper}}]
  Let $D = \left\lfloor \frac{n}{2^{14}r^5\log n} \right\rfloor$ and define
  \begin{align*}
    \FF_{t,\Part}^S & = \big\{G \in \FF_{t,\Part} : G[P] \supseteq K_{1,D} \text{ for some } P \in \Part\big\}, \\
    \FF_{t,\Part}^M & = \big\{G \in \FF_{t,\Part} : G[P] \text{ has a matching of size $\lceil t/(Dr) \rceil$ for some $P \in \Part$}\big\}.
  \end{align*}
  Since every graph with $t$ edges contains either a vertex with degree at least $D$ or a matching with at least $t/D$ edges, we have $\FF_{t,\Part} = \FF_{t,\Part}^S \cup \FF_{t,\Part}^M$ and we may bound $|\FF_{t,\Part}|$ from above in two steps.

  First, we claim that if $G \in \FF_{t,\Part}^S$ and $v \in P \in \Part$ is the centre vertex of a copy of $K_{1,D}$ in $G[P]$, then $\deg_G(v,Q) \ge D$ for all $Q \in \Part$. Indeed, if this were not true, then moving $v$ from $P$ to $Q$ would yield a partition $\Part'$ such that
  \[
    e(G \setminus K_{\Part'}) = e(G \setminus K_{\Part}) + \deg_G(v,Q) - \deg_G(v,P) < e(G \setminus K_{\Part}),
  \]
  which would contradict our assumption that $\Part = \Part(G)$. It thus follows from Lemma~\ref{lemma:Kr-free-large-degree} (which we may apply as $D \ge n^{1/2} \ge 2^rr$ when $n$ is sufficiently large) that
  \[
    |\FF_{t,\Part}^S| \le \left(n^2\right)^t \cdot 2^{e(K_\Part) - \frac{D^2}{8r^2}} \le 2^{e(K_\Part) - \frac{D^2}{8r^2} + 4t\log n} \le 2^{e(K_\Part) - n - t},
  \]
  where the last inequality holds because, by our choice of $\delta$ and $D$,
  \[
    \frac{D^2}{8r^2} \ge \frac{n^2}{2^{32}r^{12}(\log n)^2} \ge \frac{n^2}{(8\log n)^{14}}\ge 8 \delta n^2 \log n \ge 4t\log n + n + t.
  \]
  Second, it follows from Lemma~\ref{lemma:Kr-free-matching} that
  \[
    |\FF_{t,\Part}^M| \le \left(n^2\right)^t \cdot 2^{e(K_\Part) - \frac{\lceil t/(Dr) \rceil n}{2^{10}r^4}} \le 2^{e(K_\Part) - \frac{\lceil t/(Dr) \rceil n}{2^{10}r^4} + 4t\log n} \le 2^{e(K_\Part) - \frac{n}{2^{11}r^4} - t},
  \]
  where the last inequality holds because
  \[
    \frac{\lceil t/(Dr) \rceil n}{2^{10}r^4} - 4t\log n - t \ge \min_{\tau \in \{1, 2, \dotsc\}} \left\{\frac{\tau n}{2^{10}r^4} - \tau Dr \cdot (4\log n+1)\right\} \ge \frac{n}{2^{11}r^4},
  \]
  as the definition of $D$ assures that $n \ge 2^{14}Dr^5\log n$. Since $|\FF_{t,\Part}| \le |\FF_{t,\Part}^S| + |\FF_{t,\Part}^M|$, combining the two bounds above gives the assertion of the claim.
\end{proof}

\section{Lower bounds for $\eps$-nets}
\label{sec:lower-bounds-eps-nets}

\subsection{Outline}
\label{sec:outline-eps-nets}

Our (randomised) construction of planar point sets $X$ without a small $\eps$-net for the range space of lines on $X$ is a slight simplification of the construction of Balogh and Solymosi~\cite{BalSol18}. The high-level idea of both constructions, which can be traced back to the work of Alon~\cite{Alo12}, may be summarised as follows. We find an integer $s$, a finite set $X \subseteq \RR^2$, and a sub-collection $\cL$ of all lines in $\RR^2$ with the following property: Let $\HH$ be the $s$-uniform hypergraph with vertex set $X$ whose edges are all intersections of the lines in $\cL$ with $X$ that have exactly $s$ points. The independence number of $\HH$ is at most $(1-c)|X|$, for some constant $c > 0$.

Given such $s$, $X$, and $\cL$, we set $\eps = s/|X|$ and observe that the complement of every $\eps$-net $N$ for the range space of lines on $X$ is an independent set of $\HH$. Indeed, every such $N$ intersects every line that contains at least $s$ points of $X$; in particular, $N$ must intersect every line in $\cL$ that contains exactly $s$ points of $X$. This means that $|N| \ge c|X| = cs/\eps$, which improves upon the trivial bound $|N| \ge \Omega(1/\eps)$ if $s$ can be made arbitrarily large. The challenge is to make $s$ as large as possible, as a function of $|X|$.

In the construction of Alon~\cite{Alo12}, the set $X$ is a generic projection of the $d$-dimensional grid $\br{s}^d$ to $\RR^2$ and $\cL$ is the image of all combinatorial lines in $\br{s}^d$ via this projection; the key property is guaranteed, for large enough $d$, by the density version of the Hales--Jewett theorem proved by Furstenberg and Katznelson~\cite{FurKat91}. In Balogh and Solymosi's~\cite{BalSol18} construction, $X$ was a generic projection of a random subset of a larger, high-dimensional integer grid, trimmed appropriately (so that each line in $\cL$ contains no more than $s$ points of $X$), and the key property of $\HH$ was established, for a careful choice of $\cL$, with the use of the hypergraph container theorem of Saxton and Thomason~\cite{SaxTho15}. Here, we take $X$ to be a random subset of $\br{n}^2$, trimmed appropriately (as in~\cite{BalSol18}), and establish the key property of $\HH$, for a careful choice of $\cL$, using our efficient container lemma, Theorem~\ref{thm:main-packaged}.

\subsection{Proof of Theorem~\ref{thm:eps-nets}}
\label{sec:proof-eps-nets}

Let $s$ be a positive integer, let $m = 10s$, and let $M$ be a prime number satisfying $m^{m^2-1} \le M \le 2m^{m^2-1}$. Set $n = m M$, so that
\begin{equation}
  \label{eq:n-m-relation}
  n \ge m^{m^2} \qquad \text{and} \qquad m \ge \sqrt{\frac{\log n}{\log \log n}}.
\end{equation}
We shall find an $\eps \in (0, 1/n)$ and a set $X \subseteq \RR^2$ without a small $\eps$-net among the subsets of the integer grid~$\br{n}^2$, which we shall from now on denote by $P$. We will be able to prove the claimed lower bound on the smallest size of an $\eps$-net of $X$ for the range space of all lines in $\RR^2$ by considering only a fairly small family $\cL$ of lines that we now specify.

Given an integer $h \in \br{M-1}$ and a point $(x_0, y_0) \in \br{n} \times \br{M}$, we let $\ell(x_0, y_0; h)$ be the line passing through $(x_0, y_0)$ whose slope is $M / h$, that is,
\[
  \ell(x_0, y_0; h) = \big\{(x_0, y_0) + t \cdot (h, M) : t \in \RR \big\}.
\]
Since $M$ is prime, and thus co-prime with $h$, the vector $t \cdot (h,M)$ has integer coordinates if and only if $t \in \ZZ$. Moreover, if $t$ is an integer, then $y_0 + tM \in \br{n}$ if and only if $t \in \{0, \dotsc, m-1\}$. In particular, $\ell(x_0, y_0; h)$ intersects $P$ in at most $m$ points; it intersects $P$ in exactly $m$ points if and only if $x_0 + (m-1)h \le n$. Now, for every $h \in \br{M-1}$, let
\[
  \cL_h = \big\{\ell(x_0, y_0; h) : (x_0, y_0) \in \br{n} \times \br{M} \text{ and } x_0 + (m-1)h \le n\big\},
\]
so that every line in $\cL_h$ intersects $P$ in exactly $m$ points. Since the lines in $\cL_h$ are pairwise disjoint (as they are parallel), we have
\[
  \left| \left(\bigcup \cL_h\right) \cap P \right| = m \cdot |\cL_h| = m \cdot \big(n - (m-1)h\big) \cdot M = n^2 \cdot \left(1 - \frac{(m-1)h}{n}\right).
\]
Let $\hmax = \lfloor n/(10m) \rfloor$ so that $\bigcup \cL_h$ has at least $9n^2/10$ points of $P$ for every $h \in \br{\hmax}$. Finally, define
\[
  \cL = \bigcup_{h=1}^{\hmax} \cL_h
\]
and note that
\begin{equation}
  \label{eq:cL-size}
  \frac{n^3}{12m^2} \le \frac{9n^2}{10m} \cdot \hmax \le |\cL \cap P| \le \frac{n^2}{m} \cdot \hmax \le \frac{n^3}{10m^2}.
\end{equation}

We shall say that a set $A \subseteq P$ is \emph{$\cL$-collinear} if $A$ is contained in some line in $\cL$. As every line in $\cL$ contains exactly $m$ points of $P$, the number of $a$-element $\cL$-collinear subsets of $P$ is precisely $|\cL| \cdot \binom{m}{a}$ for every $a \in \{2, \dotsc, m\}$.

Suppose that $p$ satisfies
\begin{equation}
  \label{eq:eps-net-p}
  K \cdot m^{10} \cdot n^{-1/(s-1)} \cdot \log n \le p \le m^{-1} \cdot n^{-1/s}
\end{equation}
for some large absolute constant $K$; such a number does indeed exist as
\[
  n^{1/(s-1)-1/s} \ge n^{1/s^2} \ge m^{m^2/s^2} = m^{100} \ge K \cdot m^{10} \cdot \log n,
\]
provided that $m$ is sufficiently large. Let $R$ be a $p$-random subset of $P$ and let $X \subseteq R$ be a largest subset of $R$ that contains no $\cL$-collinear subset of $s+1$ points. By maximality of $X$, every point of $R \setminus X$ forms an $\cL$-collinear $(s+1)$-element set with some $s$ points of~$X$. In particular, $|R \setminus X|$ is at most the number of $\cL$-collinear $(s+1)$-element subsets of~$R$. It follows that
\[
  \Ex\big[|R \setminus X|\big] \le |\cL| \cdot \binom{m}{s+1} \cdot p^{s+1} \le \frac{n^3}{10m^2} \cdot m^{s+1} \cdot p^{s+1} \le \frac{n^2p}{10m},
\]
by the second inequality in~\eqref{eq:eps-net-p}, and consequently, by Markov's inequality,
\[
  \Pr\big(|R \setminus X| \ge n^2p/10\big) \le \frac{\Ex\big[|R \setminus X|\big]}{n^2p/10} \le \frac{1}{m} \le \frac{1}{10}.
\]
On the other hand, standard estimates for lower tails of binomial distributions (Lemma~\ref{lemma:binomial-tails}) yield 
\[
  \Pr\big(|R| \le 9n^2p/10\big) \le \exp\left(-n^2p/200\right) \le \exp(-n/200).
\]
It follows that
\[
  \Pr\big(|X| \ge 4n^2p/5 \big) \ge 4/5,
\]
provided that $n$ is sufficiently large.

\begin{claim}
  \label{claim:eps-net-containers}
  With probability at least $1/2$, every set $I \subseteq R$ with $|I| \ge 3n^2p/5$ contains an $\cL$-collinear subset of $s$ points.
\end{claim}

Together with the above calculations, Claim~\ref{claim:eps-net-containers} implies that there exists a set $X \subseteq P$ of at least $4n^2p/5$ points that has the following two properties:
\begin{enumerate}[label=(\alph*)]
\item
  $X$ has no $\cL$-collinear subset with $s+1$ elements;
\item
  every set of $3n^2p/5$ elements of $X$ contains an $\cL$-collinear $s$-element subset.
\end{enumerate}
Suppose that $X$ is such a set, let $\eps = s/|X|$, and assume that $N \subseteq X$ is an $\eps$-net for the range space of lines. In particular, $N$ intersects every $\cL$-collinear subset of $X$ that has at least $s = \eps |X|$ elements. Since $X$ contains no $\cL$-collinear set with more than $s$ points, $X \setminus N$ contains no $\cL$-collinear subset of $s$ points, and thus 
\[
  |N| = |X| - |X \setminus N| > |X| - 3n^2p/5 \ge |X|/4 = s/(4\eps).
\]
Finally, since $1/\eps = |X|/s \le |P| = n^2$, we have, using~\eqref{eq:n-m-relation},
\[
  \frac{s}{4} = \frac{m}{40} \ge \frac{1}{40} \cdot \sqrt{\frac{\log n}{\log \log n}} \ge \frac{1}{80} \cdot \sqrt{\frac{\log(1/\eps)}{\log \log(1/\eps)}}.
\]
This gives the assertion of the theorem.

We now prove Claim~\ref{claim:eps-net-containers}. Let $\HH$ be the $s$-uniform hypergraph with vertex set $P$ whose edges are all $\cL$-collinear $s$-element subsets of $P$. The assertion of the claim is that, with probability at least $1/2$, the random set $R$ contains no independent set of $\HH$ that has at least $3n^2p/5$ elements. This is a simple consequence of the following lemma, which lies at the heart of the matter.

\begin{lemma}
  \label{lemma:containers-for-eps-nets}
  There is a family $\cC$ of at most $\exp(pn^2/300)$ containers for the independent sets of $\HH$ such that $|C| \le n^2/2$ for every $C \in \cC$.
\end{lemma}

We first show how Lemma~\ref{lemma:containers-for-eps-nets} implies the assertion of Claim~\ref{claim:eps-net-containers}. Let $\cC$ be a family of containers for the independent sets of $\HH$ supplied by the lemma and let $\cB$ be the event that $R$ contains an independent set of $\HH$ with at least $3n^2p/5$ elements. Since every independent set of $\HH$ is contained in some member of $\cC$, each of which has at most $n^2/2$ elements, we have
\[
  \Pr(\cB) \le \sum_{C \in \cC} \Pr\left(|R \cap C| \ge 3n^2p/5\right) \le |\cC| \cdot \Pr\left(\Bin(n^2/2, p) \ge 3n^2p/5\right).
\]
Standard estimates for upper tails of binomial distributions (Lemma~\ref{lemma:binomial-tails}) yield
\[
  \Pr\left(\Bin(n^2/2, p) \ge 3n^2p/5\right) \le \exp\left(-n^2p/150\right)
\]
and, consequently,
\[
  \Pr(\cB) \le |\cC| \cdot \exp\left(-n^2p/150\right) \le \exp\left(-n^2p/300\right)  \le \exp(-n/300) \le 1/2,
\]
provided that $n$ is sufficiently large.

Finally, we prove Lemma~\ref{lemma:containers-for-eps-nets} by combining our `packaged' hypergraph container lemma, Theorem~\ref{thm:main-packaged}, with the following supersaturation statement for the hypergraph $\HH$ of $\cL$-collinear $s$-tuples.

\begin{lemma}
  \label{lemma:eps-nets-supersaturation}
  If $Q \subseteq P$ has at least $n^2/3$ points, then $e\big(\HH[Q]\big) \ge |\cL|$.
\end{lemma}
\begin{proof}
  Define $b_s \colon \RR \to \RR$ by
  \[
    b_s(x) =
    \begin{cases}
      \binom{x}{s} & \text{if $x \ge s-1$,}\\
      0 & \text{if $x \le s-1$,}
    \end{cases}
  \]
  so that $b_s$ is convex and $b_s(a) = \binom{a}{s}$ whenever $a$ is a nonnegative integer. Jensen's inequality gives
  \[
    e\big(\HH[Q]\big) = \sum_{\ell \in \cL} \binom{|\ell \cap Q|}{s} = \sum_{\ell \in \cL} b_s\big(|\ell \cap Q|\big) \ge |\cL| \cdot b_s\left(\frac{1}{|\cL|} \cdot \sum_{\ell \in \cL} |\ell \cap Q|\right).
  \]
  Recall that, for every $h \in \br{\hmax}$, the lines in $\cL_h$ cover all but at most $n^2/10$ points of $P$. In particular, for every such $h$,
  \[
    \sum_{\ell \in \cL_h} |\ell \cap Q| \ge |Q| - \frac{n^2}{10} \ge \frac{n^2}{5}
  \]
  and thus
  \[
    \frac{1}{|\cL|} \cdot \sum_{\ell \in \cL} |\ell \cap Q| = \frac{1}{|\cL|} \cdot \sum_{h = 1}^{\hmax} \sum_{\ell \in \cL_h} |\ell \cap Q| \ge \frac{\hmax}{|\cL|} \cdot \frac{n^2}{5} \ge \frac{m}{5}.
  \]
  Consequently,
  \[
    e\big(\HH[Q]\big) \ge |\cL| \cdot b_s\big(m/5\big) = |\cL| \cdot \binom{m/5}{s} \ge |\cL|,
  \]
  as claimed.
\end{proof}

\begin{proof}[{Proof of Lemma~\ref{lemma:containers-for-eps-nets}}]
  Set
  \[
    q = \frac{p}{300 m^5\log n} \qquad \text{and} \qquad E = |\cL|.
  \]
  We now verify that we may apply Theorem~\ref{thm:main-packaged}, with $\alpha \leftarrow 1/2$ and $\beta \leftarrow 1/3$, to the hypergraph $\HH$. First, we have
  \[
    \alpha \beta q \cdot v(\HH) = \frac{pn^2}{1800m^5\log n} \ge 10^9 s^7 \qquad \text{and} \qquad 10^4s^5q = \frac{10^4s^5p}{300m^5\log n} \le \frac{1}{\log n} \le \beta,
  \]
  provided that $n$ is sufficiently large. Second, for every $t \in \{2, \dotsc, s\}$, by~\eqref{eq:cL-size} and~\eqref{eq:eps-net-p},
  \begin{multline*}
    \left(\frac{q}{10^6s^5}\right)^{t-1} \cdot \frac{E}{v(\HH)} = \left(\frac{p}{3 \cdot 10^8s^5m^5\log n}\right)^{t-1} \cdot \frac{|\cL|}{n^2} \ge \left(\frac{K \cdot m^5 \cdot n^{-1/(s-1)}}{3 \cdot 10^8 s^5}\right)^{s-1} \cdot \frac{n}{12m^2} \\
    = \left(\frac{K}{3 \cdot 10^3}\right)^{s-1} \cdot \frac{1}{12m^2} \ge 2^m \ge \binom{m-t}{s-t} = \Delta_t(\HH),
  \end{multline*}
  provided that $K$ is sufficiently large. The theorem supplies a collection $\cC$ of containers for the independent sets of $\HH$ such that
  \[
    \begin{split}
      |\cC| & \le \exp\left(10^4s^5\beta^{-1}\log(e/\alpha) \cdot q\log(e/q) \cdot v(\HH)\right) \\
      & \le \exp\left(m^5 \cdot q \log n \cdot n^2\right) \le \exp\left(pn^2/300\right),
    \end{split}
  \]
  where we used the inequality $q \ge e/n$, which holds if $K$ is sufficiently large, and, for every $C \in \cC$, either $|C| \le \alpha \cdot v(\HH) = n^2/2$ or there is a subset $W \subseteq C$ with $|W| \ge (1-\beta)|C| = 2|C|/3$ such that $e(\HH[W]) < E = |\cL|$. We claim that, in fact, $|C| \le n^2/2$ for every $C \in \cC$. Indeed, if $|C| > n^2/2$ and $W \subseteq C$ satisfies $|W| \ge 2|C|/3 > n^2/3$, then $e(\HH[W]) \ge |\cL|$, by Lemma~\ref{lemma:eps-nets-supersaturation}.  
\end{proof}

\section{Upper bounds on Ramsey numbers}
\label{sec:upper-bounds-Ramsey}

In this section, we derive the upper bounds on Folkman numbers and on induced Ramsey numbers stated in Theorems~\ref{thm:Folkman} and~\ref{thm:induced-Ramsey}. We shall do this by building containers for non-Ramsey colourings of subgraphs of a large complete graph $K_N$ and examining how a random subgraph of $K_N$, drawn with an appropriately chosen distribution for each of the two theorems, intersects these containers. This approach to studying Ramsey properties of random graphs was introduced in the work of Nenadov and Steger~\cite{NenSte16}. The only Ramsey-theoretic ingredient in our proof is the following supersaturated version of Ramsey's theorem, which is a refinement of~\cite[Corollary~2.2]{NenSte16}.

\begin{lemma}
  \label{lemma:supersaturation-Ramsey}
  Suppose that $n$ and $k$ are positive integers and let $R = R(n;k)$. If $N \ge R$, then every colouring $c \colon E(K_N) \to \br{k+1}$ either assigns the colour $k+1$ to at least $(1/2) \cdot (N/R)^2$ edges or it contains at least $(1/2) \cdot (N/R)^n$ monochromatic copies of $K_n$ in colours $1, \dotsc, k$.
\end{lemma}
\begin{proof}
  The choice of $R$ guarantees that the edge-colouring induced by every subset of $R$ vertices of $K_N$ contains either an edge coloured $k+1$ or a monochromatic copy of $K_n$ in one of the remaining $k$ colours. On the other hand, each edge and each copy of $K_n$ are contained in, respectively, $\binom{N-2}{R-2}$ and $\binom{N-n}{R-n}$ such subsets. Denoting by $M$ the total number of monochromatic copies of $K_n$ in colours $1, \dotsc, k$, we thus have
  \begin{equation}
    \label{eq:supersaturation-Ramsey}
    \binom{N}{R} \le \left|c^{-1}(k+1)\right| \cdot \binom{N-2}{R-2} + M \cdot \binom{N-n}{R-n}.
  \end{equation}
  In particular, since, for every $\ell \in \{2, n\}$,
  \[
    \binom{N}{R} \cdot \binom{N-\ell}{R-\ell}^{-1} = \binom{N}{\ell} \cdot \binom{R}{\ell}^{-1} \ge \left(\frac{N}{R}\right)^\ell
  \]
  inequality~\eqref{eq:supersaturation-Ramsey} implies that either $|c^{-1}(k+1)| \ge (1/2) \cdot (N/R)^2$ or $M \ge (1/2) \cdot (N/R)^n$.
\end{proof}

\subsection{Folkman numbers (proof of Theorem~\ref{thm:Folkman})}
\label{sec:Folkman-proof}

Let $k$ and $n$ be positive integers, let $R = R(n;k)$ and suppose that an integer $N$ satisfies
\begin{equation}
  \label{eq:N-Folkman}
  N \ge (\Gamma knR)^{21n^2}
\end{equation}
for some large constant $\Gamma$. We shall give a randomised construction of a $K_{n+1}$-free subgraph of $K_N$ that satisfies $G \rightarrow (K_n)_k$, proving that $F(n;k) \le N$. We shall from now on assume that $k \ge 2$ and  $n \ge 3$, as otherwise the assertion of the theorem is trivial.

Suppose that $G \subseteq K_N$. We shall identify a $k$-colouring $c \colon E(G) \to \br{k}$ of the edges of~$G$ with the set
\[
  \big\{(e, c_e) : e \in E(G)\big\} \subseteq E(K_N) \times \br{k}.
\]
Let $\HH$ be the hypergraph with vertex set $E(K_N) \times \br{k}$ whose edges are all sets of the form
\[
  \varphi\big(E(K_n)\big) \times \{i\},
\]
where $\varphi \colon V(K_n) \to V(K_N)$ is an arbitrary injection and $i \in \br{k}$. If a graph $G \subseteq K_N$ admits a colouring $c \colon E(G) \to \br{k}$ with no monochromatic copy of $K_n$, then $c$, when viewed as a subset of $E(K_N) \times \br{k}$, is an independent set of $\HH$.

We shall say that a graph $G \subseteq E(K_N)$ is \emph{compatible} with a set $C \subseteq E(K_N) \times \br{k}$ if there exists a colouring $c \colon E(G) \to \br{k}$ that is contained in $C$. Equivalently, $G$ is compatible with $C$ if and only if $\big(\{e\} \times \br{k}\big) \cap C \neq \emptyset$ for every $e \in E(G)$. In other words, defining
\[
  X(C) = \big\{ e \in E(K_N) : \big(e \times \br{k}\big) \cap C = \emptyset\big\},
\]
$G$ is compatible with $C$ if and only if $X(C) \cap E(G) = \emptyset$.

Suppose that $p$ satisfies
\begin{equation}
  \label{eq:Folkman-p}
  D \cdot (knR)^{20}\cdot N^{-2/(n+1)} \log N \le p \le \frac{N^{-2/(n+2)}}{DR}
\end{equation}
for some large constant $D$; such a number does indeed exist as~\eqref{eq:N-Folkman} implies that
\[
  \frac{N^{2/(n+1)-2/(n+2)}}{\log N} \ge N^{1/n^2} \ge D^2 \cdot (knR)^{21},
\]
provided that $\Gamma$ is sufficiently large. The following lemma is key.

\begin{lemma}
  \label{lemma:containers-for-Folkman}
  There is a family $\cC$ of at most $\exp\left(\frac{pN^2}{256R^2}\right)$ containers for the independent sets of $\HH$ such that $|X(C)| \ge \left(\frac{N}{4R}\right)^2$ for every $C \in \cC$.
\end{lemma}

We first show how Lemma~\ref{lemma:containers-for-Folkman} implies the assertion of the theorem. To this end, suppose that $G \sim G_{N,p}$ and denote by $Z$ the number of copies of $K_{n+1}$ in $G$. The upper bound in~\eqref{eq:Folkman-p} implies that
\[
  \Ex[Z] \le p^{\binom{n+1}{2}} N^{n+1} = pN^2 \cdot \left(p^{(n+2)/2} N\right)^{n-1} \le \frac{pN^2}{(DR)^{(n+2)(n-1)/2}} \le \frac{pN^2}{64R^2},
\]
provided that $D$ is sufficiently large. Let $G'$ be the subgraph obtained from $G$ by deleting an arbitrary edge from every copy of $K_{n+1}$; observe that $K_{n+1} \nsubseteq G'$ and $e(G') \ge e(G) - Z$.

Suppose that $G' \not\rightarrow (K_n)_k$. This means that there is a colouring $c \colon E(G') \to \br{k}$ that is an independent set of $\HH$. Therefore, $G'$ must be compatible with some container from~$\cC$. In other words, there is some $C \in \cC$ such that $X(C) \cap E(G') = \emptyset$ and, consequently,
\[
  |X(C) \cap E(G)| \le e(G) - e(G') \le Z.
\]
We may conclude that
\begin{equation}
  \label{eq:Folkman-union-bound}
  \Pr\big(G' \not\rightarrow (K_n)_k\big) \le \Pr\big(Z > 2\Ex[Z]\big) + \sum_{C \in \cC} \Pr\big(|X(C) \cap E(G)| \le 2\Ex[Z] \big).
\end{equation}

Fix an arbitrary $C \in \cC$. Since $|X(C)| \ge \frac{N^2}{16R^2}$, standard estimates on the lower tails of binomial distributions (Lemma~\ref{lemma:binomial-tails}) yield
\[
  \Pr\big(|X(C) \cap E(G)| \le 2\Ex[Z] \big) \le \Pr\left(\Bin\left(\frac{N^2}{16R^2}, p\right) \le \frac{pN^2}{32R^2}\right) \le \exp\left(-\frac{pN^2}{128R^2}\right).
\]
Substituting this estimate and the inequality $\Pr(Z > 2\Ex[Z]) < 1/2$ into~\eqref{eq:Folkman-union-bound} yields
\[
  \Pr\big(G' \not\rightarrow (K_n)_k\big) \le \frac{1}{2} + \exp\left(-\frac{pN^2}{256R^2}\right) \le \frac{1}{2} + e^{-N} \le \frac{3}{4}.
\]
In particular, there is a graph $G' \subseteq K_N$ such that $G' \nsupseteq K_{n+1}$ and $G' \rightarrow (K_n)_k$, as claimed.

\begin{proof}[{Proof of Lemma~\ref{lemma:containers-for-Folkman}}]
  Set
  \[
    s = \binom{n}{2}, \qquad q = \frac{p}{(10knR)^{10} \log N}, \qquad \beta = \frac{1}{16kR^2},\qquad \text{and} \qquad E = \left(\frac{N}{2R}\right)^{n}.
  \]
  We now verify that we may apply Theorem~\ref{thm:main-packaged}, with $\alpha \leftarrow \frac{1}{2k}$, to the hypergraph $\HH$. First, as $s \le n^2$, we have, by~\eqref{eq:N-Folkman} and~\eqref{eq:Folkman-p},
  \[
    \alpha \beta q \cdot v(\HH) \ge \frac{D \cdot (knR)^{10} \cdot N^{-2/(n+1)}}{10^{10}} \cdot \binom{N}{2} \ge N \ge (\Gamma n)^{21} \ge 10^9 s^7
  \]
  and
  \[
    10^4s^5q \le 10^4s^5 \cdot N^{-2/(n+2)} \le \frac{10^4n^{10}}{(knR)^{21n}} \le \frac{1}{16kR^2} = \beta,
  \]
  provided that $n$ is sufficiently large. Second, suppose that $t \in \{2, \dotsc, s\}$ and let $\ell \in \{3, \dotsc, n\}$ be the unique integer satisfying $\binom{\ell-1}{2} < t \le \binom{\ell}{2}$, so that
  \[
    t-1 \le \binom{\ell}{2}-1 = \frac{(\ell+1)(\ell-2)}{2} \le \frac{(n+1)(\ell-2)}{2}.
  \]
  Since a graph with $t$ edges must have at least $\ell$ vertices, we have
  \[
    \Delta_t(\HH) = k \cdot \binom{N - \ell}{n-\ell} \le kN^{n-\ell}
  \]
  and, consequently,
  \begin{align*}
    \left(\frac{q}{10^6s^5}\right)^{t-1} \cdot \frac{E}{v(\HH)} \cdot \frac{1}{\Delta_t(\HH)}
    & \ge \left(\frac{q}{10^6s^5}\right)^{\frac{(n+1)(\ell-2)}{2}} \cdot \frac{\big(N/(2R)\big)^n}{kN^2} \cdot \frac{1}{kN^{n-\ell}} \\
    & \ge \left(\frac{p}{10^{16}n^{20}(kR)^{10}\log N}\right)^{\frac{(n+1)(\ell-2)}{2}} \cdot \frac{N^{\ell-2}}{k^2(2R)^n} \\
    & \ge \left(\frac{pN^{2/(n+1)}}{(10knR)^{20}\log N}\right)^{\frac{(n+1)(\ell-2)}{2}} \ge 1,
  \end{align*}
  where the last inequality follows from~\eqref{eq:Folkman-p}. The theorem supplies a collection $\cC$ of containers for the independent sets of $\HH$ satisfying
  \begin{equation}
    \label{eq:eps-nets-cC-size}
    |\cC| \le \exp\left(10^4s^5\beta^{-1}\log(e/\alpha) \cdot q\log(e/q) \cdot v(\HH)\right)
  \end{equation}
  such that, for every $C \in \cC$, either $|C| \le \alpha \cdot v(\HH)$ or there is a subset $W \subseteq C$ with $|W| \ge (1-\beta)|C| \ge |C| - \beta k\binom{N}{2}$ and $e(\HH[W]) < E$. We now turn to bounding the right-hand side of~\eqref{eq:eps-nets-cC-size} from above. To this end, observe first that $v(\HH) = k \binom{N}{2} \le kN^2$ and that $s^5 \beta^{-1} \log(e/\alpha) \le n^{10} R^2 k\log (2ek)$. Further, it follows from~\eqref{eq:Folkman-p} that, if $D$ is sufficiently large,
  \[
    \frac{e}{q} = \frac{e(10knR)^{10}\log N}{p} \le N^{2/(n+1)} \le N
  \]
  and thus $q\log(e/q) \le p/(10knR)^{10}$. Substituting these three estimates into~\eqref{eq:eps-nets-cC-size} yields
  \[
    |\cC| \le \exp\left(10^{4} \cdot n^{10}R^2k\log(2ek) \cdot \frac{p}{(10knR)^{10}} \cdot kN^2\right) \le \exp\left(\frac{pN^2}{256R^2}\right),
  \]
  as desired.
  
  It remains to show that  $|X(C)| \ge \left(\frac{N}{4R}\right)^2$ for every $C \in \cC$. Suppose that this were not true. If $|C| \le \alpha v(\HH) = \frac{k}{2k} \binom{N}{2} = \frac{1}{2} \binom{N}{2}$, then
  \[
    |X(C)| \ge \binom{N}{2} - |C| \ge \frac{1}{2}\binom{N}{2} \ge \left(\frac{N}{4R}\right)^2,
  \]
  so we may assume that $|C| > \alpha v(\HH)$. In particular, there must be a subset $W \subseteq C$ with $|C \setminus W| \le \beta k\binom{N}{2}$ such that $e(\HH[W]) < E$. Observe that
  \begin{equation}
    \label{eq:XW-upper-Folkman}
    |X(W)| \le |X(C)| + |C \setminus W| \le \left(\frac{N}{4R}\right)^2 + \beta k\binom{N}{2} \le \frac{N^2}{8R^2}
  \end{equation}
  and let $c \colon E(K_N) \to \br{k+1}$ be an arbitrary colouring such that $(e, c_e) \in W$ for every $e \notin X(W)$ and $c_e = k+1$ otherwise. By Lemma~\ref{lemma:supersaturation-Ramsey}, either $|X(W)| = |c^{-1}(k+1)| \ge (1/2) \cdot (N/R)^2$ or the colouring $c$ has at least $(1/2) \cdot (N/R)^n$ monochromatic copies of $K_n$ in colours $1, \dotsc, k$. However, the former inequality contradicts~\eqref{eq:XW-upper-Folkman} and thus the latter must hold. Finally, note that, if $K$ is an arbitrary copy of $K_n$ in $K_N$ that $c$ colours with some $i \in \br{k}$, then $E(K) \times \{i\} \subseteq W$. This implies that $e(\HH[W]) \ge (1/2) \cdot (N/R)^n \ge E$, contradicting our assumption.
\end{proof}

\subsection{Induced Ramsey numbers (proof of Theorem~\ref{thm:induced-Ramsey})}
\label{sec:upper-bounds-induced-Ramsey}

Let $k$ be a positive integer, let $H$ be an arbitrary $n$-vertex graph, let $R = R(n;k)$, and suppose that an integer $N$ satisfies
\begin{equation}
  \label{eq:N-induced-Ramsey}
  N \ge (10n^2kR)^{7n}.
\end{equation}
We shall prove that, with probability very close to one, the uniformly chosen random subgraph of $K_N$ is induced-Ramsey for $H$ in $k$ colours, proving that $\Rind(H;k) \le N$. We shall from now on assume that $k \ge 2$ and $n \ge 3$, as otherwise the assertion of the theorem is trivial.

Suppose that $G \subseteq K_N$. We shall identify a $k$-colouring $c \colon E(G) \to \br{k}$ of the edges of~$G$ with the set
\[
  \big\{(e, c_e) : e \in E(G)\big\} \cup \big\{(e, 0) : e \in E(K_N) \setminus E(G) \big\} \subseteq E(K_N) \times \{0, \dotsc, k\}.
\]
(That is, we extend $c$ to a colouring of $E(K_N)$ by colouring all edges of $K_N \setminus G$ zero.) Let $\HH$ be the hypergraph with vertex set $E(K_N) \times \{0, \dotsc, k\}$ whose edges are all sets of the form
\[
  \left(\varphi\big(E(H)\big) \times \{i\}\right) \cup \left(\varphi\big(E(K_n) \setminus E(H)\big) \times \{0\}\right),
\]
where $\varphi \colon V(H) \to V(K_N)$ is an arbitrary injection and $i \in \br{k}$. If a graph $G \subseteq K_N$ admits a colouring $c \colon E(G) \to \br{k}$ such that $c^{-1}(i)$ does not contain a copy of $H$ that is induced in $G$ for any $i \in \br{k}$, then $c$, when viewed as a subset of $E(K_N) \times \{0, \dotsc, k\}$, is an independent set of $\HH$.

We shall say that a graph $G \subseteq E(K_N)$ is \emph{compatible} with a set $C \subseteq E(K_N) \times \{0, \dotsc, k\}$ if there exists a colouring $c \colon E(G) \to \br{k}$ that is contained in $C$. Equivalently, $G$ is compatible with $C$ if and only if $(e,0) \in C$ for every $e \in E(K_N) \setminus E(G)$ and $\big(\{e\} \times \br{k}\big) \cap C \neq \emptyset$ for every $e \in E(G)$. In other words, defining
\begin{align*}
  X_1(C) & = \big\{ e \in E(K_N) : (e, 0) \notin C \big\} \quad \text{and} \\
  X_2(C) & = \big\{ e \in E(K_N) : \big(e \times \br{k}\big) \cap C = \emptyset\big\},
\end{align*}
$G$ is compatible with $C$ if and only if $X_1(C) \subseteq E(G)$ and $X_2(C) \cap E(G) = \emptyset$. The following lemma is key.

\begin{lemma}
  \label{lemma:containers-for-induced-Ramsey}
  There is a family $\cC$ of at most $\exp\left(\frac{N^2}{100R^2}\right)$ containers for the independent sets of $\HH$ such that $|X_1(C) \cup X_2(C)| \ge \left(\frac{N}{4R}\right)^2$ for every $C \in \cC$.
\end{lemma}

Assuming that Lemma~\ref{lemma:containers-for-induced-Ramsey} is true, suppose that $G \sim G_{N,1/2}$. If $G \not\indarrow (H)_k$, then $G$ must be compatible with some container from $\cC$. It follows that
\begin{align*}
  \Pr\big(G \not\indarrow (H)_k\big)
  & \le \sum_{C \in \cC} \Pr\big(X_1(C) \subseteq E(G) \text{ and } X_2(C) \cap E(G) = \emptyset\big) \\
  & \le \sum_{C \in \cC} 2^{-|X_1(C) \cup X_2(C)|} \le |\cC| \cdot 2^{-(N/(4R))^2} \\
  & \le \exp\left(\frac{N^2}{100R^2} - \frac{N^2\log 2}{16R^2}\right) \le \exp\left(-N\right).
\end{align*}
\begin{proof}[{Proof of Lemma~\ref{lemma:containers-for-induced-Ramsey}}]
  Set
  \[
    s = \binom{n}{2}, \qquad q = \frac{1}{(kR)^5}, \qquad \beta = \frac{1}{16kR^2},\qquad \text{and} \qquad E = \left(\frac{N}{2R}\right)^{n}.
  \]
  We now verify that we may apply Theorem~\ref{thm:main-packaged}, with $\alpha \leftarrow \frac{1}{3k}$, to the hypergraph $\HH$. First, as $s \le n^2$, we have
  \[
    \alpha \beta q \cdot v(\HH) = \frac{1}{48(kR)^7} \cdot (k+1)\binom{N}{2} \ge N \ge 2^{n/2} \ge 10^9 s^7
  \]
  and
  \[
    10^4s^5q = \frac{10^4s^5}{(kR)^5} \le \frac{10^4n^{10}}{(kR)^5} \le \frac{10^4n^{10}}{(kR)^2 \cdot 2^n} \le \frac{1}{16kR^2} = \beta,
  \]
  provided that $n$ is sufficiently large. Second, suppose that $t \in \{2, \dotsc, s\}$ and let $\ell \in \{3, \dotsc, n\}$ be the unique integer satisfying $\binom{\ell-1}{2} < t \le \binom{\ell}{2}$, so that
  \[
    t-1 \le \binom{\ell}{2}-1 = \frac{(\ell+1)(\ell-2)}{2} \le n(\ell-2).
  \]
  Since a graph with $t$ edges must have at least $\ell$ vertices, we have
  \[
    \Delta_t(\HH) = k \cdot \binom{N - \ell}{n-\ell} \le kN^{n-\ell}
  \]
  and, consequently,
  \begin{align*}
    \left(\frac{q}{10^6s^5}\right)^{t-1} \cdot \frac{E}{v(\HH)} \cdot \frac{1}{\Delta_t(\HH)}
    & \ge \left(\frac{q}{10^6s^5}\right)^{n(\ell-2)} \cdot \frac{\big(N/(2R)\big)^n}{kN^2} \cdot \frac{1}{kN^{n-\ell}} \\
    & = \left(\frac{1}{10^6(skR)^5}\right)^{n(\ell-2)} \cdot \frac{N^{\ell-2}}{k^2(2R)^n} \\
    & \ge \left(\frac{N^{1/n}}{10^7(skR)^6}\right)^{n(\ell-2)} \ge 1,
  \end{align*}
  where the last inequality follows from~\eqref{eq:N-induced-Ramsey}. The theorem supplies a collection $\cC$ of containers for the independent sets of $\HH$ satisfying
  \[
    \begin{split}
      |\cC| & \le \exp\left(10^4s^5\beta^{-1}\log(e/\alpha) \cdot q\log(e/q) \cdot v(\HH)\right) \\
      & \le \exp\left(10^7n^{10} R^2k\log k \cdot \frac{5\log\left(ekR\right)}{(kR)^5} \cdot (k+1)\binom{N}{2}\right) \le \exp\left(\frac{N^2}{100R^2}\right),
    \end{split}
  \]
  provided that $n$ is sufficiently large, such that, for every $C \in \cC$, either $|C| \le \alpha \cdot v(\HH)$ or there is a subset $W \subseteq C$ with $|W| \ge (1-\beta)|C| \ge |C| - \beta(k+1)\binom{N}{2}$ and $e(\HH[W]) < E$.
  
  It remains to show that  $|X_1(C) \cup X_2(C)| \ge \left(\frac{N}{4R}\right)^2$ for every $C \in \cC$. Suppose that this were not true and set $X(C) = X_1(C) \cup X_2(C)$. If $|C| \le \alpha v(\HH) = \frac{k+1}{3k} \binom{N}{2} \le \frac{1}{2} \binom{N}{2}$, then
  \[
    |X(C)| \ge \binom{N}{2} - |C| \ge \frac{1}{2}\binom{N}{2} \ge \left(\frac{N}{4R}\right)^2,
  \]
  so we may assume that $|C| > \alpha v(\HH)$. In particular, there must be a $W \subseteq C$ with $|C \setminus W| \le \beta(k+1)\binom{N}{2}$ such that $e(\HH[W]) < E$. Observe that
  \begin{equation}
    \label{eq:XW-upper-induced-Ramsey}
    |X(W)| \le |X(C)| + |C \setminus W| \le \left(\frac{N}{4R}\right)^2 + \beta(k+1)\binom{N}{2} \le \frac{N^2}{8R^2}
  \end{equation}
  and let $c \colon E(K_N) \to \br{k+1}$ be an arbitrary colouring such that $(e, c_e) \in W$ for every $e \notin X(W)$ and $c_e = k+1$ otherwise. By Lemma~\ref{lemma:supersaturation-Ramsey}, either $|X(W)| = |c^{-1}(k+1)| \ge (1/2) \cdot (N/R)^2$ or the colouring $c$ has at least $(1/2) \cdot (N/R)^n$ monochromatic copies of $K_n$ in colours $1, \dotsc, k$. However, the former inequality contradicts~\eqref{eq:XW-upper-induced-Ramsey} and thus the latter must hold. Let $K$ be an arbitrary copy of $K_n$ in $K_N$ that $c$ colours with some $i \in \br{k}$. Since $E(K) \cap X_1(W) = \emptyset$, then $E(K) \times \{0, i\} \subseteq W$ and, as a result, any injection $\varphi \colon V(H) \to V(K)$ corresponds to an edge of $\HH[W]$. This implies that $e(\HH[W]) \ge (1/2) \cdot (N/R)^n \ge E$, contradicting our assumption.
\end{proof}


\bibliographystyle{amsplain}

\providecommand{\bysame}{\leavevmode\hbox to3em{\hrulefill}\thinspace}
\providecommand{\MR}{\relax\ifhmode\unskip\space\fi MR }
\providecommand{\MRhref}[2]{%
  \href{http://www.ams.org/mathscinet-getitem?mr=#1}{#2}
}
\providecommand{\href}[2]{#2}

\begin{dajauthors}
\begin{authorinfo}[jozsi]
  J\'ozsef Balogh\\
  Department of Mathematics\\
  University of Illinois at Urbana-Champaign\\
  Urbana, Illinois 61801, USA\\
  \&\\
  Moscow Institute of Physics and Technology\\
  Russian Federation\\
  jobal\imageat{}illinois\imagedot{}edu
\end{authorinfo}
\begin{authorinfo}[wojtek]
  Wojciech Samotij\\
  School of Mathematical Sciences\\
  Tel Aviv Univeristy\\
  Tel Aviv 6997801, Israel\\
  samotij\imageat{}tauex\imagedot{}tau\imagedot{}ac\imagedot{}il
\end{authorinfo}
\end{dajauthors}

\end{document}